\documentclass[10pt]{amsart}
\usepackage[all]{xy}
\usepackage{amsmath}
\usepackage{amsfonts}
\usepackage{amssymb}
\usepackage{amscd}
\usepackage{amsthm}
\usepackage{latexsym}
\usepackage{amsbsy}
\usepackage{graphicx}
\usepackage{epstopdf}
\usepackage{seqsplit}
\usepackage{color}
\usepackage{float}
\usepackage{calligra}
\usepackage[T1]{fontenc}
\AppendGraphicsExtensions{.gif}
\newtheorem{lemma}{Lemma}[section]
\newtheorem{proposition}{Proposition}[section]
\newtheorem{theorem}{Theorem}[section]
\newtheorem{corollary}{Corollary}[section]
\newtheorem{definition}{Definition}[section]

\usepackage[T1]{fontenc}
\usepackage[latin1]{inputenc} 
\usepackage{bbm}

\newcommand{\Tr}{\textrm{Tr}}
\newcommand{\ncttwo}{\mathbb{T}_\theta^2}
\newcommand{\ancttwo}{C(\mathbb{T}_\theta^2)}
\newcommand{\sncttwo}{C^\infty(\mathbb{T}_\theta^2)}
\newcommand{\nctfour}{\mathbb{T}_\Theta^4}

\newcommand{\snctfour}{C^\infty(\mathbb{T}_\Theta^4)}
\newcommand{\nctm}{\mathbb{T}_\Theta^m}
\newcommand{\anctm}{C(\mathbb{T}_\Theta^m)}
\newcommand{\snctm}{C^\infty(\mathbb{T}_\Theta^m)}
\newcommand{\Aa}{\mathcal{A}}

\newcommand{\C}{\mathbbm{C}}

\newcommand{\LieG}{\mathfrak{g}}

\newcommand{\R}{\mathbbm{R}}

\newcommand{\Z}{\mathbbm{Z}}
\newcommand{\ddt}{\frac{d}{dt}\bigm|_{t= 0}}
\newcommand{\ddsone}{\partial_1}
\newcommand{\ddstwo}{\partial_2}
\newcommand{\ddsthree}{\partial_3}
\newcommand{\NT}{\mathbb{T}_\theta^2}
\newcommand{\NTp}{\mathbb{T}_{\theta'}^2}
\newcommand{\CNTp}{C^\infty(\mathbb{T}_{\theta'}^2)}
\newcommand{\NTpp}{\mathbb{T}_{\theta''}^2}
\newcommand{\CNTpp}{C^\infty(\mathbb{T}_{\theta''}^2)}
\newcommand{\CNT}{C^\infty(\mathbb{T}_\theta^2)}

\newcommand{\nab}{\nabla}

\newcommand{\dep}{\frac{d}{d\varepsilon}\bigm|_{\varepsilon = 0}}
\newcommand{\vep}{\varepsilon}
\newcommand{\vphi}{\varphi}
\newcommand{\nabep}{\nabla_{\varepsilon}}
\newcommand{\TR}{{\rm TR}}
\newcommand{\tr}{{\rm tr}}
\newcommand{\ricden}{{\bf Ric}}
\newcommand{\ricfun}{\mathcal{R}ic}
\newcommand{\ric}{{\rm Ric}}
\newcommand{\riem}{{\rm Riem}}
\newcommand{\riemop}{ Riem}
\newcommand{\scalar}{R}

\newcommand{\res}{\rm res}
\newcommand{\Res}{\rm Res}
\newcommand{\pr}{{ Q}}
\newcommand{\A}{A_\theta}
\newcommand{\Ai}{A_\theta^\infty}

\newcommand{\End}{{\rm End}}
\newcommand{\lap}{\triangle}
\newcommand{\conn}{\bigtriangledown }
\newcommand{\nctorus}[1][2]{\mathbb{T}_\theta^{#1}}
\def\Int{\int\hspace{-0.35cm}{-} \,}
\def \ord{{\rm ord}}
\newcommand{\compcent}[1]{\vcenter{\hbox{$#1\circ$}}}
\newcommand{\comp}{\mathbin{\mathchoice
{\compcent\scriptstyle}{\compcent\scriptstyle}
{\compcent\scriptscriptstyle}{\compcent\scriptscriptstyle}}} 
\makeatletter
\newcommand{\extp}{\@ifnextchar^\@extp{\@extp^{\,}}}
\def\@extp^#1{\mathop{\bigwedge\nolimits^{\!#1}}}
\renewcommand{\vec}[1]{\boldsymbol{#1}}
\renewcommand{\dim}{d}

\title{Curvature in Noncommutative Geometry}

\author{Farzad Fathizadeh and Masoud Khalkhali} 
 
\begin{document}  
 
\maketitle 
 
\vspace{0.1cm}

{\centering  \it Dedicated  to Alain Connes with admiration, affection, and much appreciation\par}

\date{}

\begin{abstract}
Our understanding of the notion of curvature in a noncommutative 
setting has progressed substantially in the past 10 years. This new episode 
in  noncommutative geometry started when a Gauss-Bonnet theorem was 
proved by  Connes and Tretkoff for a curved noncommutative two torus. 
Ideas from spectral  geometry and heat kernel asymptotic expansions suggest 
a general way of defining  local curvature invariants for noncommutative Riemannian 
type spaces where the  metric structure is encoded by a Dirac type operator. 
To carry explicit computations however one needs quite intriguing new ideas. 
We give an account of the most recent  developments on the notion of curvature 
in noncommutative geometry in this paper.
 \end{abstract}


\medskip
\medskip
\noindent

\medskip
\noindent

\tableofcontents{}
\allowdisplaybreaks


\section{{\bf Introduction}}

Broadly speaking,    the progress of  {\it noncommutative geometry} in the last  four  decades   can be  divided into  three phases: {\it topological, spectral,} and {\it arithmetical.} One can also notice the pervasive influence of quantum  physics in all aspects of the subject. 
Needless to say each of these facets of the subject  is  still evolving, and there are    many deep connections among them.
 
 In its topological phase,  noncommutative geometry  was largely informed  by index theory and a real  need to extend index theorems beyond their classical realm of   smooth manifolds,   to what  we collectively call {\it noncommutative spaces}.   Thus $K$-theory,  $K$-homology, and $KK$-theory in general,  were brought in  and with the discovery of cyclic cohomology by Connes \cite{ConSepcHom, MR777584},   a suitable framework  was created by him to formulate noncommutative index theorems.  With the appearance of the ground breaking and now classical paper of Connes \cite{MR823176}, results of which were already announced in Oberwolfach in 1981 \cite{ConSepcHom},  this phase of the theory was essentially completed. In particular a noncommutative Chern-Weil theory of characteristic classes was created with  Chern character maps for both    $K$-theory and $K$-homology with values in cyclic (co)homology.   To define all these a notion of Fredholm module (bounded or unbouded, finitely summable or theta summable) was introduced which essentially captures and brings in  many aspects of smooth manifolds  into the  noncommutative world. These results were applied to noncommutative quotient spaces such  as the space of leaves of  a foliation, or the unitary dual of noncompact and nonabelian Lie  groups. Ideas and tools   from global analysis, differential topology, operator
algebras, representation theory, and quantum statistical mechanics, were crucial.   One of the  main applications  of this resulting noncommutative index theory was 
to settle some long standing conjectures such as the Novikov conjecture, and the
Baum-Connes conjecture for specific and large classes of groups.

Next came the study of the geometry of  noncommutative spaces and the impact of {\it spectral geometry}.  Geometry,  as we understand it here, historically has dealt   with the study of spaces of increasing  complexity   and metric measurements  within such spaces. Thus in classical differential geometry one learns how to measure distances and volumes, as well as various types of curvature   of  Riemannian manifolds of arbitrary  
dimension. One can say the two notions of Riemannian metric  and the Riemann curvature tensor are hallmarks of classical differential geometry in general. 
This should be contrasted with topology where one studies spaces only from a rather soft  homotopy theoretic  point of view.   A similar
division is at work  in noncommutative geometry. Thus,  as we mentioned briefly above,  while in its earlier stage of development  noncommutative geometry  was mostly concerned with the  development of topological invariants like cyclic cohomology, Connes-Chern character maps,  and index theory,  starting in about  ten  years ago noncommutative geometry  entered a new truly geometric phase where one tries to seriously understand what a {\it curved noncommutative space} is and how to define and compute  curvature invariants for such a noncommutative space.  

This  episode in noncommutative geometry started when a  Gauss-Bonnet theorem was proved  by Connes and Cohen for a {\it curved noncommutative torus} in \cite{MR2907006} (see 
also the MPI preprint \cite{CohCon} where many ideas are already laid out). This paper was immediately followed   in \cite{MR2956317} where the Gauss-Bonnet was proved for general conformal structures. 
The metric structure of a  noncommutative space is encoded in a (twisted) spectral triple. 
Giving a state of the art report on  developments  following these works,  and  on the notion of curvature  in noncommutative geometry,  is the purpose of our present review.

\vskip 0.7cm
 
\begin{center}
\includegraphics[scale=0.2]{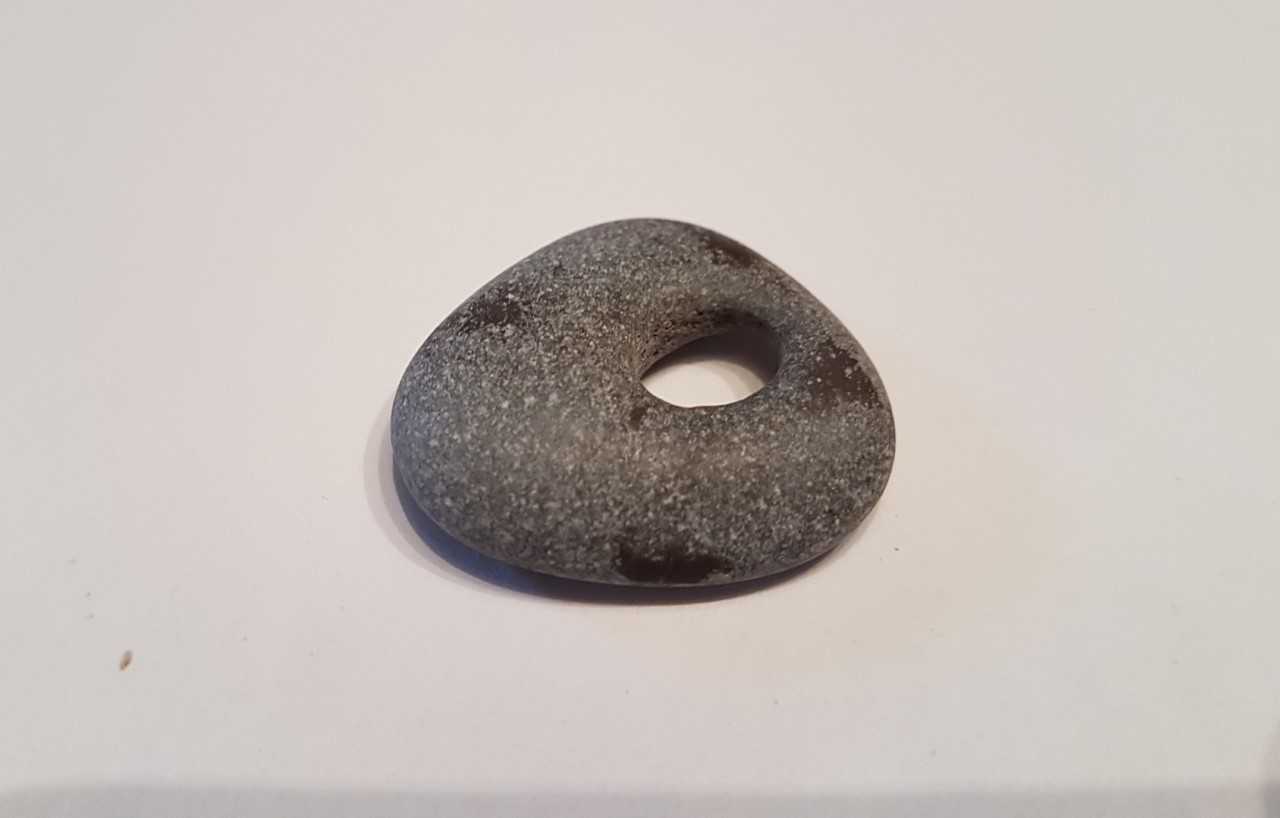}\\
\smallskip 
\smallskip
{\textcolor{red}{\LARGE{\calligra{This is not a quantum curved torus}}}}
\end{center}

\vskip 0.5 cm

Classically,  geometric invariants are usually defined explicitly and algebraically  in a local coordinate system,    in terms of a metric tensor or a connection  on the given manifold. However, methods based on local coordinates, or algebraic methods  based on commutative algebra,  have no chance of being useful  in  a noncommutative setting, in general.    But  other methods, more analytic and more subtle, based on ideas of spectral  geometry are available. In fact, thanks to  spectral geometry, we know that  there are intricate relations between  Riemannian invariants and spectra of naturally defined elliptic operators like Laplace or Dirac operators on the given manifold. A prototypical example is the celebrated {\it Weyl's law} on the asymptotic distribution of eigenvalues of the Laplacian of a closed Riemannian manifold $ M^n$ in terms of its volume:
\begin{equation} N (\lambda) \sim \frac{ \omega_n \text{Vol}(M)}{ (2\pi)^n} \lambda^{\frac{n}{2}} \quad \quad \lambda \to \infty.  \label{Weyl1}
\end{equation}
Here  $N(\lambda)$ is the number of eigenvalues of the Laplacian  in the interval $[0, \lambda]$ and $\omega_n$ is the volume of the unit ball in $\mathbb{R}^n$. 
 In the spirit of Marc Kac's article  \cite{MR0201237},
  one  says  
 one can hear the volume of a manifold. But one can ask what else about a Riemannian manifold can be heard? Or even  we can ask: what  can  we learn  by listening to a noncommutative manifold? Results so far indicate that one can effectively define and compute, not only the volume, but in fact  the scalar and   Ricci curvatures of noncommutative curved spaces, at least in many example. 

 In his Gibbs lecture of 1948, {\it Ramifications, old  and new, of the eigenvalue problem}, Hermann Weyl had this to say about possible extensions of his asymptotic law \eqref{Weyl1}: {\it I feel that these informations about the proper oscillations of a
membrane, valuable as they are, are still very incomplete. I have certain
conjectures on what a complete analysis of their asymptotic behavior
should aim at; but since for more than 35 years I have made
no serious attempt to prove them, I think I had better keep them to
myself.}

 One of the  most elaborate results  in spectral geometry,  is Gilkey's theorem that gives the first four nonzero terms in the  asymptotic expansion of the heat kernel of Laplace type  operators    in terms of covariant derivatives of the metric tensor and the Riemann curvature tensor \cite{MR1396308}. 
More precisely, if 
 $P$ is a Laplace type operator, then the heat operator  $e^{-tP}$ is a smoothing operator with a smooth kernel $k(t, x, y),$ and there is 
  an asymptotic expansion near $t=0$ for the heat kernel  restricted to  the diagonal of $M\times M$: 
$$k(t, x, x) \sim \frac{1}{(4 \pi t)^{m/2}}  (a_0 (x, P) + a_2(x, P)t + a_4 (x, P)t^2 + \cdots), $$
where  $a_i (x, P) $ are known as the  Gilkey-Seeley-DeWitt coefficients. The first term $ a_0 (x, P)$   is a constant. It was first calculated by Minakshisundaram and Pleijel
  \cite{MR0031145} for $P=\Delta$ the Laplace operator.  Using Karamata's Tauberian  theorem, one immediately obtains  Weyl's law for closed Riemannian manifolds. Note that Weyl's original proof was for bounded domains with a regular  boundary in Euclidean space and does not extend to manifolds in general. The next term $a_2 (x, P)$, for $P= \Delta$,  was calculated  by MacKean and Singer \cite{MR0217739} and it was shown that it gives the scalar curvature:
$$a_2 (x, \triangle) = \frac{1}{6}S(x).$$
This immediately shows that the scalar curvature has a spectral nature and in particular the total scalar curvature is a  spectral invariant. This result, or rather its localized version to be recalled later, is at the heart of the noncommutative geometry approach to the definition of scalar curvature. The expressions for  $a_{2k} (x, P)$  get rapidly complicated as $k$ grows, although in principle they can be recursively computed in a normal coordinate chart. They are reproduced up to term $a_6$  in the next section.

It is this  analytic point of view on geometric invariants that play an important role in   understanding  the  geometry of curved noncommutative spaces.  The  algebraic approach almost completely breaks down in the noncommutative case. Our experience so far in the past  few  years has been that
in the noncommutative case spectral and hard analytic methods based on pseudodifferential operators yield results that are in no way possible to guess or arrive at from their commutative counterparts by algebraic methods. One just needs  to take a look at our formulas for scalar, and now Ricci  curvature,  in dimensions two, three and four, in later sections   to believe in this statement. The fact that in the first step we had to rely on heavy symbolic computer calculations to start the analysis,   shows the formidable nature of this material. Surely computations, both symbolic and analytic, are quite hard and are done on a case by case basis, but the surprising end results totally justify the effort.

The spectral geometry of a {\it curved  noncommutative two torus} has been the subject of intensive studies in recent years. As we said earlier, 
this whole episode started when a  Gauss-Bonnet theorem was proved  by Connes and Tretkoff (formerly Cohen) in \cite{MR2907006} (see 
also \cite{CohCon} for an earlier version), and for general conformal structures in \cite{MR2956317}.  A natural question then was to define and compute  the scalar curvature of a curved noncommutative torus. This was done, independently,  by  Connes-Moscovici   \cite{MR3194491}, and Fathizadeh-Khalkhali  \cite{MR3148618}.  The next term in the expansion, namely the term $a_4$, which in the classical case contains explicit 
 information about the analogue of the Riemann tensor, is calculated and studied in \cite{arXiv1611.09815}.   
 A  version of the Riemann-Roch theorem is proven in 
\cite{MR3282309} and the study of local spectral invariants is extended to all finite projective modules on noncommutative two tori  in \cite{MR3540454}.
 
 A  key idea  to  define a curved noncommutative space in the above works is to conformally perturb a flat spectral triple
 by  introducing a noncommutative Weyl factor. 
The complex geometry of the noncommutative two torus, on the other hand, provides a Dirac operator which, in analogy with the classical case, originates from the Dolbeault complex. 
By perturbing this spectral triple, one can construct a  (twisted) spectral triple that can be used to study the geometry of the conformally perturbed flat metric on the noncommutative two torus.  
Then, using the pseudodifferrential operator theory for $C^\ast$-dynamical systems developed by Connes in \cite{MR572645}, the computation is performed and explicit formulas are obtained. The spectral geometry and study of scalar curvature  of  noncommutative tori has been pursued further in  
\cite{MR3402793, MR3359018,  MR3369894}.

 Finally,   for the latest  on interactions between noncommutative geometry, number theory,  and arithmetic algebraic geometry, the reader  can start with the  article by Connes and Consani  \cite{arXiv1805.10501} in this volume and references therein.

\section{{\bf Curvature in noncommutative geometry}}
This section is  of an introductory nature and  is meant to set the stage for later sections and to motivate the evolution of the concept of curvature 
in noncommutative geometry from its beginnings  to   its present form. Clearly we have no intention of giving even a brief sketch of the history of the development of the curvature concept in differential geometry. That would require a separate long article, if not a book. We shall simply highlight some key concepts that have impacted the development of the idea of curvature in noncommutative geometry.

\subsection{A  brief history of curvature}
Curvature,  as understood in classical differential geometry, is one of the most important 
features of a geometric space. It is here that geometry and topology differ in the ways they probe  a space.
To talk about curvature we need more than just topology or smooth structure on a space. The extra piece of structure is usually encoded in a (pseudo-)Riemannian metric, or at least a connection on the tangent bundle, or on a principal $G$-bundle. 
It is remarkable that  Greek geometers missed the curvature concept altogether, even for simple curves like a circle,  which  they studied so intensely. The earliest quantitative understanding of curvature, at least for circles, is due to Nicole Oresme  in fourteenth century. In his treatise, {\it De configurationibus}, he correctly 
gives the inverse of raidus as the curvature of  a circle. 
The concept had to wait for Descartes' analytic geometry and the Newton-Leibniz calculus  before  to be developed and  fully understood.   In fact the first definitions of the (signed) curvature $ \kappa$  of a plane curve $y= y(x)$  are  due to Newton, Leibniz and Huygens  in 17th century:  
$$ \kappa = \frac{y''}{(1+ y'^2)^{3/2}}.$$
  It is important to note that this is not an intrinsic concept. Intrinsically any one dimensional Riemannian manifold is locally 
isometric to $\mathbb{R}$ with its flat Euclidean metric and hence  its intrinsic curvature is zero. 

Thus the first  major case to be understood was the curvature of a surface embedded in   a  three dimensional Euclidean space with its induced metric. In his magnificient paper of 1828 entitled  {\it  disquisitiones generales circa superficies curvas,} Gauss first defines the curvature of a surface in an {\it extrinsic} way, using the Gauss map  and then he proves  his {\it theorema egregium}: the curvature so defined is in fact an {\it intrisic} concept and can solely be defined in terms of the first fundamental form.  That is the Gaussian curvature is an isometry invariant, or in Gauss' own words:
\begin{center}
\textcolor{blue}{{\it Thus the formula of the preceding article leads itself to the remarkable Theorem. If a curved surface is developed upon any other surface whatever, the measure of curvature in each point remains unchanged.}}
\end{center}
Now the first fundamental form   is just the induced Riemannian metric in more modern language. As we shall see, in the hands of Riemann, Theorema Egregium  opened the way for the idea of intrinsic geometry of spaces in general. Surfaces, and manifolds in general,  have  an intrinsic geometry defined solely by metric relations within the space itself,  independent of any ambient space. 

If $g =e^h (dx^2+ dy^2)$ is a locally conformally flat  metric, then its Gaussian curvature is given by
$$ K= -\frac{1}{2}e^{-h}{\Delta h},$$
where $\Delta$ is the flat Laplacian. We shall see later in this paper that the analogous formula in the noncommutative case, first obtained in  \cite{MR3194491, MR3148618},  takes a much more complicated form, with remarkable similarities and differences.

Another major result of Gauss' surface theory was his {\it local uniformization theorem}, which amounts to existence of {\it isothermal coordinates}: any analytic  Riemannian metric in two dimensions is locally conformally flat.  The result holds for all smooth metrics in two dimensions, but Gauss' proof only covers  analytic metrics. Since conformal classes of metrics on a two torus are parametrized by the upper half plane modulo the action of the modular group, this justifies the initial choice of metrics for noncommutative tori by Connes and Cohen  in their Gauss-Bonnet theorem in  \cite{MR2907006}, and for general conformal structures in our paper \cite{MR2956317}.  By all chances, in the noncommutatve case  one needs to go beyond the class of localy conformally flat metrics. For recent results in this direction see \cite{arXiv1811.04004}. 

  A third major achievement  of Gauss in differential geometry  is  his   local {\it Gauss-Bonnet theorem}: for any geodesic triangle drawn on a surface  with interior angles $\alpha, \beta, \gamma,$ we have 
$$\alpha +\beta +\gamma  -\pi = \int K dA,$$
where $K$ denotes the Gauss curvature and $dA$ is the surface area element. By using a geodesic triangulation 
of the surface, one can then easily prove  the global Gauss-Bonnet theorem for a closed Riemannian surface: 
$$\frac{1}{2\pi}\int_M KdA =  \chi (M),$$
where $\chi (M)$ is the Euler characteristic of the closed surface $M$. It is hard to overemphasize the importance of this result which connects geometry with topology.  It  is  the first example of an index theorem and the theory of characteristic classes. 

To find a true  analogue of the Gauss-Bonnet theorem  in a noncommutative setting  was the  motivation for Connes and Tretkoff in their ground breaking  work \cite{MR2907006}.  After conformally perturbing the flat metric of a noncommutative torus, they noticed that while the above classical  formulation   has no clear analogue in the noncommutative case, its spectral formulation 
\[ \zeta(0) + 1 = \frac{1}{12 \pi} \int_{M} K dA = \frac{1}{6}
\chi (M), \nonumber \]
makes perfect sense. Here 
\begin{equation}
\zeta (s) = \sum \lambda_{j}^{-s}, \,\,\, \quad \textnormal{Re}(s) > 1,
\end{equation}
is the {\it spectral zeta function} of the scalar Laplacian  $\triangle_{g} = d^*d$ of  $(M, g)$.  
The spectral zeta function has a meromorphic continuation to $\mathbb{C}$ with a
 unique (simple) pole at $s=1$. In particular $\zeta (0)$ is defined.
 Thus $\zeta (0)$ is a
topological invariant, and,   in particular, it remains invariant under  the conformal perturbation $g \to e^h g$ of the metric. This result was then extended to all conformal classes in the upper half plane in our paper  \cite{MR2956317}.

 After the work of Gauss, a decisive giant step was taken by Riemann in his  epoch-making  paper {\it Ueber die Hypothesen, welche der Geometrie zu Grunde liegen,}
  which is a text of his {\it Habilitationsvortrag} of June 1854.  The notion of {\it space}, as an entity that exists on its own, without any reference to an ambient space or external world,  was first conceived by Riemann. Riemannian geometry is intrinsic from the beginning.  In Riemann's conception,   a space, which he called a {\it mannigfaltigkeit}, manifold  in English, can be   discrete or  continuous, finite or  infinite dimensional. The idea of a geometric space as an abstract {\it set}  endowed with some extra structure was born in this paper of Riemann.  Local coordinates are just labels without any  intrinsic meaning,  and thus one must always make sure that the definitions are independent of the choice of coordinates. This is the {\it general principle of relativity}, which later came to be regarded as a cornerstone of modern theories of spacetime and Einstein's theory of gravitation.  This idea quickly led to the development of tensor calculus, also known as  the {\it absolute differential calculus}, by the Italian school of Ricci and his brilliant student Levi-Civita. 
  
  Riemann  also introduced the idea of a Riemannian metric  and  understood that to define the bending or curvature  
 of a space one just needs a Riemannian metric. This was of course  directly inspired by Gauss' theorema egregium. 
 In fact he gave two definitions  for  curvature. His {\it sectional} curvature is defined as the Gaussian  curvature of   two dimensional submanifolds defined via the  geodesic flow for each two dimensional subspace  of the tangent space at each point. For  his second definition he introduced the  geodesic coordinate systems and considered the Taylor expansion of the metric  components $g_{ij}(x) $ in a geodesic coordinate. Let 
 $$ c_{ij, kl}=\frac{1}{2}\frac{\partial^2 g_{ij}}{\partial x^k \partial x^l}.$$
 He shows that sectional curvature is determined by the components $c_{ij, kl} $,  and vice-versa. Also, one knows that the components $c_{ij, kl} $ are closely related to Riemann curvature tensor.

   The Riemann  curvature tensor, in modern notation,  is defined as 
 $$ R(X,Y) = \conn_X\conn_Y-\conn_Y\conn_X-\conn_{[X,Y]},$$
 where $\nabla$ is the Levi-Civita connection of the metric, and  $X$ and $Y$ are vector fields on the manifold. The analogue of this  curvature tensor of rank four is still an illusive concept in the noncommutative case. However, the components of the Riemann tensor appear in the term $a_4$ in the small time 
heat kernel expansion of the Laplacian of the metric, the analogue of which was calculated and studied in \cite{arXiv1611.09815} for noncommutative two tori and for noncommutative four tori with product geometries.

 \vskip 0.7cm
 
\begin{center}
\centering
\includegraphics[scale=0.4]{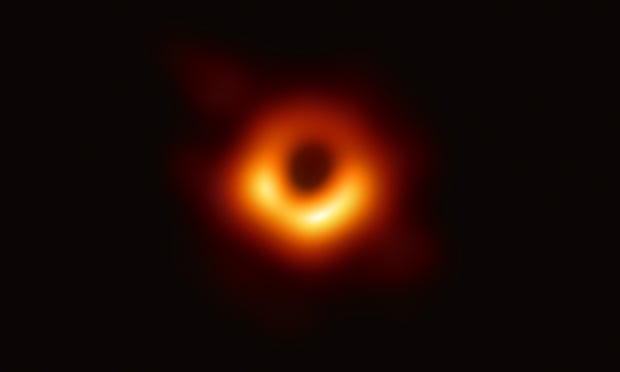}\\
\smallskip 
\smallskip
The first black hole image by Event Horizon Telescope, April 2019. 
\end{center}

\vskip 0.5 cm

 It is hard to exaggerate the importance of the {\it Ricci curvature} in  geometry and  physics. For example it plays an  indispensable role in Einstein's  theory of gravity and his field equations, and directly leads, thanks to  
 Schwarzschild solution,  to the prediction of black holes. It is also fundamental for  the Ricci flow. Ricci curvature  can be formulated  
 in spectral terms and this opened up the possibility of defining it in noncommutative settings \cite{arXiv1612.06688}. The reader should consult later sections in this survey for more on this.

Although  they  won't be a subject for the present exposition,  let us briefly mention some other aspects of curvature that have found their analogues  in noncommutative settings. These are mostly {\it  linear} aspects of curvature, and have much to do with representation theory of groups.  They include Chern-Weil theory of characteristic classes and specially the Chern-Connes character maps for both $K$-theory and $K$-homology, Chern-Simons theory, and Yang-Mills theory.  Riemannian curvature, whose noncommutative analogue we are concerned with here, is a {\it nonlinear} theory and from our point of view that is why it took  so long to find its proper formulation and first calculations in a noncommutative setting.

\subsection{Laplace type operators and Gilkey's theorem}
At the heart of spectral geometry,
  Gilkey's theorem \cite{MR1396308} gives the most precise information on asymptotic expansion of heat kernels for a large class of elliptic PDE's. Since this result and its noncommutative analogue  plays such an important role in defining and computing curvature invariants in noncommutative geometry, we shall explain it  briefly in this section. Let $M$ be a smooth closed manifold with 
  a Riemannian metric $g$  and a vector bundle $V$ on $M$.  An operator
   $P: \Gamma (M, V) \to \Gamma (M, V)$ on smooth sections of $V$  is called a  Laplace
    type operator if in local coordinates it  looks like
$$ P= -g^{ij}\partial_i \partial_j + \text{lower orders}.$$
Examples of Laplace type operators include Laplacian on forms
$$\Delta =(d+ d^*)^2: \Omega^p(M) \to \Omega^p(M),$$ and the Dirac Laplacians $\Delta= D^* D$, where 
$ D: \Gamma (S) \to \Gamma (S)$
is a generalized Dirac operator. 

Now if 
 $P$ is a Laplace type operator, 
then there exists a unique connection $\nabla$ on the vector bundle $V$ and an endomorphism $E\in {\rm End}(V)$ such that   
\begin{equation*}\label{decomposition}
P = \nabla^*\nabla - E.
\end{equation*}  
Here $\nabla^*\nabla$ is the connection Laplacian  which is locally given by $-g^{ij}\nabla_{i}\nabla_{j}$.
For example the  Lichnerowicz formula for the Dirac operator,  $D^2 =\nabla^*\nabla-\frac{1}{4}R$,  gives 
\[
E=\frac{1}{4} R,
\] 
where $R$ is the scalar curvature. 
Now $e^{-tP}$ is a smoothing operator with a smooth kernel $k(t, x, y).$ 
 There is an asymptotic expansion near $t=0$ 
$$k(t, x, x) \sim \frac{1}{(4 \pi t)^{m/2}}  (a_0 (x, P) + a_2 (x, P)t + a_4 (x, P)t^2 + \cdots), $$
where  $a_{2k} (x, P) $ are known as the  Gilkey-Seeley-De Witt coefficients. Gilkey's theorem asserts that 
$a_{2k}(x, P)$ can be expressed in terms of universal polynomials in the metric $g$ and its covariant derivatives. 
Gilkey has computed the frist  four nonzero  terms and they are as follows:
\begin{eqnarray*}
a_0 (x,P)&=& tr(1),\\
a_2(x,P)&=&  tr\big (E-\frac{1}{6}R \big ), \nonumber\\
a_4(x,P) &=& \frac{1}{360} 
tr \Big ( \big(-12 R_{;kk} + 5 R^2  - 2 R_{jk} R_{jk} +2 R_{ijkl} R_{ijkl}\big) \nonumber  \\
&&\qquad\qquad  - 60 R E+ 180 E^2+ 60 E_{;kk} + 30 \Omega_{ij} \Omega_{ij}\Big ). 
\end{eqnarray*}

\begin{align*}
a_6(x,P)&=  
tr
\Big\{
\frac{1}{7!} \Big ( 
-18 R_{; kkll} +17R_{;k} R_{;k}-2R_{jk;l}R_{jk;l} -4 R_{jk;l} R_{jl;k} \nonumber \\
&\qquad\qquad\qquad+ 9 R_{ijku;l}R_{ijku;l}+28R R_{;ll} -8 R_{jk} R_{jk;ll}+24 R_{jk} R_{jl; kl} \nonumber \\
& \qquad\qquad\qquad+12 R_{ijkl}R_{ijkl;uu} \Big ) \nonumber \\
&  + \frac{1}{9 \cdot 7!} \Big ( -35 R^3 + 42 R R_{lp} R_{lp}  
 -42 R R_{klpq} R_{klpq} +208 R_{jk} R_{jl} R_{kl}  \nonumber \\
&\qquad\qquad -192 R_{jk} R_{ul} R_{jukl}+48 R_{jk} R_{julp} R_{kulp} -44 R_{ijku} R_{ijlp} R_{kulp} \nonumber \\
&\qquad\qquad -80 R_{ijku} R_{ilkp} R_{jlup} \Big) \nonumber \\
& + \frac{1}{360} \Big ( 
8 \Omega_{ij;k} \Omega_{ij;k} +2 \Omega_{ij;j} \Omega_{ik;k} +12 
\Omega_{ij} \Omega_{ij;kk}  -12 \Omega_{ij} \Omega_{jk} \Omega_{ki}  \nonumber \\
&\qquad\qquad -6 R_{ijkl} \Omega_{ij} \Omega_{kl}  + 4 R_{jk} \Omega_{jl} \Omega_{kl} 
-5 R \Omega_{kl} \Omega_{kl} 
\Big ) \nonumber \\
&+ \frac{1}{360} \Big( 
6E_{;iijj} +60 EE_{;ii} +30 E_{;i} E_{;i} +60 E^3+30 E \Omega_{ij} \Omega_{ij}   -10 RE_{;kk}\nonumber \\
&\qquad\qquad -4R_{jk} E_{;jk}-12 R_{;k}E_{;k} -30R E^2  -12 R_{;kk} E + 5 R^2 E  \nonumber \\
&\qquad\qquad  -2 R_{jk} R_{jk} E + 
2 R_{ijkl} R_{ijkl} E
\Big )
\Big \}. \nonumber 
\end{align*}
Here $R_{ijkl}$ is the Riemann curvature tensor, $R$ is the scalar curvature, $\Omega$ is the curvature matrix of two  forms, and $;$ denotes the covariant derivative operator.

 As we shall later see in this survey, the first two  terms in the above list   allow us to  define   the scalar and Ricci curvatures in terms of heat kernel coefficients and extend them to noncommutative settings.  

Alternatively,  one can use  spectral zeta functions to extract information from the spectrum. Heat trace and spectral zeta functions  are related via  Mellin transform. 
For a  concrete example, let $\triangle$  denote  the Laplacian on functions on an $m$-dimensional closed Riemannian manifold. Define
$$ \zeta_{\triangle} (s) = \sum \lambda_i^{-s}   \quad \quad \text{Re}(s) >\frac{m}{2}.$$ 
The spectral invariants $a_i$ in the heat trace asymptotic expansion
$$ \text{Trace} (e^{-t \triangle}) \sim
  (4 \pi t)^{\frac{-m}{2}} \sum_{j=0}^\infty a_{j} t^j \quad \qquad  \quad (t \to 0^+)
$$
are related to residues of spectral zeta function by 
$$  \text{Res}_{s=  \alpha}\zeta_{\triangle} (s)=  (4 \pi )^{-\frac{m}{2}}\frac{ a_{\frac{m}{2}-\alpha}}{\Gamma (\alpha)}, \quad \quad   \alpha = \frac{m}{2}-j>0.$$
To get to the local invariants like scalar curvature  we can consider localized zeta functions. 
 Let $\zeta_f (s):= \text{Tr}\, (f\triangle^{-s}),  \, \,  f \in C^{\infty}(M)$. Then we have

$$\text{Res}\, \zeta_f (s)|_{s= \frac{m}{2}-1}= \frac{(4 \pi)^{-m/2}}{\Gamma (m/2-1)}\int_M f (x) R  (x) d vol_x,  \quad m \geq 3,$$
$$\zeta_f (s)|_{s=0} =  \frac{1}{4\pi}\int_M f (x) R  (x) d vol_x    - \text{Tr} (f P),   \quad \quad   m=2,$$
where $P$ is projection onto zero eigenmodes of $\triangle$. Thus the scalar curvature $R$ appears as the density function for the  localized spectral zeta function.

\subsection{Noncommutative Chern-Weil theory} Although it is not our intention to review  this  subject in  the present survey, we shall nevertheless  explain some ideas of noncommutative Chern-Weil theory here.  Many aspects of Chern-Weil theory of characteristic classes for vector bundles and principal bundles over smooth manifolds can be cast in an algebraic formalism and as such is even used in commutative algebra and algebraic geometry \cite{MR3643159}. Thus one can formulate notions like de Rham cohomology, connection, curvature, Chern classes and Chern character,  over a commutative algebra and then for a scheme. This is a commutative theory which is more or less straightforward in the characteristic zero case. But  there  seemed to be  no obvious extension of de Rham theory and the rest of Chern-Weil theory  to the noncommutatuve case. 

In \cite{MR572645} Connes realized that many aspects of Chern-Weil theory  can  be implemented in a noncommutative setting. The crucial ingredient was the discovery of cyclic cohomology that replaces de Rham homology of currents in a noncommutative setting \cite{MR777584, MR823176}. 
Let $A$ be a not necessarily commutative   algebra over the field of complex numbers.   By a \index{noncommutative differential calculus}{\it noncommutative differential calculus} on $A$ we mean
a triple   $(\Omega, \, d, \rho)$ such that $(\Omega, \, d)$ is  a
differential graded algebra and $\rho : A \to \Omega^0$ is an algebra
homomorphism. Given a  right $A$-module
$\mathcal{E}$, a {\it connection} on $\mathcal{E}$ is a
$\mathbb{C}$-linear map
$\nabla :  \mathcal{E} \longrightarrow \mathcal{E} \otimes_A \Omega^1  $
 satisfying the Leibniz rule
$\nabla ( \xi a)= \nabla (\xi)a + \xi \otimes da,$
for all  $\xi \in \mathcal{E}$  and $a\in A$. Let $ \hat{\nabla} : \mathcal{E}\otimes_{\mathcal{A}} \Omega^{\bullet} \to \mathcal{E}\otimes_{\mathcal{A}} \Omega^{\bullet +1}$
be the (necessarily unique) extension of $\nabla$  which satisfies the  graded Leibniz rule
$\hat{\nabla} (\xi \omega) = \hat{\nabla} (\xi) \omega + (-1)^{\text{deg}\, \xi}\xi
d\omega$ with respect to the right $\Omega$-module structure on
$\mathcal{E}\otimes_{\mathcal{A}} \Omega$. 
The
{\it curvature}  of $\nabla$ is the operator of degree 2,    $\hat{\nabla}^2: \mathcal{E} \otimes_A \Omega^{\bullet} \to
\mathcal{E} \otimes_A \Omega^{\bullet}, $  which
can be
 easily checked to be $\Omega$-linear. 

Now to obtain Connes' Chern character pairing between $K$-theory and cyclic cohomology,  $K_0(A) \otimes  HC^{2n}(A) \to \mathbb{C},$ one can proceed as follows. Given a finite projective $A$-module $\mathcal{E}$, one can always equip $\mathcal{E}$ with a connection over the universal differential calculus $\Omega A$. An element of $HC^{2n}(A)$ can be represented by a closed graded trace  $\tau$ on $\Omega^{2n} A.$    The value of the pairing is then simply $\tau (\hat{\nabla}^{2n})$. Here we used the same symbol $\tau$ to denote the extension of $\tau$ to the ring $ End_{ \Omega^{\bullet}}(\mathcal{E} \otimes_A \Omega^{\bullet}).$ One checks that this definition is independent of all choices that we made \cite{MR823176}. Connes in fact initially developed the more sophisticated Chern-Connes pairing in $K$-homology with explicit formulas that do  not have a commutative counterpart. For all this and more the reader should check, Connes' book and  his above cited article \cite{MR823176,MR1303779} as well as the book \cite{MR3134494}.

\subsection{From spectral geometry to spectral triples}

 The very notion of Riemannian manifold itself  is now subsumed and vastly generalized through 
Connes'  notion of {\it  spectral triples}, which is a centerpiece of noncommutative  geometry
and applications of noncommutative geometry to particle physics.

Let us first  motivate the definition of a spectral triple. During the course of their heat equation proof of the index theorem, it was discovered by Atiyah-Bott-Patodi \cite{MR0650828} that it is enough to prove the theorem for  Dirac operators twisted by vector bundles. The reason is that these twisted  Dirac operators in fact generate the whole K-homology group of a spin manifold and thus it suffices to prove the theorem only for these first order elliptic operators. This indicates the preeminence of Dirac operators in topology. As we shall see below,  Dirac operators also encode metric information of a Riemannian manifold in a succint way. Broadly speaking, spectral triples,  suitably enhanced, are noncommutative spin manifolds and form a backbone  of noncommutative geometry, specially its metric aspects.  One precise formulation of this idea is  {\it Connes' reconstruction theorem} \cite{MR3032810}  which  states that a commutative spectral triple satisfying some natural conditions is in fact the standard spectral triple of a $\text{spin}^c$ manifold  described below.

Recall that the Dirac operator $D$ 
  on a
compact Riemannian $\text{spin}^c$  manifold  acts  as an unbounded
selfadjoint operator on  the Hilbert space $L^2(M,S)$ of $L^2$-spinors on $M$. If we  let
$C^{\infty}(M)$ act  on $L^2(M,S)$ by multiplication operators, then  one can check that for any smooth
function $f$, the commutator $[D,f]=Df-fD$ extends to a bounded
operator on  $L^2(M,S)$. The metric $d$ on $M$, that is   the geodesic distance of  $M$, can be recovered thanks to the 
{\it distance formula} of Connes \cite{MR1303779}:
\[d(p,q) = \textnormal{Sup} \{ |f(p)-f(q)|; \parallel [D,f] \parallel \leq 1\}.
\nonumber \]
The triple $(C^{\infty}(M), L^2(M,S), {D \hspace{-7pt} \slash})$ is a commutative example of
a spectral triple. 

The  general definition of a spectral triple, in the odd case, is as
follows.
\begin{definition}
Let $A$ be a unital algebra. An odd spectral triple on $ A$ is a triple $(A,
\mathcal{H}, D)$ consisting  of a Hilbert space $\mathcal{H}$, a selfadjoint unbounded operator
{\rm $D:  \text{Dom} (D) \subset \mathcal{H} \to \mathcal{H}$} with
compact resolvent, i.e.,
$(D - \lambda)^{-1} \in \mathcal{K}(\mathcal{H}),  \text{for all} \,
\lambda \notin \mathbb{R},$
 and a
representation  $\pi: A \to \mathcal{L}(\mathcal{H})$  of $A$ such that for
all $a\in A$, the commutator $[D, \pi (a)]$ is defined on {\rm $ \text{Dom} (D)$} and extends to  a bounded operator on $\mathcal{H}$.
\end{definition}
A  spectral
triple is called {\it finitely summable}  if for some $n\geq 1$
\begin{equation*}\label{stfs} |D|^{-n} \in \mathcal{L}^{1, \, \infty}(\mathcal{H}).
\end{equation*}
Here $\mathcal{L}^{1, \, \infty}(\mathcal{H})$ is the Dixmier ideal. It is an ideal of compact operators which is slightly bigger than the ideal of trace class operators and is the natural domain of the Dixmier trace. 
Spectral triples provide a refinement of Fredholm modules. Going from Fredholm modules to spectral triples is
similar to going from the conformal class of a Riemannian metric to the
metric itself. Spectral triples simultaneously  provide a notion of
{\it Dirac operator} in noncommutative geometry, as well as a Riemannian
type {\it distance function} for noncommutative spaces. In later sections we shall define and work with concrete examples of spectral triples and their conformal perturbations.

\section{{\bf Pseudodifferential calculus and heat expansion}}
\label{pseudoandheatexpsec}

In this section we discuss the classical 
pseudodifferential calculus on the Euclidean space 
and will then provide practical details of the 
pseudodifferential calculus of \cite{MR572645} that we use 
for heat kernel calculations on noncommutative tori. 

\subsection{Classical pseudodifferential calculus}

In the Euclidean case we follow the notations and 
conventions of \cite{MR1396308} as follows. For any 
multi-index $\alpha=(\alpha_1, \dots, \alpha_m)$ of 
non-negative integers and coordinates 
$x=(x_1, \dots, x_m) \in \mathbb{R}^m$ we set: 
\begin{equation*}
|\alpha| = \alpha_1 + \cdots + \alpha_m, 
\qquad 
\alpha! = \alpha_1 ! \cdots \alpha_m !,  
\qquad
x^\alpha = x_1^{\alpha_1} \cdots x_m^{\alpha_m}, 
\end{equation*}
\begin{equation*}
\partial_x^\alpha
= 
\left ( \frac{\partial}{\partial x_1}  \right )^{\alpha_1}  
\cdots 
\left ( \frac{\partial}{\partial x_m}  \right )^{\alpha_m}, 
\qquad 
D_x^\alpha = (-i)^{|\alpha|}   \partial_x^\alpha.  
\end{equation*}
Also we normalize the Lebesgue measure on 
$\mathbb{R}^m$ by a multiplicative factor of 
$(2 \pi)^{-m/2}$ and still denote it by $dx$. Therefore 
we have: 
\begin{equation*}
\int_{\mathbb{R}^m} \exp \left ( -\frac{1}{2} |x|^2 \right ) \, dx=1. 
\end{equation*}

The main idea behind pseudodifferential calculus 
is that it uses the {\it  Fourier transform} to turn a 
differential operator into multiplication by a function, 
namely the {\it symbol} of the 
differential operator. The Fourier transform 
$\hat f$ of a {\it Schwartz function } $f$ 
on $\mathbb{R}^m$ is defined by the following 
integration: 
\begin{equation*}
\hat f (\xi) = \int_{\mathbb{R}^m} e^{- i x \cdot \xi} f(x) \, dx, \qquad \xi \in \mathbb{R}^m. 
\end{equation*}
This integral is convergent because, by definition, 
the set of Schwartz functions 
$\mathcal{S}(\mathbb{R}^m)$ consists of all 
complex-valued smooth functions 
$f$ on the Euclidean space such that for any 
multi-indices $\alpha$ and $\beta$ of non-negative 
integers 
\begin{equation*}
\sup_{x \in \mathbb{R}^m} |x^\alpha D^\beta f (x)|  < \infty.  
\end{equation*}
It turns out that the Fourier transform preserves 
the $L^2$-norm, hence it extends to a 
unitary operator on $L^2(\mathbb{R}^m)$.

The differential operator $D_x^\alpha$ turns in 
the Fourier mode to multiplication by 
the monomial $\xi^\alpha$, in the sense that: 
\begin{equation*}
\widehat{ (D_x^\alpha f )}(\xi) = \xi^ \alpha  \hat f (\xi). 
\end{equation*}
The monomial $\xi^\alpha$ is therefore called 
the symbol of the differential operator 
$D_x^\alpha$. Then,  the {\it Fourier inversion formula,}
\begin{equation*}
f(x) = \int_{\mathbb{R}^m} e^{i \xi \cdot x} \hat f (\xi) \, d\xi, \qquad f \in \mathcal{S}(\mathbb{R}^m), 
\end{equation*}
implies that 
\begin{equation} \label{symboldiffoperatorformula}
D_x^\alpha f(x) = \int_{\mathbb{R}^m} e^{i x \cdot \xi} \xi^\alpha \hat f(\xi) \, d\xi 
= \int_{\mathbb{R}^m} \int_{\mathbb{R}^m} e^{i(x-y) \cdot \xi} \xi^\alpha f(y) \, dy \, d\xi. 
\end{equation}

 It is now clear from the above facts that the 
 symbol of any differential operator, given by 
 a finite sum of the form $\sum a_{\alpha}(x) D_x^\alpha$, 
 is the polynomial in $\xi$ of the form 
 $\sum a_{\alpha}(x) \xi^\alpha,$ whose coefficients are 
the functions $a_\alpha(x)$ (which we assume to be 
smooth). Using the notation 
$\sigma(\cdot)$ for the symbol it is an easy exercise 
to see that given two differential 
operators $P_1$ and $P_2$, the symbol of their 
composition $\sigma \left ( P_1 \circ P_2 \right ) $
is given by the following expression: 
\begin{equation} \label{compositiondiffsymbols}
\sum_{\alpha \in \mathbb{Z}_{\geq 0}^m} \frac{1}{\alpha !} \partial_\xi^\alpha \sigma(P_1) \, D_x^{\alpha} \sigma(P_2),  
\end{equation}
which is a finite sum because only finitely many 
of the summands are non-zero. 

By considering a wider family of symbols, 
one obtains a larger family of operators which are 
called {\it pseudodifferential operators}. A smooth function $p : \mathbb{R}^m \times \mathbb{R}^m \to \mathbb{C}$ 
is {\it a pseudodifferential symbol of order $d \in \mathbb{R}$} if it satisfies the following 
conditions: 
\begin{itemize}
\item $p(x, \xi)$ has compact support in $x$, 
\item for any multi-indices $\alpha, \beta \in \mathbb{Z}_{\geq 0}^m$, there 
exists a constant $C_{\alpha, \beta}$ such that
\begin{equation} \label{classicalsymboldef}
|\partial_\xi^\beta \partial_x^\alpha p (x, \xi) | \leq C_{\alpha, \beta} (1 + |\xi|)^{d-|\beta|}.  
\end{equation}
\end{itemize}
Clearly the space of pseudodifferential symbols 
possesses a filtration because,  denoting  
the space of symbols of order $d$ by $S^d$, we have: 
\begin{equation*}
d_1 \leq d_2   \implies S^{d_1} \subset S^{d_2}.  
\end{equation*}
Existence of symbols of arbitrary orders can be 
assured by observing that 
for any $d \in \mathbb{R}$ and any compactly 
supported function $f_0$, the 
function $p(x, \xi) = f_0(x) (1 + |\xi|^2)^{d/2}$ 
belongs to $S^d$. 

Given a symbol $p \in S^d$, inspired by formula \eqref{symboldiffoperatorformula}, the corresponding 
pseudodifferential operator $P$ is defined by 
\begin{equation} \label{pseudodiffformula}
P f(x) = \int_{\mathbb{R}^m} e^{i x \cdot \xi} p(x, \xi) \hat f(\xi) \, d\xi,  
\qquad f \in \mathcal{S}(\mathbb{R}^m).  
\end{equation}
The space of pseudodifferential operators 
associated with symbols of order $d$ 
is denoted by $\Psi^d(\mathbb{R}^m)$. 
Search for an analog of formula 
\eqref{compositiondiffsymbols} for general 
pseudodifferential operators 
leads to a complicated analysis which, at the 
end, gives {\it an asymptotic expansion} 
for the symbol of the composition of such operators. 
The formula is written as 
\begin{equation} \label{symbolcompositionformulaeuclidean}
\sigma \left ( P_1 P_2 \right ) \sim 
\sum_{\alpha \in \mathbb{Z}_{\geq 0}^m} 
\frac{1}{\alpha !} \partial_\xi^\alpha \sigma(P_1) \, D_x^{\alpha} \sigma(P_2).   
\end{equation}

It is important to put in order some explanations about this formula. 
If $\sigma(P_1) \in S^{d_1}$ and $\sigma(P_2) \in S^{d_2}$ then 
there is a symbol in $S^{d_1 + d_2}$ that gives $P_1 \circ P_2$ 
via formula \eqref{pseudodiffformula}.  However $\sigma(P_1 \circ P_2)$ has a complicated 
formula which involves integrals, which can be seen  by writing the formulas 
directly. The trick is then to use Taylor series and to perform analytic 
manipulations on the closed formula for  $\sigma(P_1 \circ P_2)$ 
to derive the expansion \eqref{symbolcompositionformulaeuclidean}. The error terms in the Taylor series 
that one uses in the manipulations are responsible for having 
an asymptotic expansion rather than a strict identity. The precise 
meaning of this expansion is that given any $d \in \mathbb{R}$, there 
exists a positive integer $N$ such that 
\begin{equation*}
\sigma \left ( P_1  P_2 \right ) - 
\sum_{|\alpha| \leq N } 
\frac{1}{\alpha !} \partial_\xi^\alpha \sigma(P_1) \, D_x^{\alpha} \sigma(P_2) 
\in S^{d}.   
\end{equation*} 
Therefore, as one subtracts the terms 
$\frac{1}{\alpha !} \partial_\xi^\alpha \sigma(P_1) \, D_x^{\alpha} \sigma(P_2) $ 
from $\sigma(P_1 \circ P_2)$, the orders of the resulting symbols tend to 
$- \infty$. Regarding this, it is convenient to introduce the space 
$S^{-\infty} = \cap_{d\in \mathbb{R}} S^d$ of 
the {\it  infinitely smoothing pseudodifferential symbols}.  
For example for any compactly supported function $f_0$, the symbol 
$p(x, \xi) = f_0(x) e^{-|\xi|^2}$ belongs to $S^{-\infty}$. 

The composition rule \eqref{symbolcompositionformulaeuclidean} is a very useful tool. For instance it can 
be used to find a {\it  parametrix} for {\it elliptic} pseudodifferential operators.   
Important geometric operators such as Laplacians are elliptic, and by 
finding a parametrix, as we shall explain, one finds an approximation of 
the fundamental solution of the partial differential equation defined by 
such an important operator. Intuitively, a pseudodifferential symbol $p(x, \xi)$ 
of order $d \in \mathbb{R}$ is {\it elliptic} if it is non-zero when $\xi$ is away from 
the origin (or invertible in the case of matrix-valued symbols), and $|p(x, \xi)^{-1}|$ 
is bounded by a constant times $(1+|\xi|)^{-d}$ as $\xi \to \infty$. For our purposes, 
it suffices to know that a differential operator $D = \sum a_\alpha(x) D_x^\alpha$ 
of order $d = \max_\alpha |\alpha| $ is elliptic if its {\it leading symbol}, 
\begin{equation*}
\sigma_L(D) = \sum_{|\alpha|= d} a_\alpha(x) \xi^\alpha,  
\end{equation*} 
is non-zero (or invertible) for $\xi  \neq 0$. Given such an elliptic differential operator 
one can use formula \eqref{symbolcompositionformulaeuclidean} to find an 
inverse for $D$, called a {\it  parametrix}, 
in the quotient $\Psi/ \Psi^{-\infty}$ 
of the algebra of pseudodifferential operators $\Psi$ by infinitely smoothing operators 
$\Psi^{-\infty}$. This process can be described as follows. One makes the natural 
assumption that the symbol of the parametrix has an expansion starting with a leading term 
of order $-d$ and other terms whose orders descend to $-\infty$, namely terms of orders $-d-1$, $-d-2$, $\dots$, and one continues as follows. 
The formula given by \eqref{symbolcompositionformulaeuclidean} can be used to 
find these terms recursively and thereby find a parametrix $R$ such that 
\begin{equation*}
DR -I \sim RD - I \sim 0. 
\end{equation*}
We will illustrate this carefully in \S \ref{heatexpsubsec} in a slightly more complicated situation, 
where a parameter $\lambda$ and a parametric pseudodifferential calculus is involved in 
deriving heat kernel expansions. We just mention that invertibility of $\sigma_L(D)$ is the 
crucial point that allows one to start the recursive process, and to continue on to find the 
parametrix $R$.

\subsection{Small-time heat kernel expansion} \label{heatexpsubsec}

For simplicity and practical purposes we assume that $P$ is a positive elliptic 
differential operator of order $2$ with 
\begin{equation*}
\sigma(P) = p_2(x, \xi) + p_1(x, \xi) + p_0(x, \xi), 
\end{equation*}
where each $p_j$ is (homogeneous) of order $j$ in $\xi$. We know that 
$p_2(x, \xi)$ is non-zero (or invertible) for non-zero $\xi$. The first step in 
deriving a small time asymptotic expansion for $\textrm{Tr}(\exp(-t P))$ as 
$t \to 0^+$, is to use the Cauchy integral formula to write 
\begin{equation} \label{CauchyIntformula}
e^{-t P} = \frac{1}{2 \pi i} \int_\gamma e^{-t \lambda} (P- \lambda)^{-1} \, d\lambda, 
\end{equation}
where the contour $\gamma$ goes clockwise around the non-negative real axis, 
where the eigenvalues of $P$ are located. The term $(P- \lambda)^{-1}$ in the 
above integral can now be approximated by pseudodifferential operators as follows. 
We look for an approximation $R_\lambda$ of $(P- \lambda)^{-1}$ such that 
\begin{equation*}
\sigma(R_\lambda) \sim r_0(x, \xi, \lambda) + r_1(x, \xi, \lambda) + r_2(x, \xi, \lambda) +\cdots,  
\end{equation*}
where each $r_j$ is a symbol of order $-2-j$ in the parametric sense which we will 
elaborate on later. For now one can use formula \eqref{symbolcompositionformulaeuclidean} to find the $r_j$ recursively 
out of the equation 
\begin{equation*}
R_\lambda  (P- \lambda) \sim I. 
\end{equation*}

This means that the terms $r_j$ in the expansion should satisfy 
\begin{equation} \label{parametrixeq}
\sum_j r_j \circ \left (  (p_2 - \lambda) + p_1 + p_0\right ) \sim 1, 
\end{equation}
where the composition $\circ$ is given by \eqref{symbolcompositionformulaeuclidean}. 
By writing the expansion 
one can see that there is only one leading term, which is of order 0, namely 
$r_0 (p_2 - \lambda)$ and needs to be set equal to 1 so that it matches the corresponding 
(and the only term) on the right hand side of the equation \eqref{parametrixeq}. 
Therefore the leading 
term $r_0$ is found to be 
\begin{equation} \label{r0formula}
r_0 = (p_2 - \lambda)^{-1}. 
\end{equation}
Here the ellipticity plays an important role, because we need to be 
ensured that the inverse of $p_2 - \lambda$ exists. Since, in our 
examples, $P$ will be a Laplace type operator, the leading term 
$p_2$ is a positive number  (or a positive invertible matrix in the vector bundle 
case) for any  $\xi \neq 0$. 
Therefore for any $\lambda$ on the contour  $\gamma$, we 
know that $p_2 - \lambda$ is 
invertible. One can then proceed by considering the term that 
is homogeneous of order $-1$ in the expansion 
of the left hand side of \eqref{parametrixeq} and set it equal to 0 since 
there is no term of order $-1$ on the right hand side. This will yield a 
formula for the next term $r_1$. By 
continuing this process one finds recursively that for $n=1,2,3, \dots$, we have 
\begin{equation}  \label{rnformula}
r_n = - \left ( 
\sum_{\substack{|\alpha|+j+2-k=n, \\ 0 \leq j<n, \, 0 \leq k \leq 2}} \frac{1}{\alpha!}
\partial_\xi^\alpha r_j \, D_x^\alpha p_k  
\right ) 
r_0. 
\end{equation}
It turns out that the $r_n$ calculated by this formula have the following homogeneity property:
\begin{equation*}
r_n (x, t \xi, t^2 \lambda) = t^{-2-n} r_n(t, \xi, \lambda). 
\end{equation*}

Having an approximation of the resolvent $R_\lambda \sim (P- \lambda)^{-1}$ 
via the symbols $r_n$, one can use the formulas \eqref{CauchyIntformula} and \eqref{pseudodiffformula} to approximate 
the kernel $K_t$ of the operator $e^{-tP}$, namely the unique smooth function such 
that 
\begin{equation*}
e^{-t P} f (x) = \int K_t(x, y) \, f(y) \, dy, \qquad f \in \mathcal{S}(\mathbb{R}^m). 
\end{equation*}  
Since $\Tr(e^{-t P})$ can be calculated by integrating the kernel on the diagonal, 
\begin{equation*}
\Tr \left ( e^{-t P} \right ) = \int K_t(x, x) \, dx, 
\end{equation*}
the integration of the approximation of the kernel obtained by going through 
the procedure described above leads to an asymptotic expansion of the 
following form: 
\begin{equation} \label{heatexpgeneralform}
\Tr \left ( e^{-t P} \right ) \sim_{t \to 0^+} t^{-m/2} \sum_{n=0}^\infty a_{2n} (P)\, t^n, 
\end{equation}
where each coefficient $a_{2n}$ is the integral of a density $a_{2n}(x, P)$ given by 
\begin{equation*}
a_{2n}(x, P) =   \frac{1}{2 \pi i} \int \int_\gamma e^{-\lambda} \textrm{tr} ( r_{2n} (x, \xi, \lambda) ) \, d\lambda \, d\xi.  
\end{equation*}
In this integrand, the $\textrm{tr}$ denotes the matrix trace which needs to be considered 
in the case of vector bundles.

It is a known fact that when $P$ is a geometric operator such as 
the Laplacian of a metric, each $a_{2n}(x, P)$ can be written 
in terms of the Riemann curvature tensor, its contractions, and 
covariant derivatives, see for example \cite{MR2371808}. However, 
in practice, as $n$ grows, these terms 
become so complicated rapidly. One can refer to \cite{MR1396308} for the formulas 
for the terms up $a_6$ derived using invariant theory.

\subsection{Pseudodifferential calculus and heat kernel expansion for noncommutative tori} 
\label{nctorusheatexpsubsec}

Now that we have illustrated the derivation of the heat kernel expansion \eqref{heatexpgeneralform}, we explain 
briefly in this subsection that using the pseudodifferential calculus developed in \cite{MR572645} for 
$C^*$-dynamical systems, heat 
kernel expansions of Laplacians on noncommutative tori can be derived by taking a parallel 
approach.  We note that, in \cite{MR3825195}, for {\it toric manifolds,} the Widom pseudodifferential calculus is adapted 
to their noncommutative deformations and it is used for the derivation of heat kernel expansions.

We first recall the pseudodifferential calculus on the algebra of noncommutative 
$m$-torus. A {\it pseudodifferential symbol} of order $d \in \mathbb{Z}$ on $\nctm$ 
is a smooth mapping $\rho : \mathbb{R}^m \to \snctm$ such that for any multi-indices 
$\alpha$ and $\beta$ of non-negative integers, there exists a constant $C_{\alpha, \beta}$ 
such that 
\begin{equation*}
|| \partial_\xi^\beta \delta^\alpha \rho (\xi)|| \leq C_{\alpha, \beta} (1 + |\xi|)^{d - |\beta|}. 
\end{equation*}
Here $||\cdot||$ denotes the $C^*$-algebra norm, which is the equivalent of the 
supremum norm in the commutative setting. Therefore this definition is the noncommutative 
analog of the definition given by \eqref{classicalsymboldef} in the classical case. A symbol of order $d$ is 
{\it elliptic} if $\rho(\xi)$ is invertible for large enough $\xi$ and there exists a constant 
$C_\rho > 0$ such that 
\begin{equation*}
||\rho(\xi)^{-1}|| \leq C_\rho (1+ |\xi|)^{-d}. 
\end{equation*}

Given a pseudodifferential symbol on $\nctm$ the 
corresponding pseudodifferential operator $P_\rho : \snctm \to \snctm$ is defined in \cite{MR572645} by 
the oscillatory integral
\begin{equation} \label{nctoruspseudodiff}
P_\rho(a) = \iint e^{-is \cdot \xi} \rho(\xi) \, \alpha_s(a) \, ds \, d\xi , \qquad a \in \snctm, 
\end{equation}
where $\alpha_s$ is the dynamics given by 
\[
\alpha_s(U^\alpha) = e^{i s\cdot \alpha} U^\alpha.
\] For example, the symbol of 
a differential operator of the form $\sum_{|\alpha| \leq d} a_\alpha \delta^\alpha$, $a_\alpha \in \snctm$, 
is $\sum_{|\alpha| \leq d} a_\alpha \xi^\alpha$.

Given a positive elliptic operator $P$ of order 2 acting on $\snctm$, 
such as the Laplacian of a metric, in order to 
derive an asymptotic expansion for $\Tr(e^{-tP})$ one can start by writing 
the Cauchy integral formula as we did in formula \eqref{CauchyIntformula}. However now one 
has to use the pseudodifferential calculus given by \eqref{nctoruspseudodiff} to write $P-\lambda$ 
in terms of its symbol and thereby approximate its inverse. In this calculus, 
if $\rho_1$ and $\rho_2$ are respectively symbols of orders $d_1$ and $d_2$, then 
the composition $P_{\rho_1} P_{\rho_2} $ has a symbol of order $d_1 + d_2$ with 
 the following asymptotic expansion: 
 \begin{equation} \label{pseudonctoruscompositionrule}
\sigma \left ( P_{\rho_1}  P_{\rho_2} \right ) \sim 
\rho_1 \circ \rho_2  := 
\sum_{\alpha \in \mathbb{Z}_{\geq 0}^m} 
\frac{1}{\alpha !} \partial_\xi^\alpha \rho_1 \, \delta^{\alpha} \rho_2.   
\end{equation}

Having these tools available, one can then perform calculations as in the 
process illustrated in \S \ref{heatexpsubsec} to derive an asymptotic expansion for 
$\Tr(e^{-tP})$. That is, one writes $\sigma(P)= p_2 + p_1+p_0$, where 
each $p_j$ is homogeneous  of order $j$, and finds recursively the terms 
$r_j$, $j=0,1,2, \dots$, that are homogeneous of order $-2-j$ and 
\begin{equation*}
\sum_j r_j \circ \left (  (p_2 - \lambda) + p_1 + p_0\right ) \sim 1.
\end{equation*}
This means that we are using the composition rule \eqref{pseudonctoruscompositionrule} to 
approximate the inverse of $P- \lambda$. The result of this process is 
a recursive formula  similar to the one given by \eqref{r0formula} and \eqref{rnformula}.  
That is, one finds that 
\begin{equation} \label{r0nctorusformula}
r_0 = (p_2 - \lambda)^{-1}. 
\end{equation}
and for $n=1,2,3, \dots,$
\begin{equation} \label{rnnctorusformula}
r_n = - \left ( 
\sum_{\substack{|\alpha|+j+2-k=n, \\ 0 \leq j<n, \, 0 \leq k \leq 2}} \frac{1}{\alpha!}
\partial_\xi^\alpha r_j \, \delta^\alpha p_k  
\right ) 
r_0. 
\end{equation}

Then one finds the small asymptotic expansion
\begin{equation*}
\Tr(e^{-t P}) \sim_{t \to 0^+} t^{-m/2} \sum_{n=0}^\infty \varphi_0(a_{2n}) t^n, 
\end{equation*} 
where $\varphi_0$ is the canonical trace 
\[
\varphi_0 \left ( \sum_{\alpha \in \Z^m}a_\alpha U^\alpha \right  ) = a_0
\] providing us with integration on the 
noncommutative torus $\nctm$. The terms $a_{2n} \in \snctm$ can be calculated 
using \eqref{r0nctorusformula} and \eqref{rnnctorusformula} as follows: 
\begin{equation} \label{nctorusa2nformula}
a_{2n} = \frac{1}{2 \pi i} \int_{\mathbb{R}^m} \int_\gamma 
e^{-\lambda} r_{2n} (\xi, \lambda)\, d\lambda \, d\xi. 
\end{equation}
We shall see in \S \ref{GBSCnc2section} that in order to perform this type of integrals 
in the noncommutative setting one encounters noncommutative features 
which will lead to the appearance of a functional calculus with a modular 
automorphism  in the outcome of the integrals.

\section{{\bf Gauss-Bonnet theorem and curvature for  noncommutative 2-tori}}
\label{GBSCnc2section}

The Gauss-Bonnet theorem for smooth oriented surfaces is a fundamental 
result that establishes a bridge between topology and differential geometry of surfaces. 
Given a surface, its Euler characteristic is a topological invariant which can 
be calculated by choosing an arbitrary triangulation on the surface and 
forming an alternating summation on the number of its vertices, edges and 
faces. It is quite remarkable that the Euler characteristic is independent of 
the choice of  triangulation and depends only on the genus of the surface. 
Clearly, under a diffeomorphism, or roughly 
speaking under changes on the surface that do not change the genus, 
the Euler characteristic remains unchanged. However the scalar curvature 
of the surface changes under such changes by 
diffeomorphisms, say when the surface is embedded in the 3-dimensional 
Euclidean space and has inherited the metric of the ambient space. However, 
the striking fact, namely the statement of the Gauss-Bonnet theorem, is that 
the change of curvature on the surface occurs in a way that, the increase 
and decrease of curvature over the surface compensate for each other 
to the effect that the curvature integrates to the Euler characteristic,  up to multiplication 
by a universal 
constant that is independent of the surface. Hence, the total curvature, namely 
the integral of the scalar 
curvature over the surface,  is a topological invariant. 

\subsection{Scalar curvature and Gauss-Bonnet theorem for $\ncttwo$}
\label{SCandGB2subsec}

In noncommutative geometry, the analog of the Gauss-Bonnet theorem 
has been investigated for the noncommutative two torus. In this 
setting, the flat geometry of  $\ncttwo$ was conformally perturbed by means 
of a conformal factor $e^{-h}$, where $h$ is a selfadjoint element in $\sncttwo$. 
In late 1980's, a heavy calculation was performed by P. Tretkoff and A. Connes 
to find an expression for the 
analog of the total curvature  of the perturbed 
metric on $\ncttwo$. The expression had a heavy dependence on the element 
$h$ used for changing the metric, therefore it was not clear whether the analog 
of the Gauss-Bonnet theorem holds for $\ncttwo$, and they just recorded the 
result of their calculations in an MPI preprint \cite{CohCon}. However, following calculations 
for the spectral action in the presence of a dilaton \cite{MR2239979} and developments in the 
theory of twisted spectral triples \cite{MR2427588}, there were indications that the complicated expression 
for the total curvature has to be independent of the element $h$. By further calculations, 
simplifications and using symmetries in the result, it was shown in \cite{MR2907006} that the terms 
in the complicated expression for the total curvature indeed cancel 
each other out  to $0$, hence the analog of the Gauss-Bonnet theorem for $\ncttwo$. 
The conformal class of metrics that was used in \cite{MR2907006} is associated with the simplest 
translation-invariant complex structure on $\ncttwo$, namely the complex structure 
associated with $i=\sqrt{-1}$. The Gauss-Bonnet theorem for $\ncttwo$ for 
the complex structure associated with an arbitrary complex number $\tau$ in the 
upper-half plane was established in \cite{MR2956317}.

After considering a general complex number $\tau$ in the upper half-plane 
to induce a complex structure and thereby a conformal structure on $\ncttwo$, 
and by conformally perturbing the flat metric in this class by a fixed conformal 
factor $e^{-h}$, $h=h^* \in \sncttwo$, the Laplacian of the curved metric is 
shown \cite{MR2907006, MR2956317} to be anti-unitarily equivalent to the operator 
\[
\triangle_{\tau, h} = e^{h/2} \triangle_{\tau, 0} \,e^{h/2},  
\]
where 
\[
\triangle_{\tau, 0} = \delta_1^2 + 2 \tau_1 \delta_1 \delta_2 + |\tau|^2 \delta_2^2
\]
is the Laplacian of the flat metric in the conformal class determined by 
$\tau=\tau_1 + i \tau_2$ in the upper half-plane.  
The pseudodifferential symbol of $\triangle_{\tau, h} $ is the sum of the following 
homogeneous components of order 2, 1 and 0, in which we use $k=h/2$ for 
simplicity:  
\begin{eqnarray*}  
p_2(\xi)&=& \xi_1^2k^2+|\tau|^2\xi_2^2k^2+2\tau_1\xi_1\xi_2k^2, \\ 
p_1(\xi)&=&2\xi_1k\delta_1(k) + 2|\tau|^2\xi_2k\delta_2(k) +
2\tau_1\xi_1k\delta_2(k)+2\tau_1\xi_2k\delta_1(k),\\
p_0(\xi)&=& k\delta_1^2(k)+ |\tau|^2k\delta_2^2(k) + 2\tau_1k\delta_1\delta_2(k).
\end{eqnarray*}

The analog of the scalar curvature is then the term $a_2(\triangle_{\tau, h}) \in \sncttwo$ 
appearing in the small time ($t \to 0^+$) asymptotic expansion
\begin{equation}
\label{heatexpnc2ttauh}
\Tr(a e^{-t \triangle_{\tau, h}}) 
\sim 
t^{-1} \sum_{n=0}^\infty \varphi_0 \left (a \, a_{2n}(\triangle_{\tau, h}) \right ) \, t^n, 
\qquad a \in \sncttwo. 
\end{equation}

By going through the process illustrated in \S \ref{nctorusheatexpsubsec} one can calculate $a_2$. However, there 
is a purely noncommutative obstruction for the calculation of the involved integrals 
in formula \eqref{nctorusa2nformula}, 
namely one encounters integration of $C^*$-algebra valued functions defined on the 
Euclidean space, $\mathbb{R}^2$ in this case. By passing to a suitable variation of the 
polar coordinates, the angular integration can be performed easily, and the main 
obstruction remains in the radial integration which can be overcome by the following 
rearrangment lemma \cite{MR2907006, MR2947960, MR3194491, MR3626561}: 

\begin{lemma} \label{rearrlemma}
For any tuple $m=(m_0,m_1, \dots, m_\ell) \in \mathbb{Z}^{\ell+1}_{>0}$ 
and elements $\rho_1, \dots, \rho_\ell \in C^\infty(\mathbb{T}_\theta^2)$,  
one has 
\[
\int_0^\infty \frac{u^{|m| -2}}{  (e^hu+1)^{m_0}} \prod_{1}^\ell 
\rho_j (e^hu+1)^{-m_j} \,du 
= e^{-(|m| -1)h} F_m (\Delta_{(1)}, \dots, \Delta_{(\ell)}) \Big (
\prod_1^\ell \rho_j \Big), 
\]
where
\[
 F_m(u_1, \dots, u_\ell) 
=\int_0^\infty \frac{x^{|m| -2}}{(x+1)^{m_0}}  \prod_1^\ell \Big (
x \prod_{1}^j u_h+1 \Big)^{-m_j}\, dx, 
\]
and $\Delta$ is the modular automorphism
\[
\Delta(a)=e^{-h}ae^h, \qquad a \in C(\mathbb{T}_\theta^2).
\]
\end{lemma}

After applying this lemma to the numerous integrands with the help of 
computer programming, the result for the scalar curvature $a_2(\triangle_{\tau, h})$ 
was calculated in \cite{MR3194491, MR3148618}:

\begin{figure}[H]
\includegraphics[scale=0.6]{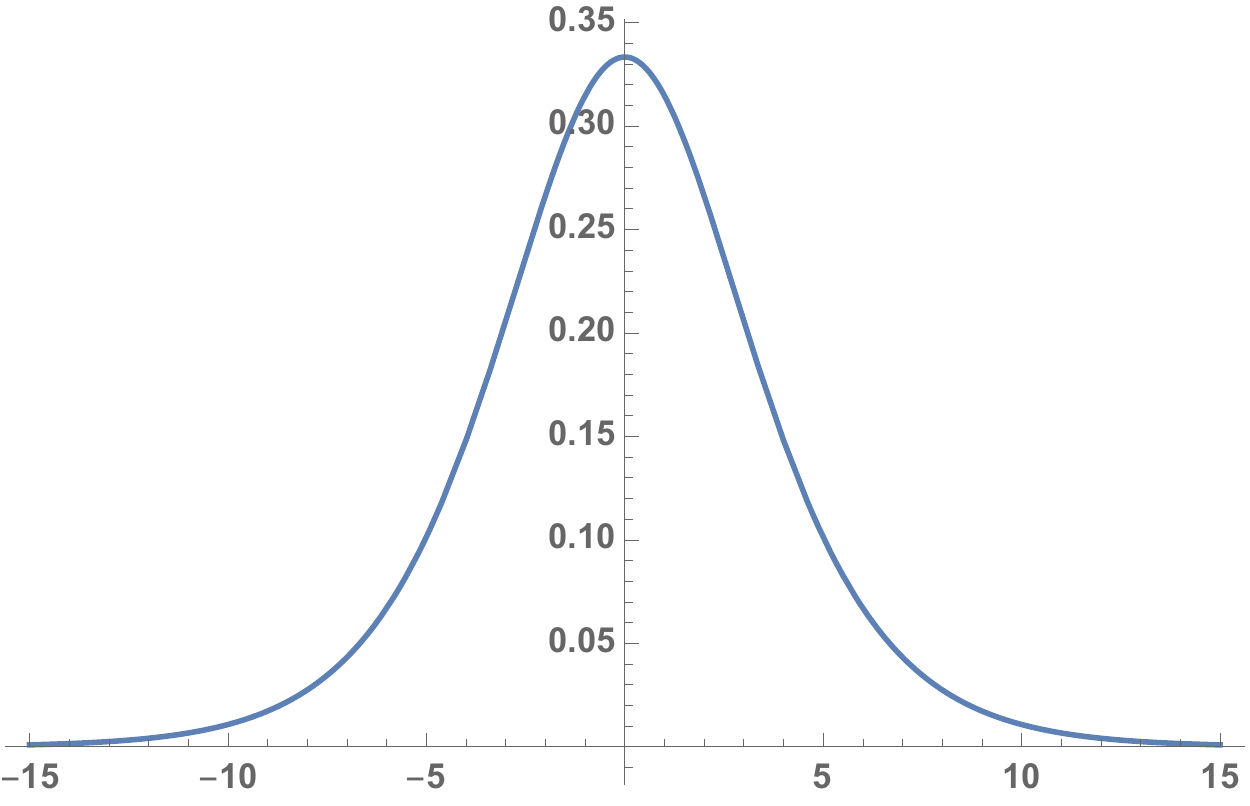}
\caption{Graph of $K$ given by \eqref{1varsc2}.}
\label{Kfor2torus}
\end{figure}

\begin{theorem}
The scalar curvature $a_2(\triangle_{\tau, h}) \in \sncttwo$ of a general metric 
in the conformal class associated with a complex number $\tau = \tau_1 + i \tau_2$ 
in the upper half-plane is given by 

\begin{eqnarray}\label{localexp}
&&a_2(\triangle_{\tau, h}) 
=K(\nabla) \big (\delta_1^2(\frac{h}{2}) + |\tau|^2 \delta_2^2(\frac{h}{2}) +2 \tau_1 \delta_1\delta_2(\frac{h}{2}) \big )  \nonumber \\
&& \, \, +  H(\nabla, \nabla) \big ( \delta_1(\frac{h}{2}) \delta_1(\frac{h}{2}) + |\tau|^2 \delta_2(\frac{h}{2}) \delta_2(\frac{h}{2}) +  \tau_1 \delta_1(\frac{h}{2}) \delta_2(\frac{h}{2}) + \tau_1 \delta_2(\frac{h}{2}) \delta_1(\frac{h}{2}) \big), \nonumber 
\end{eqnarray}
where
\begin{equation} \label{1varsc2}
K(x) = \frac{2 e^{x/2} (2 + e^x (-2 + x) + x)}{(-1 + e^x)^2 x} ,
\end{equation}
and
\begin{eqnarray} \label{2varsc2}
H(s, t) = 
\end{eqnarray}
\begin{eqnarray} \label{H}
-\frac{-t (s + t) \cosh{s} +  s (s + t) \cosh{t} - (s - t) (s + t + \sinh{s} + \sinh{t} - \sinh(s + t))}{s t (s + t)\sinh(s/2) \sinh(t/2) \sinh^2 ((s + t)/2) }.\nonumber 
\end{eqnarray}
Here the flat metric is conformally perturbed by $e^{-h}$, where $h=h^* \in \sncttwo$, and $\nabla$ 
is the logarithm of the modular automorphism $\Delta(a) = e^{-h} a e^{h},$ hence the derivation given by taking 
commutator with $-h$. 
\end{theorem}

\begin{figure}[H]
\includegraphics[scale=0.6]{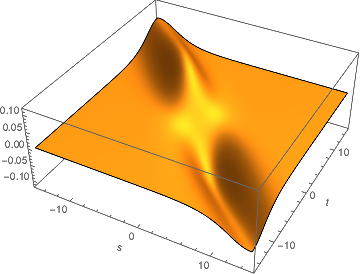}
\caption{Graph of $H$ given by \eqref{2varsc2}.}
\label{Hfor2torus}
\end{figure}

Using the symmetries of these functions describing the term $a_2(\triangle_{\tau, h})$ 
integrates to 0, 
hence the analog of the Gauss-Bonnet theorem. This result was proved in 
\cite{MR2907006, MR2956317} in a kind of simpler manner as by exploiting the trace property of 
$\varphi_0$ from the beginning of the symbolic calculations, only a one variable 
function was necessary to describe $\varphi_0(a_2 (\triangle_{\tau, h}))$. However, for the description 
of $a_2$ one needs both one and two variable functions, which are given by \eqref{1varsc2} and \eqref{2varsc2}. So we can state the Gauss-Bonnet theorem for $\ncttwo$ from \cite{MR2907006, MR2956317} as follows. 

\begin{theorem}
For any choice of the complex number $\tau$ in the upper half-plane and 
any conformal factor $e^{-h}$, where $h=h^* \in \sncttwo$, one has 
\[
\varphi_0(a_2(\triangle_{\tau, h})) = 0. 
\]
Hence the total curvature of $\ncttwo$ is independent of $\tau$ and $h$ defining the metric. 
\end{theorem}

As we mentioned earlier, the validity of the Gauss-Bonnet theorem for $\ncttwo$ was suggested 
by developments on the spectral action in the presence of a dilaton \cite{MR1463819} and also studies on 
twisted spectral triples \cite{MR2427588}. In harmony with these developments, in fact a 
non-computational proof of the Gauss-Bonnet theorem can be given, as written in \cite{MR3194491}, 
in the spirit of conformal invariance of the value at the origin of the spectral zeta function of 
conformally 
covariant operators \cite{MR869104}. The argument is based on a variational technique: 
one can write a formula for the variation of the heat coefficients  
as one varies the metric conformally with $e^{-sh}$, where $h$ is a dilaton, 
and the real parameter $s$ goes from 0 to 1.
However, the non-computational proof does not 
lead to an explicit formula for the curvature term $a_2(\triangle_{\tau, h})$.   
Hence the remarkable achievements in \cite{MR2907006,MR2956317, MR3194491, MR3148618} after heavy computer aided 
calculations include the explicit expression for the scalar curvature of 
$\ncttwo$ and the fact that the analog of the Gauss-Bonnet theorem holds for it.

\subsection{The Laplacian on $(1, 0)$-forms on $\ncttwo$ with curved metric}
The analog of the Laplacian on $(1, 0)$-forms is also considered in \cite{MR3194491, MR3148618} and the 
second term in its small time heat kernel expansion is calculated. The operator is 
anti-unitarily equivalent to the operator 
 $\triangle_{\tau, h}^{(1,0)} = \bar \partial e^h \partial$, where 
 $\partial = \delta_1 + \bar \tau \delta_2$ and $\bar \partial = \delta_1 +  \tau \delta_2$. 
 The symbol of this Laplacian is equal to  
 $c_2(\xi) + c_1(\xi)$ where 
\begin{eqnarray*}
c_2(\xi) &= & \xi_1^2 k^2 + 2 \tau_1 \xi_1 \xi_2 k^2 + |\tau|^2 \xi_2^2 k^2, \\
c_1(\xi) &=&  (\delta_1(k^2)  + \tau \delta_2(k^2) )\xi_1 + (\bar \tau \delta_1(k^2)  + |\tau|^2 \delta_2(k^2) )\xi_2.
 \end{eqnarray*}
Therefore by using the same strategy of using computer aided symbol calculations one 
can calculate the terms appearing in the following heat kernel expansion: 
\[
\Tr \left (a e^{-t \triangle_{\tau, h}^{(1, 0)}} \right ) 
\sim 
t^{-1} \sum_{n=0}^\infty \varphi_0 \left (a \, a_{2n}(\triangle_{\tau, h}^{(1,0)}) \right ) \, t^n, 
\qquad a \in \sncttwo. 
\]

The result for the second term  in this expansion is that \cite{MR3194491, MR3148618}
\begin{eqnarray*}\label{localexp}
&& a_{2}(\triangle_{\tau, h}^{(1,0)}) = S(\nabla) \big (\delta_1^2(\frac{h}{2}) + |\tau|^2 \delta_2^2(\frac{h}{2}) +2 \tau_1 \delta_1\delta_2(\frac{h}{2}) \big )   \\
&& \quad + T(\nabla, \nabla) \big ( \delta_1(\frac{h}{2}) \delta_1(\frac{h}{2}) + |\tau|^2 \delta_2(\frac{h}{2}) \delta_2(\frac{h}{2}) +  \tau_1 \delta_1(\frac{h}{2}) \delta_2(\frac{h}{2}) + \tau_1 \delta_2(\frac{h}{2}) \delta_1(\frac{h}{2}) \big)  \\
&& \quad - i \tau_2 W(\nabla, \nabla) \big (  \delta_1(\frac{h}{2}) \delta_2 (\frac{h}{2}) - \delta_2(\frac{h}{2}) \delta_1(\frac{h}{2})   \big ), 
\end{eqnarray*}
where
\begin{equation} \label{S}
S(x)  = -\frac{4 e^x (-x + \sinh x)}{(-1 + e^{x/2})^2 (1 + e^{x/2})^2 x}, \nonumber
\end{equation}

\begin{eqnarray} \label{T}
T(s, t)  = - \cosh ((s+t)/2) \times \qquad \qquad \qquad \qquad \qquad \qquad \qquad \qquad \qquad \qquad \quad \quad \quad \quad \nonumber\\
 \frac{-t (s + t) \cosh{s} +  s (s + t) \cosh{t} - (s - t) (s + t + \sinh{s} + \sinh{t} - \sinh(s + t))}{s t (s + t)\sinh(s/2) \sinh(t/2) \sinh^2 ((s + t)/2) }, \nonumber 
\end{eqnarray}
and
\begin{eqnarray} \label{W}
W(s,t)
&=& \frac{-s - t + t \cosh s + s \cosh t + \sinh s + \sinh t - \sinh (s + t)}{st\sinh (s/2) \sinh (t/2) \sinh ((s + t)/2)}. \nonumber 
\end{eqnarray}

\begin{figure}[H]
\includegraphics[scale=0.6]{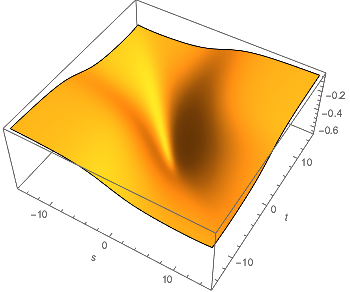}
\caption{Graph of $W$.}
\label{Wfor2torus}
\end{figure}

Using a simple iso-spectrality argument for the operators 
$\triangle_{\tau, h}$ and $\triangle_{\tau, h}^{(1, 0)}$ one can 
argue that $\varphi_0 \left (a_{2}(\triangle_{\tau, h}^{(1,0)}) \right ) =0 $, based 
on the Gauss-Bonnet theorem proved in \cite{MR2907006,MR2956317}. 
However, one can also use 
properties of the functions $S, T, W$ to prove this directly.

\section{{\bf Noncommutative residues for noncommutative tori 
and curvature of  noncommutative 4-tori}} 

In this section we discuss noncommutative residues and illustrate an 
application of a noncommutative residue defined for noncommutative 
tori in calculating the scalar curvature of the noncommutative 4-torus 
in a convenient way with certain advantages. 

\subsection{Noncommutative residues}
Noncommutative residues  are trace functionals on algebras of 
pseudodifferential operators, which were first discovered by 
Adler and Manin in dimension $1$ \cite{MR520927, MR0501136}. In order to illustrate their construction 
in dimension $1$ we consider the algebra 
$C^\infty(\mathbb{S}^1)$ of smooth functions on the circle  
$\mathbb{S}^1 = \mathbb{R}/\mathbb{Z}$, and the differentiation 
$(-i) d/dx$, whose pseudodifferential symbol is $\sigma(\xi) = \xi$. 
We then consider the algebra of pseudodifferential symbols 
of the form 
\[
\sum_{n=-\infty}^{N} a_n(x)  \xi^n, \qquad a_n(x) \in C^\infty(\mathbb{S}^1), 
\qquad N \in \mathbb{Z}.   
\]
The product rule of this algebra can be deduced from the following relations: 
\[
\xi a(x) = a(x) \xi + a'(x), \qquad a_n(x) \in C^\infty(\mathbb{S}^1), 
\]
which are dictated by the Leibniz property of differentiation. The Adler-Manin trace 
is the linear functional defined by 
\[
\sum_{n=-\infty}^{N} a_n(x)  \xi^n \mapsto \int_{\mathbb{S}^1} a_{-1}(x) \, dx, 
\]
which is shown to be a trace functional on the algebra of pseudodifferential 
symbols on the circle \cite{MR520927, MR0501136}. 
A twisted version of this trace was worked out in \cite{MR2720254}, 
motivated by the notion of twisted spectral triples \cite{MR2427588}.  

Wodzicki generalized this functional, in a remarkable work, to higher dimensions 
\cite{MR923140}. 
Consider a closed manifold $M$ of dimension $m$ and the algebra of classical 
pseudodifferential operators $M$. A classical pseudodifferential symbol $\sigma$ 
of order $d$ has an expansion with homogeneous terms, of the form 
\[
 \sigma(x, \xi) \sim \sum_{j=0}^\infty \sigma_{d-j} (x, \xi), 
 \]
where $\sigma_{d-j}(x, t \xi) = t^{d-j} \sigma_{d-j}(x, \xi)$ for any $t >0$. The composition 
rule of this algebra is induced by the composition rule for the symbol of 
pseudodifferential operators: 
\[
 \sigma_{P_1 P_2}(x, \xi) 
 \sim 
 \sum_{\alpha \in \mathbb{Z}_{\geq 0}^m}
 \frac{(-i)^{|\alpha|}}{\alpha!} \partial_\xi^a  
 \sigma_{P_1}(x, \xi) \, \partial_x^\alpha \sigma_{P_2}(x, \xi),  
 \]
 which we mentioned and used in \S \ref{pseudoandheatexpsec} as well. 
Wodzicki's noncommutative residue WRes is the linear functional defined on the 
algebra of classical pseudodifferential symbols by 
\begin{equation} \label{WodResformula}
\textrm{WRes} \left ( \sum_{j=0}^\infty \sigma_{d-j} (x, \xi) \right ) = 
\int_{S^*M} \textnormal{tr}(\sigma_{-m}(x, \xi)) \, 
d^{m-1}\xi \, d^mx, 
\end{equation}
where $S^*M$ is the cosphere bundle of the manifold with respect to a 
Riemannian metric. We stress that in this formula $m$ is the dimension of the manifold 
$M$. It is proved that WRes is the unique trace functional on the algebra of 
classical pseudodifferential symbols on $M$ \cite{MR923140}. 

The noncommutative residue has a spectral formulation as well. 
That is, one can fix a Laplacian $\triangle$ on $M$ and define the 
noncommutative residue of a pseudodifferential operator $P_\sigma$ 
to be the residue at $s=0$ of the meromorphic extension of the 
zeta function defined, for complex numbers $s$ with large enough 
real parts, by 
\[
s \mapsto \Tr(P_\sigma \triangle^{-s}).
\] 
This formulation is 
used in noncommutative geometry, when one works with the 
algebra of pseudodifferential operators associated with a 
spectral triple \cite{MR1334867}.

For noncommutative tori, the analog of formula \eqref{WodResformula} can be written and it was 
shown in \cite{MR2831659} that it gives the unique {\it continuous} trace functional on the 
algebra of classical pseudodifferential operators on the noncommutative 2-torus. 
Although the argument  written in \cite{MR2831659} is for dimension 2, but it is general enough 
that works for any dimension, see for example \cite{MR3359018} for the illustration in dimension 4. 
Given a classical pseudodifferential symbol $\rho: \R^m \to \snctm$
of order $d$ on the noncommutative $m$-torus, by definition, there is an 
asymptotic expansion for $\xi \to \infty$ of the form 
\[
\rho (\xi) \sim \sum_{j=0}^\infty \rho_{d-j}(\xi), 
\]
where each $\rho_{d-j}$ is positively homogeneous of order $d-j$. One can 
define the noncommutative residue Res of the corresponding pseudodifferential 
symbol as 
\begin{equation} \label{nctorusncresidue}
\textrm{Res}(P_\rho) = \int_{\mathbb{S}^{m-1}} 
\varphi_0 \left ( \rho_{-m} \right )   \, d\Omega, 
\end{equation}
where $\varphi_0$ is the canonical trace on $\anctm$ and $d\Omega$ 
is the volume form of the round metric on the $(m-1)$-dimensional sphere in $\R^m$. 
The same argument as the one given in \cite{MR2831659} shows that Res is the 
unique continuous trace on the algebra of classical pseudodifferential 
symbols on $\nctm$.

\subsection{Scalar curvature of the noncommutative 4-torus}
The Laplacian associated with the flat geometry of the noncommutative four 
torus $\nctfour$ is simply given by the sum of the squares of the canonical 
derivatives, namely: 
\[
\triangle_0 = \delta_1^2 + \delta_2^2 +\delta_3^2 +\delta_4^2. 
\]
After conformally perturbing the flat metric on $\nctfour$ by means of a conformal 
factor $e^{-h}$, for a fixed $h=h^* \in \snctfour$, the perturbed Laplacian is shown 
in \cite{MR3359018} to be anti-unitarily equivalent to the operator 
\[
\triangle_h= 
e^{h} \bar \partial_1 e^{-h} \partial_1 e^{h} + 
e^{h} \partial_1 e^{-h} \bar \partial_1 e^{h} + 
e^{h} \bar \partial_2 e^{-h} \partial_2 e^{h} + 
e^{h}  \partial_2 e^{-h} \bar \partial_2 e^{h}, 
\]
where
\[ 
\partial_1 =  \delta_1 - i \delta_3, \qquad 
\partial_2=  \delta_2 - i \delta_4,
\]
\[ \bar \partial_1 =  \delta_1 + i \delta_3, 
\qquad \bar \partial_2= \delta_2 + i \delta_4.
\]
The latter are the analogues of the Dolbeault operators.  

The scalar curvature of the metric on $\nctfour$ encoded in $\triangle_h$ 
is the term $a_2 \in \snctfour$ appearing in the following small time 
asymptotic expansion: 
\[
\Tr(a e^{-t \triangle_h}) \sim t^{-2} \sum_{n=0}^{\infty} 
\varphi_0(a \,a_{2n})
t^n, \qquad a \in \snctfour. 
\]
The curvature term $a_2 \in \snctfour$ was calculated in \cite{MR3359018} by going through the procedure 
explained in \S \ref{nctorusheatexpsubsec}. As we explained earlier, there is a purely noncommutative obstruction 
in this procedure that needs to be overcome by Lemma \ref{rearrlemma}, the so-called rearrangement 
lemma. That is, one encounters integration over the Euclidean space of $C^*$-algebra valued 
functions.  For this type of integrations, one can pass to polar coordinates and take care of 
the angular integrations with no problem. However, the redial integration brings forth 
the necessity of the rearrangement lemma. 

Striking is the fact that after applying the 
rearrangement lemma to hundreds of terms, each of which involves a function from 
this lemma to appear in the calculations,  the final formula for the curvature 
simplifies significantly with computer aid. In \cite{MR3369894}, by using properties of the noncommutative 
residue \eqref{nctorusncresidue}, it was shown that the curvature term $a_2 \in \snctfour$ can be calculated 
as the integral over the 3-sphere of a homogeneous symbol. Therefore, with this method, the 
calculation of $a_2$ does not require radial integration, hence the calculation without 
using the rearrangement lemma and clarification of the reason for the significant simplifications. 
In fact, in \cite{MR3369894}, the term is shown to be a scalar multiple of 
$\int_{\mathbb{S}^3} b_2(\xi)  \, d\Omega$, where $b_2$ is the homogeneous term  of order 
$-4$ in the expansion of the symbol of the parametrix of $\triangle_h$. The result, in agreement 
with the calculation of \cite{MR3359018}, is that
\begin{eqnarray}  \label{nc4sc}
a_2 = e^{-h} K(\nabla) \Big (\sum_{i=1}^4 \delta_i^2(h) \Big ) +e^{-h} 
H(\nabla, \nabla)\Big (\sum_{i=1}^4 \delta_i(h)^2 \Big ) \in \snctfour, 
\end{eqnarray}
where $\nabla = [-h, \cdot]$, and 
\begin{eqnarray}  \label{nc4scfunctions}
K(x)&=& \frac{1-e^{-x}}{2 x}, \nonumber \\
H(s,t)&=&-\frac{e^{-s-t} \left(\left(-e^s-3\right) s \left(e^t-1\right)+
\left(e^s-1\right) \left(3 e^t+1\right) t\right)}{4 s t (s+t)}.  
\end{eqnarray}

The simplicity of this calculation also revealed in \cite{MR3369894} the following 
functional relation between the functions $K$ and $H$. 

\begin{figure}[H]
\includegraphics[scale=0.6]{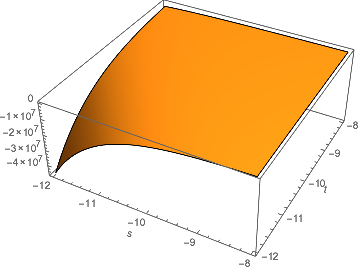}
\caption{Graph of $H$ given by \eqref{nc4scfunctions}.}
\label{Hfor4torus}
\end{figure}

\begin{theorem}
Let $\tilde K(s)= e^s K(s)$ and $\tilde H(s, t) = e^{s+t} H(s, t)$, where the 
function $K$ and $H$ are given by \eqref{nc4scfunctions}. Then 
\[
\tilde H(s, t) = 2 \frac{\tilde K(s+t) - \tilde K(s)}{t} + \frac{3}{2} \tilde K(s) \tilde K(t). 
\]
\end{theorem}

Another important result that we wish to recall from \cite{MR3359018}  is about 
the extrema of the analog of the Einstein-Hilbert action for $\nctfour$, namely 
$\varphi_0(a_2)$: 

\begin{theorem}
For any conformal factor $e^{-h}$, where $h=h^* \in \snctfour$, 
\[
\varphi_0(a_2) \leq 0,  
\]
where $a_2 \in \snctfour$ is the scalar curvature given by \eqref{nc4sc}. 
Moreover, we have $\varphi_0(a_2) =0$ if and only if $h$ is a scalar. 
\end{theorem}

\section{ {\bf The Riemann curvature tensor and the term $a_4$ for noncommutative tori}}

The Riemann curvature tensor appears in the term $a_4$ in the heat kernel expansion 
for the Laplacian of any closed Riemannian manifold $M$. That is, if $\Delta_g$ is the Laplacian 
of a Riemannian metric $g$, which acts on $C^\infty(M)$, then 
\[
a_4(x, \Delta_g) = (4\pi)^{-1} (1/360) (-12 \Delta_g R(x) +5R(x)^2 - 2|Ric(x)|^2 + 2|Riem(x)|^2).  
\]

In this section we recall from \cite{arXiv1611.09815} the formula obtained for the analog of the term 
$a_4$ in a noncommutative setting. Recall that in \S \ref{SCandGB2subsec}, we discussed the term $a_2$, 
namely the analog of the scalar curvature, for the noncommutative two torus when the 
flat metric is perturbed by a positive invertible element $e^{-h} \in \sncttwo$, where 
$h=h^*$. These geometric terms appear in the expansion given by \eqref{heatexpnc2ttauh}.  Setting, 
\[
\ell = \frac{h}{2}
\]
for the simplest conformal class (associated with $\tau=i$), the main calculation of 
\cite{arXiv1611.09815} gives the term $a_4$ by a formula of the following form: 
\begin{equation}
\label{a4closedformula}
a_4(h)=
\end{equation}
\begin{center}
$
- e^{2 \ell} \Big ( K_1(\nabla) \left ( \delta _1^2 \delta _2^2 ( \ell ) \right ) 
+ 
K_2 (\nabla) \left (   \delta _1^4( \ell )+\delta _2^4 (\ell ) \right )  
+
K_3 (\nabla, \nabla) \left (
 \left(\delta _1 \delta _2(\ell
   )\right) \cdot \left(\delta _1 \delta _2(\ell
   )\right) 
   \right )
+ 
K_4 (\nabla, \nabla) \left (  \delta _1^2(\ell ) \cdot \delta _2^2(\ell )+\delta
   _2^2(\ell ) \cdot \delta _1^2(\ell )\right )  
+
K_5 (\nabla, \nabla) \left ( \delta _1^2( \ell )\cdot \delta _1^2(\ell )+\delta
   _2^2(\ell ) \cdot \delta _2^2(\ell ) \right ) 
+ 
K_6 (\nabla, \nabla) \left ( 
\delta _1(\ell )\cdot \delta _1^3(\ell )+\delta_1(\ell )\cdot \left(\delta _1 \delta _2^2 (\ell
   )\right)+\delta _2(\ell )\cdot \delta _2^3(\ell
   )+\delta _2(\ell )\cdot \left(\delta _1^2
   \delta _2(\ell )\right)
 \right ) 
+ 
K_7 (\nabla, \nabla) \left ( 
\delta _1^3(\ell )\cdot \delta _1(\ell
   )+\left(\delta _1 \delta _2^2(\ell
   )\right)\cdot \delta _1(\ell )+\delta _2^3(\ell
   )\cdot \delta _2(\ell )+\left(\delta _1^2
   \delta _2(\ell )\right)\cdot \delta _2(\ell )
 \right ) 
+ 
K_8 (\nabla, \nabla, \nabla) \left ( 
\delta _1(\ell ) \cdot \delta _1(\ell )\cdot \delta
   _2^2(\ell )+\delta _2(\ell )\cdot \delta
   _2(\ell )\cdot \delta _1^2(\ell )
\right ) 
+ 
K_9 (\nabla, \nabla, \nabla) \left ( 
\delta _1(\ell )\cdot \delta _2(\ell
   )\cdot \left(\delta _1 \delta _2(\ell
   )\right)+\delta _2(\ell )\cdot \delta _1(\ell
   )\cdot \left(\delta _1 \delta _2(\ell )\right)
\right )  
+ 
K_{10} (\nabla, \nabla, \nabla) \left ( 
\delta _1(\ell )\cdot \left(\delta _1 \delta
   _2(\ell )\right)\cdot \delta _2(\ell )+\delta
   _2(\ell )\cdot  \left(\delta _1 \delta _2(\ell
   )\right) \cdot  \delta _1(\ell )
\right ) 
+ 
K_{11} (\nabla, \nabla, \nabla) \left ( 
\delta _1(\ell )\cdot \delta _2^2(\ell ) \cdot \delta
   _1(\ell )+\delta _2(\ell )\cdot \delta
   _1^2(\ell ) \cdot \delta _2(\ell )
\right ) 
+ 
K_{12} (\nabla, \nabla, \nabla) \left ( 
\delta _1^2(\ell ) \cdot \delta _2(\ell )\cdot \delta
   _2(\ell )+\delta _2^2( \ell ) \cdot \delta
   _1(\ell )\cdot \delta _1(\ell )
\right ) 
+ 
K_{13} (\nabla, \nabla, \nabla) \left ( 
\left(\delta _1 \delta _2(\ell
   )\right)\cdot \delta _1(\ell )\cdot \delta _2(\ell
   )+\left(\delta _1 \delta _2(\ell
   )\right) \cdot \delta _2(\ell )\cdot \delta _1(\ell )
\right )  
+ 
K_{14} (\nabla, \nabla, \nabla) \left ( 
\delta _1^2(\ell ) \cdot \delta _1(\ell )\cdot 
\delta_1(\ell )+\delta _2^2(\ell )\cdot \delta_2(\ell )\cdot \delta _2(\ell )
\right ) 
+ 
K_{15} (\nabla, \nabla, \nabla) \left ( 
\delta _1(\ell ) \cdot \delta _1(\ell )\cdot \delta
   _1^2(\ell )+\delta _2(\ell )\cdot \delta
   _2(\ell )\cdot \delta _2^2(\ell )
\right ) 
+ 
K_{16} (\nabla, \nabla, \nabla) \left ( 
\delta _1(\ell )\cdot \delta _1^2(\ell ) \cdot \delta
   _1(\ell )+\delta _2(\ell )\cdot \delta
   _2^2(\ell )\cdot \delta _2(\ell )
\right ) 
+ 
K_{17} (\nabla, \nabla, \nabla, \nabla ) \left ( 
\delta _1(\ell )\cdot \delta _1(\ell )\cdot \delta
   _2(\ell )\cdot \delta _2(\ell )+\delta _2(\ell
   )\cdot \delta _2(\ell )\cdot \delta _1(\ell )\cdot \delta
   _1(\ell )
\right ) 
+ 
K_{18} (\nabla, \nabla, \nabla, \nabla ) \left ( 
\delta _1(\ell )\cdot \delta _2(\ell )\cdot \delta
   _1(\ell )\cdot \delta _2(\ell )+\delta _2(\ell
   )\cdot \delta _1(\ell )\cdot \delta _2(\ell )\cdot \delta
   _1(\ell )
\right ) 
+ 
K_{19} (\nabla, \nabla, \nabla, \nabla ) \left ( 
\delta _1(\ell )\cdot \delta _2(\ell )\cdot \delta
   _2(\ell )\cdot \delta _1(\ell )+\delta _2(\ell
   )\cdot \delta _1(\ell )\cdot \delta _1(\ell )\cdot \delta
   _2(\ell )
\right ) 
+ 
K_{20} (\nabla, \nabla, \nabla, \nabla ) \left ( 
\delta _1(\ell )\cdot \delta _1(\ell )\cdot \delta
   _1(\ell )\cdot \delta _1(\ell )+\delta _2(\ell
   )\cdot \delta _2(\ell )\cdot \delta _2(\ell )\cdot \delta
   _2(\ell )
\right ) \Big ). 
$
\end{center}

We provide the explicit formulas for a few of the functions appearing in 
\eqref{a4closedformula}, and we refer the reader to \cite{arXiv1611.09815} 
for the remaining functions, most of which have lengthy expressions.  
We have for example: 

\begin{figure}[H]
\includegraphics[scale=0.6]{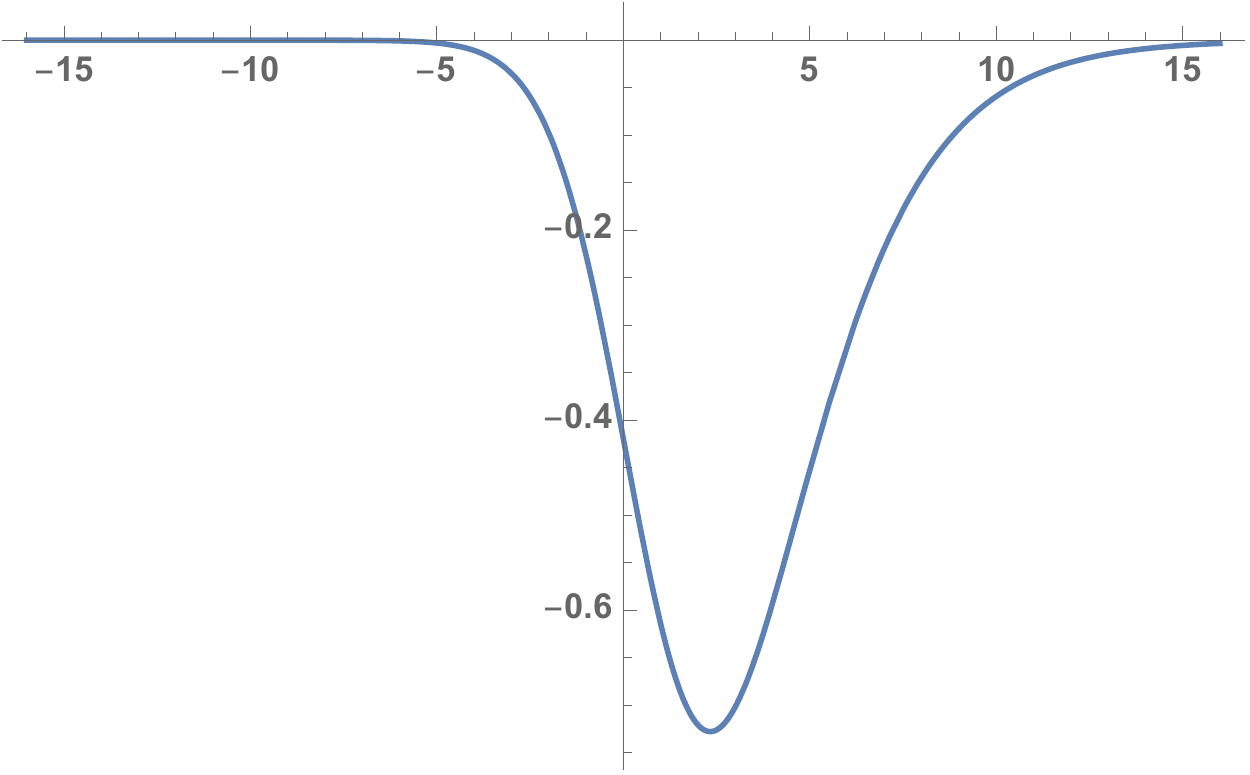}
\caption{Graph of $K_1$.}
\label{K1graph}
\end{figure}

\begin{equation} \label{K_1explicitformula}
K_1(s_1)= -\frac{4 \pi  e^{\frac{3 s_1}{2}}
   \left(\left(4 e^{s_1}+e^{2 s_1}+1\right)
   s_1-3 e^{2
   s_1}+3\right)}{\left(e^{s_1}-1\right){}^4
   s_1}, 
\end{equation}
and 
\begin{equation} \label{K_3explicitformula}
K_3(s_1, s_2) = \frac{K_3^{\text{num}}(s_1, s_2)}{
\left(e^{s_1}-1\right){}^2 \left(e^{s_2}-1\right){}^2
   \left(e^{s_1+s_2}-1\right){}^4 s_1 s_2 \left(s_1+s_2\right)}, 
   \end{equation}
where the numerator  is given by 
\begin{eqnarray*}
  K_3^{\text{num}}(s_1, s_2)&=&
16 \, e^{\frac{3 s_1}{2}+\frac{3 s_2}{2}} \pi \Big  [\left(e^{s_1}-1\right) \left(e^{s_2}-1\right) \left(e^{s_1+s_2}-1\right)  \big \{ 
\\
&&
\left(-5 e^{s_1}-e^{s_2}+6 e^{s_1+s_2}-e^{2 s_1+s_2}-5 e^{s_1+2 s_2}+3 e^{2 s_1+2 s_2}+3\right) s_1+
\\
&&
\left(e^{s_1}+5 e^{s_2}-6 e^{s_1+s_2}+5 e^{2 s_1+s_2}+e^{s_1+2 s_2}-3 e^{2 s_1+2 s_2}-3\right)s_2 \big \}
\\
&&
  -2 \left(e^{s_1}-e^{s_2}\right) \left(e^{s_1+s_2}-1\right)
\\
&&
\left(-e^{s_1}-e^{s_2}-e^{2 s_1+s_2}-e^{s_1+2 s_2}+2 e^{2 s_1+2 s_2}+2\right) s_1 s_2
\\
&&
+2 e^{s_1} \left(e^{s_2}-1\right){}^3 \left(e^{s_1}-e^{s_1+s_2}+2 e^{2 s_1+s_2}-2\right) s_1^2 
\\
&&
-2 e^{s_2} \left(e^{s_1}-1\right){}^3\left(e^{s_2}-e^{s_1+s_2}+2 e^{s_1+2 s_2}-2\right) s_2^2 \Big ].
\end{eqnarray*}

\begin{figure}[H]
\includegraphics[scale=0.6]{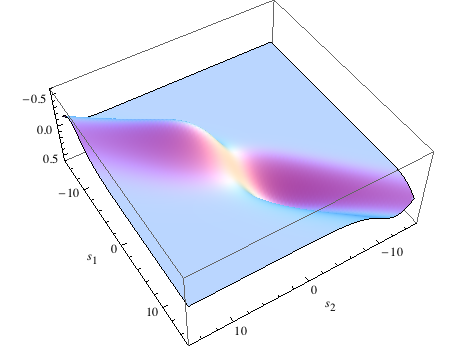}
\caption{Graph of $K_3$.}
\label{K3graph}
\end{figure}

\subsection{Functional relations}
\label{functionalrelatioinssubsec}
One of the main results of \cite{arXiv1611.09815} is the derivation of a family of 
conceptually predicted functional relations among the functions 
$K_1,$ $ \dots,$ $ K_{20}$ appearing in \eqref{a4closedformula}. As we shall see shortly 
the functional relations are highly nontrivial. There are two main reasons 
for the derivation of the relations, both of which are extremely important. 
First, the calculation of the term $a_4$ involves a really heavy computer 
aided calculation, hence, the need for a way of confirming the validity 
of the outcome by checking that the expected functional relations are 
satisfied. Second, patterns and structural properties of the relations 
give significant information that help one to obtain conceptual understandings 
about the structure of the complicated functions appearing in the formula for 
$a_4$. In order to present the relations, we need to consider the modification 
of each function $K_j$ in \eqref{a4closedformula} to a new function 
$\widetilde K_j$ by the formula 
\[
\widetilde K_j(s_1, \dots, s_n) 
= 
\frac{1}{2^n} \frac{\sinh \left ( (s_1+\dots+s_n)/2 \right )}{(s_1+\dots+s_n)/2} K_j(s_1, \dots, s_n), 
\]
where $n \in \{1, 2, 3, 4 \}$ is the number of variables, on which $K_j$ depends. We also 
need to introduce the restriction of the functions $K_j$ to certain hyperplanes by 
setting
\[
k_j(s_1, \dots, s_{n-1}) = K_j(s_1,  \dots, s_{n-1}, -s_1-\cdots -s_{n-1}).  
\]
We shall explain shortly how these functional relations are predicted, using 
fundamental identities and lemmas  \cite{MR3194491, arXiv1611.09815}.

Let us first list a few of the 
functional relations in which some auxiliary functions $G_n(s_1, \dots, s_n)$ 
appear. These functions are mainly useful for relating the derivatives of $e^h$ 
and those of $h$ and we recall from  \cite{arXiv1611.09815} their recursive formula:  
\begin{lemma} \label{BabyFunctions1} The functions $G_n(s_1, \dots, s_n)$ 
are given recursively by 
\[
G_0=1, 
\]
and 
\[
G_n(s_1, \dots, s_n) 
= 
\int_{0}^1 r^{n-1} e^{s_1 r} \,G_{n-1} (r s_2, r s_3, \dots, r s_n)  \, dr. 
\]
\end{lemma}
Explicitly, for $n=1, 2, 3$, one has: 

\begin{eqnarray}  \label{ExplicitBabyFunctions}
G_1(s_1)&=&\frac{e^{s_1}-1}{s_1}, \\ 
G_2(s_1, s_2) &=& \frac{e^{s_1} \left(\left(e^{s_2}-1\right) s_1-s_2\right)+s_2}{s_1 s_2 \left(s_1+s_2\right)}, \nonumber \\
G_3(s_1, s_2, s_3) &=& \nonumber
\end{eqnarray}
\[
\frac{e^{s_1} \left(e^{s_2+s_3} s_1 s_2 \left(s_1+s_2\right)+\left(s_1+s_2+s_3\right) \left(\left(s_1+s_2\right) s_3-e^{s_2} s_1 \left(s_2+s_3\right)\right)\right)-s_2 s_3 \left(s_2+s_3\right)}{s_1 s_2 \left(s_1+s_2\right) s_3 \left(s_2+s_3\right) \left(s_1+s_2+s_3\right)}.  
\]

We can now write the relations. 
The functional relation associated with the function $K_1$ is given by
\begin{equation} \label{basicK1eqn}
 \widetilde K_1(s_1) =
\end{equation}
\begin{center} 
\begin{math}
-\frac{1}{15} \pi  G_1\left(s_1\right)+\frac{1}{4} e^{s_1} k_3\left(-s_1\right)+\frac{1}{4} k_3\left(s_1\right)+\frac{1}{2} e^{s_1} k_4\left(-s_1\right)+\frac{1}{2} k_4\left(s_1\right)-\frac{1}{2} e^{s_1} k_6\left(-s_1\right)-\frac{1}{2} k_6\left(s_1\right)-\frac{1}{2} e^{s_1} k_7\left(-s_1\right)-\frac{1}{2} k_7\left(s_1\right)-\frac{\pi  \left(e^{s_1}-1\right)}{15 s_1}. 
\end{math}
\end{center}
It is quite remarkable that such a nontrivial relation should exist among the 
functions, and it gets even more interesting when one looks at the case associated  
 with a 2-variable function. For  $K_3$ one finds the associated 
relation to be: 
\begin{equation} \label{basicK3eqn}
\widetilde K_3(s_1, s_2) =
\end{equation}
\begin{center}
\begin{math} 
 \frac{1}{15} (-4) \pi  G_2\left(s_1,s_2\right)+\frac{1}{2} k_8\left(s_1,s_2\right)+\frac{1}{4} k_9\left(s_1,s_2\right)-\frac{1}{4} e^{s_1+s_2} k_9\left(-s_1-s_2,s_1\right)-\frac{1}{4} e^{s_1} k_9\left(s_2,-s_1-s_2\right)-\frac{1}{4} k_{10}\left(s_1,s_2\right)-\frac{1}{4} e^{s_1+s_2} k_{10}\left(-s_1-s_2,s_1\right)+\frac{1}{4} e^{s_1} k_{10}\left(s_2,-s_1-s_2\right)+\frac{1}{2} e^{s_1} k_{11}\left(s_2,-s_1-s_2\right)+\frac{1}{2} e^{s_1+s_2} k_{12}\left(-s_1-s_2,s_1\right)-\frac{1}{4} k_{13}\left(s_1,s_2\right)+\frac{1}{4} e^{s_1+s_2} k_{13}\left(-s_1-s_2,s_1\right)-\frac{1}{4} e^{s_1} k_{13}\left(s_2,-s_1-s_2\right)+\frac{1}{4} e^{s_2} G_1\left(s_1\right) k_3\left(-s_2\right)+\frac{1}{4} G_1\left(s_1\right) k_3\left(s_2\right)-G_1\left(s_1\right) k_6\left(s_2\right)-e^{s_2} G_1\left(s_1\right) k_7\left(-s_2\right)+\frac{\left(e^{s_1+s_2}-1\right) k_3\left(s_1\right)}{4 \left(s_1+s_2\right)}+\frac{k_3\left(s_2\right)-k_3\left(s_1+s_2\right)}{4 s_1}+\frac{k_3\left(s_1+s_2\right)-k_3\left(s_1\right)}{4 s_2}+\frac{k_6\left(s_1\right)-k_6\left(s_1+s_2\right)}{s_2}+\frac{k_6\left(s_1+s_2\right)-k_6\left(s_2\right)}{s_1}+\frac{e^{s_1} \left(k_7\left(-s_1\right)-e^{s_2} k_7\left(-s_1-s_2\right)\right)}{s_2}+\frac{e^{s_2} \left(e^{s_1} k_7\left(-s_1-s_2\right)-k_7\left(-s_2\right)\right)}{s_1}-\frac{e^{s_2} \left(e^{s_1} k_3\left(-s_1-s_2\right)-k_3\left(-s_2\right)\right)}{4 s_1}-\frac{e^{s_1} \left(k_3\left(-s_1\right)-e^{s_2} k_3\left(-s_1-s_2\right)\right)}{4 s_2}-\frac{e^{s_1} \left(k_3\left(-s_1\right)+e^{s_2} k_3\left(s_1\right)-e^{s_2} k_3\left(-s_2\right)-k_3\left(s_2\right)\right)}{4 \left(s_1+s_2\right)}.  
\end{math}
\end{center}

The rapid pace in growing complexity of the functional relations can be seen in the higher 
variable cases as for example the functional relation corresponding to the 3-variable 
function $K_8$ is the following expression: 
\begin{equation} \label{basicK8eqn}
\widetilde K_8(s_1, s_2, s_3) =
\end{equation}
\begin{center}
\begin{math}
 \frac{1}{15} (-2) \pi  G_3\left(s_1,s_2,s_3\right)+\frac{1}{2} e^{s_3} G_2\left(s_1,s_2\right) k_4\left(-s_3\right)-\frac{e^{s_3} \left(e^{s_2} s_1 k_4\left(-s_2-s_3\right)+e^{s_2} s_2 k_4\left(-s_2-s_3\right)-e^{s_1+s_2} s_2 k_4\left(-s_1-s_2-s_3\right)-s_1 k_4\left(-s_3\right)\right)}{2 s_1 s_2 \left(s_1+s_2\right)}+\frac{1}{2} G_2\left(s_1,s_2\right) k_4\left(s_3\right)+\frac{G_1\left(s_1\right) \left(k_4\left(s_3\right)-k_4\left(s_2+s_3\right)\right)}{2 s_2}+\frac{s_1 k_4\left(s_3\right)-s_1 k_4\left(s_2+s_3\right)-s_2 k_4\left(s_2+s_3\right)+s_2 k_4\left(s_1+s_2+s_3\right)}{2 s_1 s_2 \left(s_1+s_2\right)}-\frac{1}{2} G_2\left(s_1,s_2\right) k_6\left(s_3\right)+\frac{G_1\left(s_1\right) \left(k_6\left(s_2\right)-k_6\left(s_2+s_3\right)\right)}{4 s_3}+\frac{k_6\left(s_2\right)-k_6\left(s_1+s_2\right)-k_6\left(s_2+s_3\right)+k_6\left(s_1+s_2+s_3\right)}{4 s_1 s_3}+\frac{-s_3 k_6\left(s_1\right)+s_2 k_6\left(s_1+s_2\right)+s_3 k_6\left(s_1+s_2\right)-s_2 k_6\left(s_1+s_2+s_3\right)}{4 s_2 s_3 \left(s_2+s_3\right)}+\frac{-s_1 k_6\left(s_3\right)+s_1 k_6\left(s_2+s_3\right)+s_2 k_6\left(s_2+s_3\right)-s_2 k_6\left(s_1+s_2+s_3\right)}{2 s_1 s_2 \left(s_1+s_2\right)}+\frac{e^{s_2} G_1\left(s_1\right) \left(k_7\left(-s_2\right)-e^{s_3} k_7\left(-s_2-s_3\right)\right)}{4 s_3}-\frac{e^{s_1} \left(s_3 k_7\left(-s_1\right)-e^{s_2} s_2 k_7\left(-s_1-s_2\right)-e^{s_2} s_3 k_7\left(-s_1-s_2\right)+e^{s_2+s_3} s_2 k_7\left(-s_1-s_2-s_3\right)\right)}{4 s_2 s_3 \left(s_2+s_3\right)}-\frac{e^{s_2} \left(e^{s_1} k_7\left(-s_1-s_2\right)-k_7\left(-s_2\right)+e^{s_3} k_7\left(-s_2-s_3\right)-e^{s_1+s_3} k_7\left(-s_1-s_2-s_3\right)\right)}{4 s_1 s_3}+\frac{e^{s_3} G_1\left(s_1\right) \left(e^{s_2} k_7\left(-s_2-s_3\right)-k_7\left(-s_3\right)\right)}{2 s_2}-\frac{1}{2} e^{s_3} G_2\left(s_1,s_2\right) k_7\left(-s_3\right)+\frac{e^{s_3} \left(e^{s_2} s_1 k_7\left(-s_2-s_3\right)+e^{s_2} s_2 k_7\left(-s_2-s_3\right)-e^{s_1+s_2} s_2 k_7\left(-s_1-s_2-s_3\right)-s_1 k_7\left(-s_3\right)\right)}{2 s_1 s_2 \left(s_1+s_2\right)}+\frac{\left(-1+e^{s_1+s_2+s_3}\right) k_8\left(s_1,s_2\right)}{8 \left(s_1+s_2+s_3\right)}+\frac{k_8\left(s_1,s_2+s_3\right)-k_8\left(s_1,s_2\right)}{8 s_3}-\frac{1}{8} e^{s_2+s_3} G_1\left(s_1\right) k_8\left(-s_2-s_3,s_2\right)+\frac{e^{s_1+s_2+s_3} \left(k_8\left(-s_1-s_2-s_3,s_1\right)-k_8\left(-s_1-s_2-s_3,s_1+s_2\right)\right)}{8 s_2}+\frac{1}{8} G_1\left(s_1\right) k_9\left(s_2,s_3\right)+\frac{k_9\left(s_2,s_3\right)-k_9\left(s_1+s_2,s_3\right)}{8 s_1}+\frac{k_9\left(s_1+s_2,s_3\right)-k_9\left(s_1,s_2+s_3\right)}{8 s_2}+\frac{1}{8} e^{s_2} G_1\left(s_1\right) k_{10}\left(s_3,-s_2-s_3\right)+\frac{e^{s_2} \left(k_{10}\left(s_3,-s_2-s_3\right)-e^{s_1} k_{10}\left(s_3,-s_1-s_2-s_3\right)\right)}{8 s_1}+\frac{e^{s_1} \left(e^{s_2} k_{10}\left(s_3,-s_1-s_2-s_3\right)-k_{10}\left(s_2+s_3,-s_1-s_2-s_3\right)\right)}{8 s_2}-\frac{1}{8} G_1\left(s_1\right) k_{11}\left(s_2,s_3\right)+\frac{k_{11}\left(s_1,s_2+s_3\right)-k_{11}\left(s_1+s_2,s_3\right)}{8 s_2}+\frac{k_{11}\left(s_1+s_2,s_3\right)-k_{11}\left(s_2,s_3\right)}{8 s_1}-\frac{1}{8} e^{s_2} G_1\left(s_1\right) k_{12}\left(s_3,-s_2-s_3\right)+\frac{1}{8} e^{s_2+s_3} G_1\left(s_1\right) k_{13}\left(-s_2-s_3,s_2\right)+\frac{e^{s_2+s_3} \left(k_{13}\left(-s_2-s_3,s_2\right)-e^{s_1} k_{13}\left(-s_1-s_2-s_3,s_1+s_2\right)\right)}{8 s_1}-\frac{1}{16} k_{17}\left(s_1,s_2,s_3\right)-\frac{1}{16} e^{s_1+s_2} k_{17}\left(s_3,-s_1-s_2-s_3,s_1\right)-\frac{1}{16} e^{s_1} k_{19}\left(s_2,s_3,-s_1-s_2-s_3\right)-\frac{1}{16} e^{s_1+s_2+s_3} k_{19}\left(-s_1-s_2-s_3,s_1,s_2\right)-\frac{e^{s_2+s_3} \left(k_8\left(-s_2-s_3,s_2\right)-e^{s_1} k_8\left(-s_1-s_2-s_3,s_1+s_2\right)\right)}{8 s_1}-\frac{e^{s_2} \left(k_{12}\left(s_3,-s_2-s_3\right)-e^{s_1} k_{12}\left(s_3,-s_1-s_2-s_3\right)\right)}{8 s_1}-\frac{e^{s_3} G_1\left(s_1\right) \left(e^{s_2} k_4\left(-s_2-s_3\right)-k_4\left(-s_3\right)\right)}{2 s_2}-\frac{G_1\left(s_1\right) \left(k_6\left(s_3\right)-k_6\left(s_2+s_3\right)\right)}{2 s_2}-\frac{e^{s_1} \left(e^{s_2} k_{12}\left(s_3,-s_1-s_2-s_3\right)-k_{12}\left(s_2+s_3,-s_1-s_2-s_3\right)\right)}{8 s_2}-\frac{e^{s_1+s_2+s_3} \left(k_{13}\left(-s_1-s_2-s_3,s_1\right)-k_{13}\left(-s_1-s_2-s_3,s_1+s_2\right)\right)}{8 s_2}-\frac{e^{s_1} \left(k_{11}\left(s_2,-s_1-s_2\right)-k_{11}\left(s_2+s_3,-s_1-s_2-s_3\right)\right)}{8 s_3}-\frac{e^{s_1+s_2} \left(k_{12}\left(-s_1-s_2,s_1\right)-e^{s_3} k_{12}\left(-s_1-s_2-s_3,s_1\right)\right)}{8 s_3}-\frac{e^{s_1+s_2+s_3} \left(k_8\left(s_1,s_2\right)-k_8\left(-s_2-s_3,s_2\right)\right)}{8 \left(s_1+s_2+s_3\right)}-\frac{e^{s_1} \left(k_{11}\left(s_2,-s_1-s_2\right)-k_{11}\left(s_2,s_3\right)\right)}{8 \left(s_1+s_2+s_3\right)}-\frac{e^{s_1+s_2} \left(k_{12}\left(-s_1-s_2,s_1\right)-k_{12}\left(s_3,-s_2-s_3\right)\right)}{8 \left(s_1+s_2+s_3\right)}. 
\end{math}
\end{center} 
The interested reader can refer to \cite{arXiv1611.09815} to see that the functional relations of the 
4-variable functions get even more complicated. The main point, which will be 
elaborated further, is that all these functional relations are derived conceptually, 
and by checking that our calculated functions $K_1, \dots, K_{20}$ satisfy 
these relations, the validity of the calculations and their outcome, such as 
the explicit formulas \eqref{K_1explicitformula}, \eqref{K_3explicitformula}, is confirmed.

\subsection{A partial differential system associated with the functional relations}
When one takes a close look at the functional relations, one notices that there are 
terms in the right hand sides (in the finite difference 
expressions)  with $s_1+\cdots+ s_n$ in their denominators. 
For example in \eqref{basicK3eqn} one can see that there is a term with $s_1+s_2$ in the denominator. 
The question answered in \cite{arXiv1611.09815}, which leads to a differential system with interesting 
properties, is what happens when one restricts the functional relations to the 
hyperplanes $s_1+\cdots +s_n=0$ by setting  $s_1+\cdots +s_n = \varepsilon$ and letting 
$\varepsilon \to 0$. For example the restriction of the functional relation \eqref{basicK3eqn} to the 
hyperplane $s_1+s_2 =0$ yields: 

\begin{equation} \label{diffeqK3}
\frac{1}{4} e^{s_1} k_3'\left(-s_1\right)-\frac{1}{4} k_3'\left(s_1\right) = 
\end{equation}
\begin{center}
\begin{math}
\frac{1}{60 s_1} \Big ( 16 \pi  s_1 G_2(s_1,-s_1)-30 s_1 k_8(s_1,-s_1)+15 s_1 k_9(0,s_1)+15 e^{s_1} s_1 k_9(-s_1,0)-15 s_1 k_9(s_1,-s_1)+15 s_1 k_{10}(0,s_1)-15 e^{s_1} s_1 k_{10}(-s_1,0)+15 s_1 k_{10}(s_1,-s_1)-30 e^{s_1} s_1 k_{11}(-s_1,0)-30 s_1 k_{12}(0,s_1)-15 s_1 k_{13}(0,s_1)+15 e^{s_1} s_1 k_{13}(-s_1,0)+15 s_1 k_{13}(s_1,-s_1)-15 s_1 G_1(s_1) k_3(-s_1)-15 e^{-s_1} s_1 G_1(s_1) k_3(s_1)+60 s_1 G_1(s_1) k_6(-s_1)+60 e^{-s_1} s_1 G_1(s_1) k_7(s_1)-15 e^{s_1} k_3(-s_1)-15 k_3(-s_1)-15 e^{-s_1} k_3(s_1)-15 k_3(s_1)+60 k_6(-s_1)+60 k_6(s_1)+60 e^{s_1} k_7(-s_1)+60 e^{-s_1} k_7(s_1)+60 k_3(0)-120 k_6(0)-120 k_7(0)   \Big ).  
\end{math}
\end{center}

The restriction of the functional relation \eqref{basicK8eqn} to the hyperplane 
$s_1+s_2+s_3 =0$ yields
\begin{equation} \label{diffeqK8}
\frac{1}{8} e^{s_1} \ddstwo k_{11}{}\left(s_2,-s_1-s_2\right)
-\frac{1}{8} e^{s_1+s_2} \ddstwo k_{12}{}\left(-s_1-s_2,s_1\right)
\end{equation}
\[
-\frac{1}{8} \ddsone k_8{} \left(s_1,s_2\right)+\frac{1}{8} e^{s_1+s_2} \ddsone k_{12}{}\left(-s_1-s_2,s_1\right) = 
\]
\begin{center}
\begin{math}
-\Big ( \widetilde K_{8}(s_1, s_2, s_3) -   \widetilde K_{8, \,\textnormal{s}}(s_1, s_2, s_3)
\Big )\bigm|_{s_3 = -s_1-s_2},  
\end{math}
\end{center}
where  
\[
\widetilde K_{8, \,\textnormal{s}}(s_1, s_2, s_3) = 
\]
\[
\frac{1}{8 \left(s_1+s_2+s_3\right)} \Big (  -k_8\left(s_1,s_2\right)+e^{s_1+s_2+s_3} k_8\left(-s_2-s_3,s_2\right)-e^{s_1} k_{11}\left(s_2,-s_1-s_2\right)\]
\[+e^{s_1} k_{11}\left(s_2,s_3\right)-e^{s_1+s_2} k_{12}\left(-s_1-s_2,s_1\right)+e^{s_1+s_2} k_{12}\left(s_3,-s_2-s_3\right) \Big ). 
\]

In order to see the general structure in a  4-variable case, we just mention that  the restriction to the hyperplane 
$s_1+s_2+s_3+s_4=0$ of the functional relation corresponding to the function $\widetilde K_{17}$ gives 
\begin{equation} \label{diffeqK17}
-\frac{1}{16} e^{s_1+s_2} \ddsthree k_{17}{}\left(s_3,-s_1-s_2-s_3,s_1\right)+\frac{1}{16} e^{s_1} \ddsthree k_{19}{}\left(s_2,s_3,-s_1-s_2-s_3\right)
\end{equation}
\[
+\frac{1}{16} e^{s_1+s_2} \ddstwo k_{17}{}\left(s_3,-s_1-s_2-s_3,s_1\right)-\frac{1}{16} e^{s_1+s_2+s_3} \ddstwo k_{19}{}\left(-s_1-s_2-s_3,s_1,s_2\right)
\]
\[
-\frac{1}{16} \ddsone k_{17}{}\left(s_1,s_2,s_3\right)+\frac{1}{16} e^{s_1+s_2+s_3} \ddsone k_{19}{}\left(-s_1-s_2-s_3,s_1,s_2\right) = 
\] 
\begin{center}
\begin{math}
- 
\Big ( \widetilde K_{17}(s_1, s_2, s_3, s_4) -   \widetilde K_{17, \,\textnormal{s}}(s_1, s_2, s_3, s_4)
\Big )\bigm|_{s_4 = -s_1-s_2-s_3}, 
\end{math}
\end{center}
where
\[
 \widetilde K_{17, \,\textnormal{s}}(s_1, s_2, s_3, s_4) =
\]
\[
\frac{1}{16 \left(s_1+s_2+s_3+s_4\right)} \Big (   
-k_{17}\left(s_1,s_2,s_3\right)-e^{s_1+s_2} k_{17}\left(s_3,-s_1-s_2-s_3,s_1\right)
\]
\[
+e^{s_1+s_2} k_{17}\left(s_3,s_4,-s_2-s_3-s_4\right)
+e^{s_1+s_2+s_3+s_4} k_{17}\left(-s_2-s_3-s_4,s_2,s_3\right)
\]
\[
-e^{s_1} k_{19}\left(s_2,s_3,-s_1-s_2-s_3\right)+e^{s_1} k_{19}\left(s_2,s_3,s_4\right)
\]
\[
-e^{s_1+s_2+s_3} k_{19}\left(-s_1-s_2-s_3,s_1,s_2\right)+e^{s_1+s_2+s_3} k_{19}\left(s_4,-s_2-s_3-s_4,s_2\right) \Big ). 
\]

\subsection{Action of cyclic groups in the differential system, invariant expressions 
and simple flow of the system}
In the partial differential system of the form given by \eqref{diffeqK3}, \eqref{diffeqK8}, \eqref{diffeqK17} 
the action of the cyclic groups $\mathbb{Z}/2 \mathbb{Z}$, $\mathbb{Z}/3 \mathbb{Z}$, 
$\mathbb{Z}/4 \mathbb{Z}$  is involved. For example, in \eqref{diffeqK3} one can see very 
easily that  $\mathbb{Z}/2 \mathbb{Z}$ is acting by 
\[
T_2(s_1) = -s_1, \qquad s_1 \in \R. 
\]  
Using this fact, in \cite{arXiv1611.09815}, symmetries of some lengthy expressions are 
explored, which we recall in this subsection. 
\begin{theorem} \label{k1tok7eqnsinvariance}
For any  integers $j_0, j_1$ in $\{3, 4, 5, 6, 7\}$, 
\[
e^{-\frac{s_1}{2}} \big ( - (k'_{j_0}(s_1) + k'_{j_1}(s_1) )+ e^{s_1} \left( k'_{j_0}(-s_1)+ k'_{j_1}(-s_1) \right ) \big ), 
\]
is in the kernel of $1+T_2$. Moreover,  considering the finite difference expressions 
in the differential system corresponding to the following cases, one can find explicitly finite differences of the $k_j$ that are in the kernel of $1+T_2$: 
\begin{enumerate}
\item When $(j_0, j_1) = (3, 3)$.
\item When $(j_0, j_1) = (4, 4)$.
\item When $(j_0, j_1) = (5, 5)$.
\item When $(j_0, j_1) = (6, 7)$.
\end{enumerate}
\end{theorem}

In \eqref{diffeqK8}, the action of $\mathbb{Z}/3 \mathbb{Z}$ is involved as we have 
the following transformation acting on the variables: 
\begin{equation} \label{T3actionformula}
T_3(s_1,s_2) = (-s_1-s_2,s_1).   
\end{equation}
Using the latter, symmetries of more complicated expressions are 
discovered in \cite{arXiv1611.09815}:

\begin{theorem} \label{k8tok16eqnsinvariance}
for any integers $j_0, j_1, j_2$ in $\{8, 9, \dots, 16 \}$,   
\begin{eqnarray*}
&&e^{-\frac{2s_1}{3}-\frac{s_2}{3} } \Big ( 
 -\ddsone (k_{j_0}+k_{j_1}+k_{j_2})\left(s_1,s_2\right) \\
 && \qquad \qquad - e^{s_1+s_2} (\ddstwo - \ddsone )(k_{j_0}+k_{j_1}+k_{j_2}) 
\left(-s_1-s_2,s_1\right) \\
&& \qquad \qquad + e^{s_1} \ddstwo (k_{j_0}+k_{j_1}+k_{j_2})\left(s_2,-s_1-s_2\right) \Big ) 
\end{eqnarray*}
is in the kernel of $1+T_3+T_3^2$ . 
Also there are finite differences of the functions $k_j$ associated with the 
following cases that are in the kernel of $1+T_3+T_3^2$: 
\begin{enumerate}
\item  When $(j_0, j_1, j_2)=( 8, 12, 11)$. 
\item When $(j_0, j_1, j_2)= ( 9, 13, 10)$. 
\item When $(j_0, j_1, j_2)= ( 14, 16, 15)$. 
\end{enumerate}
\end{theorem}

The action of $\mathbb{Z}/4 \mathbb{Z}$ in the partial differential system 
can be seen in \eqref{diffeqK17} since the following transformation is involved:
\[
T_4(s_1,s_2,s_3) = (-s_1-s_2-s_3,s_1,s_2). 
\] 
The symmetries of the functions with respect to this action are also 
analysed in \cite{arXiv1611.09815}: 
\begin{theorem} \label{k17tok20eqnsinvariance}
For any pair of integers $j_0, j_1$ in $\{17, 18, 19, 20 \}$, 
\begin{eqnarray*}
&&e^{-\frac{3s_1}{4}-\frac{s_2}{2}-\frac{s_3}{4}} \Big (  -  \ddsone (k_{j_0}{} + k_{j_1})\left(s_1,s_2,s_3\right) \\
&& \qquad \qquad \qquad - e^{s_1+s_2+s_3} (\ddstwo - \ddsone) (k_{j_0}{} + k_{j_1}) (-s_1-s_2-s_3, s_1, s_2)  \\
&& \qquad \qquad \qquad - e^{s_1+s_2} (\ddsthree - \ddstwo) (k_{j_0}{} + k_{j_1}) (s_3, -s_1-s_2-s_3, s_1) \\
&& \qquad  \qquad \qquad +e^{s_1} \ddsthree (k_{j_0}{} + k_{j_1}) (s_2, s_3, -s_1-s_2-s_3) \Big ) 
\end{eqnarray*}
is in the kernel of $1 + T_4 + T_4^2 + T_4^3$.  
Moreover, there are expressions given by finite differences of the $k_j$ corresponding to the 
following cases that are 
in the kernel of $1 + T_4 + T_4^2 + T_4^3$: 
\begin{enumerate}
\item When $(j_0, j_1) = (17, 19)$.
\item When $(j_0, j_1) =(18, 18)$. 
\item When $(j_0, j_1)= (20, 20)$. 
\end{enumerate}
\end{theorem}

Moreover, in \cite{arXiv1611.09815}, it is shown that a very simple flow defined by 
\[
(s_1, s_2, \dots, s_n) \mapsto (s_1 +t , s_2, \dots, s_n), 
\] 
combined with the action of the cyclic groups as described above, 
can be used to write the differential part of the 
partial differential system. In order to illustrate the idea, we just mention 
that for example in the case that the action of $\Z/3\Z$ is involved via 
the transformation \eqref{T3actionformula}, one defines the orbit $\mathcal{O}k$ of any 2-variable 
function $k$ by  
\[
\mathcal{O}k (s_1, s_2) = \left (  k(s_1, s_2), k(-s_1-s_2, s_1), k(s_2, -s_1 -s_2)  \right ). 
\]
Then one has to use the auxiliary function 
\[
\alpha_2(s_1, s_2) = e^{-\frac{2s_1}{3}-\frac{s_2}{3} },
\] 
to write 
\[
 \left ( \ddt   \mathcal{O} k (s_1 +t, s_2) \right )  \cdot  \left (\mathcal{O} \alpha_2 (s_1, s_2) \right )
 \]
as a finite difference expression when  $k= k_{j_0} + k_{j_1} + k_{j_2}$ and 
$(j_0, j_1, j_2)$ is either $( 8, 12, 11)$, $( 9, 13, 10)$, or  $( 14, 16, 15)$.  One can refer 
to \S 4.3 of \cite{arXiv1611.09815} for more details and to see the treatment of all cases in detail. 

\subsection{Gradient calculations leading to functional relations}

Here we explain how the functional relations written in \S \ref{functionalrelatioinssubsec} 
were derived in \cite{arXiv1611.09815}. In fact the idea comes from \cite{MR3194491}, where a fundamental identity was 
proved and by means of a  functional relation, the 2-variable function 
of the scalar curvature term $a_2$ was written in terms of its 1-variable function.  
The main identity to use from \cite{MR3194491} is that, if we consider the conformally 
perturbed Laplacian, 
\[
\triangle_h = e^{h/2} \triangle e^{h/2}.    
\]
then for the spectral zeta function defined by 
\[
\zeta_{h}(a, s) = \Tr(a \, \triangle_h^{-s}), 
\qquad s \in \C, \,\, \Re(s) \gg 0, 
\]
one has 
\begin{equation} \label{GradientIdentity}
\dep \zeta_{h + \vep a}(1, s) 
=  
-\frac{s}{2}\, \zeta_h \left( \widetilde a , s \right),  
\end{equation}
where 
\[
\widetilde a = \int_{-1}^1 e^{uh/2} a e^{-uh/2} \, du. 
\]
One can then see that 
\[
\zeta_h(a, -1) = - \vphi_0(a \, a_4(h)), \qquad a \in \CNT, \,\, h=h^* \in \CNT.  
\]
Therefore, it follows from \eqref{GradientIdentity} that 
\begin{equation} \label{tildea4start}
\dep \vphi_0(a_4(h + \vep a)) 
= 
-\frac{1}{2} \zeta_h (\widetilde a, -1) 
=
\frac{1}{2} \vphi_0(\widetilde a \, a_4(h)) 
=
- \vphi_0 \left ( a e^{h} \, \widetilde a_4(h) \right). 
\end{equation} 
where $\widetilde a_4(h)$ is given by the same formula as \eqref{a4closedformula} when 
the functions  $K_j(s_1, \dots, s_n)$  are replaced by 
\[
\widetilde K_j(s_1, \dots, s_n) 
= 
\frac{1}{2^n} \frac{\sinh \left ( (s_1+\dots+s_n)/2 \right )}{(s_1+\dots+s_n)/2} K_j(s_1, \dots, s_n).  
\]
Hence, the gradient $\dep \vphi_0(a_4(h + \vep a))  $ can be calculated mainly by using the 
important identity \eqref{tildea4start}. 

There is a second way of calculating the gradient $\dep \vphi_0(a_4(h + \vep a))  $ 
which yields finite difference expressions. For this approach a series of lemmas were 
necessary as proved in \cite{arXiv1611.09815}, which are of the following type. 

\begin{lemma}  \label{GrofFuncCalc3}  
For any smooth function $L(s_1, s_2, s_3)$ and any elements $x_1, x_2, x_3, x_4$ in $C(\NT)$, 
under the trace $\varphi_0$, one has: 
\begin{eqnarray*}
&& e^h  \left ( \dep L(\nabep, \nabep, \nabep)(x_1 \cdot x_2 \cdot x_3) \right )  x_4 \\
&=& a e^h L^\vep_{3,1}(\nab, \nab, \nab, \nab)(x_1 \cdot x_2 \cdot x_3 \cdot x_4) 
+
a e^h L^\vep_{3,2}(\nab, \nab, \nab, \nab)(x_2 \cdot x_3 \cdot x_4 \cdot x_1) \\
&&+
a e^h L^\vep_{3,3}(\nab, \nab, \nab, \nab)(x_3 \cdot x_4 \cdot x_1 \cdot x_2) 
+
a e^h L^\vep_{3,4}(\nab, \nab, \nab, \nab)(x_4 \cdot x_1 \cdot x_2 \cdot x_3),    
\end{eqnarray*}
where 
\begin{eqnarray*}
L^\vep_{3,1}(s_1, s_2, s_3, s_4) &:=& e^{s_1+s_2+s_3+s_4} \frac{L(-s_2-s_3-s_4, s_2, s_3) - L(s_1, s_2, s_3)}{s_1+s_2+s_3+s_4}, \\
L^\vep_{3,2}(s_1, s_2, s_3, s_4) &:=& e^{s_1+s_2+s_3} \frac{L(s_4, - s_2- s_3 -s_4, s_2) - L(-s_1 - s_2 -s_3, s_1, s_2)}{s_1+s_2+s_3+s_4}, \\
L^\vep_{3,3}(s_1, s_2, s_3, s_4) &:=& e^{s_1+s_2} \frac{L( s_3, s_4, -s_2 -s_3 -s_4) - L(s_3, -s_1-s_2-s_3, s_1)}{s_1+s_2+s_3+s_4}, \\
L^\vep_{3,4}(s_1, s_2, s_3, s_4) &:=& e^{s_1} \frac{L(s_2, s_3, s_4) - L(s_2, s_3, -s_1-s_2-s_3)}{s_1+s_2+s_3+s_4}. 
\end{eqnarray*}
\end{lemma}

Also, in order to perform necessary manipulations in the second calculation of the 
gradient $\dep \vphi_0(a_4(h + \vep a))  $, one needs lemmas of this type:

\begin{lemma} \label{DerFuncCalc3}
For any smooth function $L(s_1, s_2, s_3)$ and any elements 
$x_1, x_2, x_3$ in  $C(\NT)$, one has: 
\begin{eqnarray*}
&&\delta_j \left (  L(\nab, \nab, \nab )(x_1 \cdot x_2 \cdot x_3) \right ) \\
&=& 
L(\nab, \nab, \nab)( \delta_j(x_1) \cdot x_2 \cdot x_3) 
+ L(\nab, \nab, \nab)( x_1 \cdot \delta_j(x_2) \cdot x_3) \\
&&
+ L(\nab, \nab, \nab)( x_1 \cdot x_2 \cdot \delta_j(x_3))
+ L^{\delta}_{3, 1}(\nab, \nab, \nab, \nab) (\delta_j(h) \cdot x_1 \cdot x_2 \cdot x_3) \\
&& 
+ L^{\delta}_{3, 2}(\nab, \nab, \nab, \nab) (  x_1 \cdot \delta_j(h) \cdot x_2 \cdot x_3) 
+ L^{\delta}_{3, 3}(\nab, \nab, \nab, \nab) (  x_1  \cdot x_2 \cdot \delta_j(h) \cdot x_3) \\
&& 
+ L^{\delta}_{3, 4}(\nab, \nab, \nab, \nab) (  x_1  \cdot x_2 \cdot x_3 \cdot \delta_j(h) ),  
\end{eqnarray*}
where 
\begin{eqnarray*}
L^{\delta}_{3, 1}(s_1, s_2, s_3, s_4) &:=& \frac{L(s_2, s_3, s_4) - L(s_1+s_2, s_3, s_4)}{s_1}, 
\\
L^{\delta}_{3, 2}(s_1, s_2, s_3, s_4) &:=& \frac{L(s_1+s_2, s_3, s_4) - L(s_1, s_2+s_3, s_4)}{s_2}, \\
L^{\delta}_{3, 3}(s_1, s_2, s_3, s_4) &:=& \frac{L(s_1, s_2+s_3, s_4) - L(s_1, s_2, s_3+s_4)}{s_3}, \\
L^{\delta}_{3, 4}(s_1, s_2, s_3, s_4) &:=& \frac{L(s_1, s_2, s_3+s_4) - L(s_1, s_2, s_3)}{s_4}. 
\end{eqnarray*}

\end{lemma}

After performing the second gradient calculation in \cite{arXiv1611.09815}, and comparing it with the 
first calculation based on \eqref{tildea4start}, the functional relations were derived conceptually.

\subsection{The term $a_4$ for non-conformally flat metrics on noncommutative four tori}
It was illustrated in \cite{arXiv1611.09815} that, having the calculation of the term $a_4$ for the noncommutative 
two torus in place, one can conveniently write a formula for the term $a_4$ of a non-conformally 
flat metric on the noncommutative four torus that is the product of two noncommutative two tori. 
The metric is the noncommutative analog of the 
following metric. Let 
$(x_1, y_1, x_2, y_2) \in$  $\mathbb{T}^4 = (\R/ 2\pi \Z)^4$ be the coordiantes of the ordinary 
four torus and consider the metric 
\begin{equation*} 
g =  e^{-h_1(x_1, y_1)} \left ( dx_1^2 + dy_1^2\right ) + e^{-h_2(x_2, y_2)} \left ( dx_2^2 + dy_2^2\right ),  
\end{equation*}
where $h_1$ and $h_2$ are smooth real valued functions. The Weyl tensor is conformally invariant, 
and one can assure that the above metric is not conformally flat by calculating the components of 
its Weyl tensor and observing that they do not all vanish. The non-vanishing components are:  
\[
C_{1212} =
\]
\begin{center}
\begin{math}
\frac{1}{6} e^{-h_1\left(x_1,y_1\right)} \partial_{y_1}^2 h_1{}\left(x_1,y_1\right)
+\frac{1}{6} e^{h_2\left(x_2,y_2\right) -2 h_1\left(x_1,y_1\right)}  \partial_{y_2}^2 h_2{}\left(x_2,y_2\right)
+\frac{1}{6} e^{-h_1\left(x_1,y_1\right)}  \partial_{x_1}^2 h_1\left(x_1,y_1\right)
+\frac{1}{6} e^{h_2\left(x_2,y_2\right) -2 h_1\left(x_1,y_1\right)}  \partial_{x_2}^2h_2\left(x_2,y_2\right), 
\end{math}
\end{center}
\[
C_{1313}= -\frac{1}{2}e^{-h_2\left(x_2,y_2\right) + h_1\left(x_1,y_1\right)} C_{1212}, 
\]
\[
C_{2424} = C_{2323} = C_{1414}= C_{1313}, 
\]
\[
C_{3434} = e^{-2 h_2\left(x_2,y_2\right) +2 h_1\left(x_1,y_1\right)} C_{1212}. 
\]

Now, one can consider a noncommutative four torus of the form 
$\NTp \times \NTpp$ that is the product of two noncommutative two tori. Its 
 algebra has four unitary generators 
$U_1, V_1, U_2, V_2$ with the following relations:  each element of the pair $(U_1, V_1)$ 
commutes with each element of the pair $(U_2, V_2)$, and there are fixed 
irrational real numbers $\theta'$ and $\theta''$ such that: 
\[
V_1 \, U_1 = e^{2 \pi i \theta'}\, U_1 \,V_1, 
\qquad 
V_2 \,U_2= e^{2 \pi i \theta''}\, U_2\, V_2.    
\] 
One can then choose conformal factors  
$e^{-h'}$ and $e^{-h''}$, where $h'$ and $h''$ are 
selfadjoint elements in $\CNTp$ and $\CNTpp$, respectively, 
and use them to perturb the flat metric of each component conformally. 
Then the Laplacian of the  product geometry is given by 
\[
\triangle_{\vphi', \vphi''} 
= 
\triangle_{\vphi'} \otimes 1 + 1 \otimes \triangle_{\vphi''}, 
\]
where 
$\triangle_{\vphi'}$ and $\triangle_{\vphi''}$ are respectively 
the Laplacians of the perturbed metrics on $\NTp$ and $\NTpp$. 
Now one can use a simple Kunneth formula to find the term 
$a_4$ in the asymptotic expansion 
\begin{eqnarray} \label{heatexp4d}
&& \Tr(a \exp ( -t \triangle_{\vphi', \vphi''}))  \sim \nonumber  \\
&& \qquad \qquad  t^{-2} \left ( ( \vphi'_0 \otimes \vphi''_0 ) (a\, a_0) + (\vphi'_0 \otimes \vphi''_0 ) (a\, a_2)\, t + ( \vphi'_0 \otimes \vphi''_0) (a\, a_4)\, t^2 + \cdots \right )   
\end{eqnarray}
in terms of the known terms appearing in the following expansions: 
\begin{eqnarray*}
\Tr( a' \exp (-t \triangle_{\vphi'}) )  
&\sim& 
t^{-1} \left ( \vphi'_0(a' \, a'_0) + \vphi'_0( a'\, a'_2) \, t + \vphi'_0(a'\, a'_4) \,t^2 + \cdots \right ),  \\
\Tr( a'' \exp ( -t \triangle_{\vphi''}) )  
&\sim& 
t^{-1} \left ( \vphi''_0(a'' \, a''_0) + \vphi''_0(a''\, a''_2) \, t + \vphi''_0(a''\, a''_4)\, t^2 + \cdots \right ). 
\end{eqnarray*}
The general formula is
\[
a_{2n} = \sum_{i=0}^{n} a'_{2i} \otimes a''_{2(n-i)}   \in   C^\infty(\NTp \times \NTpp),  
\]
hence an explicit formula for $a_4$ of the noncommutative four torus with the product 
geometry explained above since there are explicit formulas for its ingredients.  

In this case of the non-conformally flat metric on the product geometry,  
two modular automorphisms are involved in the formulas for the geometric invariants 
and this motivates further systematic research on {\it twistings} that involve two dimensional modular 
structures, cf.  \cite{MR866491}.

\section{{\bf Twisted spectral triples and Chern-Gauss-Bonnet 
theorem for ergodic $C^*$-dynamical systems}}.  

This section is devoted to the notion of twisted spectral triples 
and some details of their appearance in the context of noncommutative 
conformal geometry. In particular we explain construction of twisted 
spectral triples for ergodic $C^*$-dynamical systems and the validity 
of the Chern-Gauss-Bonnet theorem in this vast setting.

\subsection{Twisted spectral triples }

The notion of twisted spectral triples was introduced in \cite{MR2427588} to incorporate 
the study of type III algebras using noncommutative differential geometric 
techniques. In the definition of this notion, in addition to a triple $(A, H, D)$ 
of a $*$-algebra $A$, a Hilbert space $H$, and an unbounded operator 
$D$ on $H$ which plays the role of the Dirac operator, one has to bring 
into the picture an automorphism $\sigma$ of $A$ which interacts  with 
the data as follows. Instead of the ordinary commutators $[D, a]$ as in the 
definition of an ordinary spectral triple, in the twisted case one asks for the 
boundedness of the twisted commutators $[D, a]_\sigma = Da - \sigma(a) D$. 
More precisely, here also one assumes a representation of $A$ by bounded 
operators on $H$ such that the operator $D a - \sigma (a) D$ is defined on 
the domain of $D$ for any $a \in A$, and that it extends by continuity to a 
bounded operator on $H$. 

This twisted notion of a spectral triple is essential for type III examples as 
this type of algebras do not possess non-zero trace functionals, and ordinary 
spectral triples with suitable properties cannot be constructed over them for the 
following reason \cite{MR2427588}. If $(A, H, D)$ is an $m^+$-summable ordinary spectral triple 
then the following linear functional on $A$ defined by 
\[
a \mapsto \Tr_\omega(a |D|^{-m}) 
\]
gives a trace. The main reason for this is that the kernel of the Dixmier trace $\Tr_\omega$ is  
a large kernel that contains all operators of the form  $|D|^{-m} a - a |D|^{-m} $, $a \in A$, 
if the ordinary commutators are bounded. In fact we are using the {\it regularity} 
assumption on the spectral triple, which in particular requires boundedness of the 
commutators of elements of $A$ with $|D|$ as well as with $D$ (indeed this is a natural 
condition and the main point is that one is using ordinary commutators).   
Hence, trace-less algebras cannot fit into the paradigm of ordinary spectral triples. 

It is quite amazing that in \cite{MR2427588}, examples are provided where one can obtain 
boundedness of twisted commutators $D a - \sigma(a) D$ and $|D|a - \sigma(a) |D|$ 
for all elements $a$ of the algebra by means of an algebra automorphism 
$\sigma$, where the Dirac operator $D$ has the $m^+$-summability property.  
Then they use the boundedness of the twisted commutators to show that operators 
of the form $|D|^{-m} a - \sigma^{-m}(a) |D|^{-m} $ are in the kernel of the Dixmier trace 
and the linear functional $a \mapsto \Tr_\omega(a |D|^{-m})$ yields a twisted trace on 
$A$.

\subsection{Conformal perturbation of a spectral triple}
One of the main examples in \cite{MR2427588} that demonstrates the need 
for the notion of twisted spectral triples in noncommutative geometry is 
related to conformal perturbation of Riemannian metrics. That is, 
if $D$ is the Dirac operator of a spin manifold equipped with a Riemannian 
metric $g$, then, after a conformal perturbation of the metric to 
$\tilde g = e^{-4h} g$ by means of a smooth real valued function $h$ 
on the manifold, the Dirac operator 
of the perturbed metric $\tilde g$ is unitarily equivalent to the operator 
\[
\tilde D= e^h D e^h. 
\]
So this suggests that given an ordinary spectral triple $(A, H, D)$ with 
a noncommutative algebra $A$, 
since the metric is encoded in the analog $D$ of the Dirac operator, 
one can implement conformal perturbation of the metric by fixing a 
self-adjoint element $h \in A$ and by then replacing $D$ with $\tilde D= e^h D e^h. $ 
However, it turns out that the tripe $(A, H, \tilde D)$ is not necessarily 
a spectral triple any more, since, because of noncommutativity of $A$, 
the commutators $[\tilde D, a]$, $a \in A$, are not necessarily bounded 
operators. Despite this, interestingly, the remedy brought forth in \cite{MR2427588} is 
to introduce the automorphism   
\[
\sigma (a) = e^{2h} a e^{-2h}, \qquad a \in A, 
\]
which yields the bounded twisted commutators 
\[
[\tilde D, a]_\sigma = \tilde D a - \sigma(a) \tilde D, \qquad a \in A. 
\]

\subsection{Conformal perturbation of the flat metric on $\ncttwo$} \label{conformalnc2torustwistedspec}

Another important example, which is given in \cite{MR2907006}, shows that twisted 
spectral triples can arise in a more intrinsic manner, compared to the 
example we just illustrated, when a conformal perturbation is implemented. 
In \cite{MR2907006}, the flat geometry of $\ncttwo$ is perturbed by a fixed conformal 
factor $e^{-h}$, where $h=h^* \in \sncttwo$. This is done by replacing the 
canonical trace $\varphi_0$ on $\ancttwo$ (playing the role of the 
volume form) by the tracial state $\varphi(a) = \varphi_0(a e^{-h})$, $a \in \ancttwo$.  
In order to represent the opposite algebra of $\ancttwo$ on the Hilbert space 
$\mathcal{H}_\varphi$, obtained from $\ancttwo$ by the GNS construction, 
one has to modify the ordinary action induced by right multiplication.  That is, 
one has to consider the action defined by 
\[
a^{op} \cdot \xi = \xi e^{-h/2} a e^{h/2}.  
\]

It then turns out that with the new action, the ordinary commutators 
$[D, a]$, $a \in \sncttwo^{op}$, are not bounded any more, where 
$D$ is the Dirac operator 
\[ 
D=
\left(\begin{array}{c c}
0 & \partial_\varphi^* \\
\partial_\varphi & 0
\end{array}\right)
: \mathcal{H} \to \mathcal{H}. 
\] 
Here,  
\[
\partial_\varphi = \delta_1 +i \delta_2 : \mathcal{H}_\varphi \to \mathcal{H}^{(1,0)},
\] 
where  $\mathcal{H}^{(1,0)}$, 
the analogue of $(1, 0)$-forms, is the Hilbert space 
completion of finite sums $\sum a \partial(b), a, b 
\in A_\theta^\infty $, 
with respect to the inner product 
\[
(a \partial b, c \partial d) = \varphi_0 (c^* a (\partial b) (\partial d)^*), 
\]
and 
\[
\mathcal{H}= \mathcal{H}_\varphi \oplus \mathcal{H}^{(1,0)}. 
\]

The remedy for obtaining bounded commutators  is to  use  a 
twist given by the automorphism
\[
\sigma(a^{op}) = e^{-h/2} a e^{h/2},    
\]
which leads to bounded twisted commutators \cite{MR2907006} 
\[
[D, a^{op}]_\sigma = Da^{op} -\sigma(a^{op}) D,  \qquad a \in \sncttwo^{op}. 
\]

\subsection{Conformally twisted spectral triples for $C^*$-dynamical systems}
\label{ergodiccstartsubsec}

The example in \S \ref{conformalnc2torustwistedspec} 
inspired the construction of twisted spectral triples for 
general ergodic $C^*$-dynamical systems in \cite{MR3458156}. The Dirac operator 
used in this work, following more closely the geometric approach taken 
originally in \cite{MR572645}, is the analog of the de Rham operator. An important 
reason for this choice is that an important goal in \cite{MR3458156} was to confirm 
the validity of the analog of the Chern-Gauss-Bonnet theorem in the vast 
setting of ergodic $C^*$-dynamical systems.

In this subsection we consider a $C^*$-algebra $\Aa$ with a strongly continuous ergodic 
action $\alpha$ of a compact Lie group $G$ of dimension $n$, and we let  
$\Aa^\infty$ denote the smooth subalgebra of $\Aa$, which is defined as:
\[
\Aa^\infty = 
\{ a \in \Aa : \textnormal{the map }g \mapsto \alpha_g(a) \text{ is in }C^\infty (G, A) \}.
\]
Following closely the construction in \cite{MR572645}, we can define a space of 
{\it differential forms} on $\Aa$ by using the exterior powers of $\LieG^*$, 
namely that for  $k = 0, 1, 2, \dots, n, $ we set: 
\begin{equation}
\label{Eqn:DefDiffForms}
\Omega^k(\Aa, G) = \Aa \otimes \bigwedge^k \LieG^*,
\end{equation}
where $\LieG^*$ is the linear dual of the Lie algebra $\LieG$ of the Lie group $G$. 
We consider the inner product $\langle \cdot, \cdot \rangle$ on  $\LieG^*$ induced by the Killing form, and  
extend it to an inner product on $\bigwedge^k \LieG^*$ by setting
$$
\langle v_1 \wedge \cdots  \wedge  v_k, w_1 \wedge \cdots \wedge w_k \rangle = \det( \langle v_i, w_j \rangle ). 
$$

After fixing an orthonormal basis $(\omega_j)_{j = 1, \ldots , n}$ 
for $\LieG^*$, we equip the above differential forms with an exterior 
derivative $d : \Omega^k(\Aa, G) \to \Omega^{k+1}(\Aa, G)$ given by 
\begin{multline}
\label{Eqn:Defd}
d( a \otimes \omega_{i_1} \wedge \omega_{i_2} \wedge \cdots \wedge \omega_{i_k} ) =\sum_{j = 1}^n \partial _j(a) \otimes \omega_j \wedge \omega_{i_1} \wedge \omega_{i_2} \wedge \cdots \wedge \omega_{i_k} \\
 - \frac{1}{2} \sum_{j=1}^k \sum_{\alpha, \beta} (-1)^{j+1} c^{i_j}_{\alpha \beta} a \otimes \omega_\alpha \wedge \omega_\beta \wedge \omega_{i_1} \wedge \cdots \wedge \omega_{i_{j-1}} \wedge \omega_{i_{j+1}} \wedge \cdots  \wedge \omega_{i_k},
\end{multline}
where  the coefficients $c^{i}_{\alpha \beta}$ are the \emph{structure constants} 
of the Lie algebra $\LieG$ determined by the relations  
\[
[\partial _\alpha, \partial _\beta] = \sum_{i=1}^n c^i_{\alpha \beta} \partial _i
\] 
for the predual $(\partial_j)_{j=1, \dots, n}$ of $(\omega_j)_{j=1, \dots, n}$.  
This exterior derivative satisfies $d \circ d = 0$ on $\Omega^\bullet(\Aa, G)$, therefore we have a complex  $\left (\Omega^\bullet(\Aa, G) , d \right)$. This complex is called the Chevalley-Eilenberg cochain complex 
with coefficients in $\Aa$, one can refer to \cite{MR938524} for more details.

We now define an inner  product on $\Omega^k(\Aa, G)$,  for which we make use 
 of the unique $G$-invariant tracial state $\varphi_0$ on $\Aa$, see \cite{MR625345}. The inner 
 product is defined by  
 \begin{equation}
(a \otimes v_1 \wedge \cdots \wedge v_k, a' \otimes w_1 \wedge \cdots \wedge w_k )_0= \varphi_0( a^* a' ) \det( \langle v_i, w_j \rangle ). 
\end{equation}
We denote the Hilbert space completion of $\Omega^k(\Aa, G)$ with respect 
to this inner product by $\mathcal{H}_{k,0}$. 

In order to implement a conformal perturbation, we fix a selfadjoint element 
$h \in \Aa^\infty$, define the following new inner product on $\Omega^k(\Aa, G)$: 
\begin{equation}
\label{Eqn:ScalProd}
(a \otimes v_1 \wedge \cdots \wedge v_k, a' \otimes w_1 \wedge \cdots \wedge w_k )_h = \varphi_0( a^* a' e^{(n/2-k)h}) \det( \langle v_i, w_j \rangle ), 
\end{equation}
and denote the associated Hilbert space by $\mathcal{H}_{k, h}$.

One of the goals is to construct ordinary and twisted spectral triples 
by using the unbounded operator $d+ d^*$, the analog of the de Rham 
operator, acting on the direct sum of all $\mathcal{H}_{k, h}$.   
Here the adjoint $d^*$ of $d$ is of course taken with respect to the conformally  
perturbed inner product $(\cdot, \cdot)_h$.  
The Hilbert spaces are simply related by the unitary maps $U_k \colon \mathcal{H}_{k, 0} \to \mathcal{H}_{k, h} $ given on degree $k$ forms by:
\[
U_k( a \otimes v_1 \wedge \cdots \wedge v_k ) = a e^{-(n/2-k)h/2} \otimes v_1 \wedge \cdots \wedge v_k.
\]
Therefore, for simplicity, we use these unitary maps to transfer the 
operator $d+d^*$ to an unbounded operator $D$ acting on the Hilbert space 
$\mathcal{H}$ that is the direct sum of all $\mathcal{H}_{k,0}$.   
We can now state the following result from \cite{MR3458156}.   
 
\begin{theorem}
\label{thm:twistedspecdynamical}
Consider the above constructions associated with a $C^*$-algebra $\Aa$ 
with an ergodic action of an $n$-dimensional Lie group $G$. 
The operator $D$ has a selfadjoint extension which is $n^+$-summable. 
With the representation of $\Aa^\infty$ on 
$\mathcal{H} = \bigoplus_k \mathcal{H}_{k,0}$ induced by left multiplication,  
the triple $(\Aa^\infty, \mathcal{H}, D)$ is an ordinary spectral triple. 
However, when one represents the opposite algebra of $\Aa^\infty$ on 
$\mathcal{H}$ using multiplication from right, one obtains a twisted spectral 
triple with respect to the automorphism defined by $\beta(a^{op}) = e^{h } a e^{-h }$. 
\end{theorem}

The spectral triples described in the above theorem can in fact be equipped 
with the grading operator given by
\[
\gamma( a \otimes v_1 \wedge \cdots \wedge v_k) = (-1)^k ( a \otimes v_1 \wedge \cdots \wedge v_k). 
\] 
Related to this grading, it is interesting to study the Fredholm index of the 
operator $D$, which is unitarily equivalent to $d+d^*$, when viewed as 
an operator from the direct sum of all even differential forms to the direct sum 
of all odd differential forms. We shall discuss this issue shortly.

 \subsection{The Chern-Gauss-Bonnet theorem for $C^*$-dynamical systems}

 In \S \ref{GBSCnc2section} we briefly discussed the Gauss-Bonnet theorem for surfaces, which 
 states that for any  closed oriented two dimensional Riemannian 
 manifold $\Sigma$ with scalar curvature $R$, one has 
 \[
 \frac{1}{4 \pi} \int_\Sigma R  =  \chi(\Sigma), 
 \]
 where $\chi(\Sigma)$ is the Euler characteristic of $\Sigma$. 
 The Chern-Gauss-Bonnet theorem generalizes this result to higher 
 even dimensional manifolds. That is, in higher dimensions as well, 
 the Euler characteristic, which is a topological invariant, 
 coincides with the integral of a certain geometric invariant, namely 
 the {\it Pfaffian} of the curvature form. Given a closed oriented Riemannian 
 manifold $M$ of even dimension $n$, consider the Levi-Civita connection, which 
 is the unique torsion-free metric-compatible connection  on the tangent 
 bundle $TM$. Let us denote the matrix of local 2-forms representing the curvature of 
 this connection by $\Omega$. The Chern-Gauss-Bonnet theorem 
 states that the Pfaffian of $\Omega$ (the square root of the determinant defined on the space of anti-symmetric 
 matrices) integrates over the manifold to the Euler characteristic of the 
 manifold, up to multiplication by a universal  constant: 
 \[
 \frac{1}{(2 \pi)^n}\int_M \text{Pf}(\Omega) = \chi(M). 
 \]
 
 Interestingly, there is a spectral way of interpreting such relations between 
 local geometry and topology of manifolds. Relevant to our discussion is 
 indeed the Fredholm index of the de Rham operator $d+d^*$, where $d$ 
 is the de Rham exterior derivative and $d^*$ is its adjoint with respect to the 
 metric on the differential forms induced by the Riemannian metric. The Fredholm 
 index of $d+d^*$ should be calculated when the operator is viewed as a map 
 from the direct sum of all even differential forms to the direct sum of all odd 
 differential forms on $M$: 
 \[
 d+d^* : \Omega^{even} M  = \bigoplus \Omega^{2k}M \to \Omega^{odd} M= \bigoplus \Omega^{2k+1}M
 \]  
 The index of this operator is certainly an important geometric quantity since 
 the adjoint $d^*$ of $d$ heavily depends on the choice of metric on the manifold. 
 Amazingly, using the {\it Hodge decomposition theorem}, one can find a 
 canonical identification of the de Rham cohomology group $H^k(M)$ with  
 the kernel of the Laplacian $\triangle_k = d^* d + d d^* : \Omega^k M \to \Omega^k M$.  
 This can then be used to show that the index of $d+d^*$ is equal to the Euler 
 characteristic of $M$. Moreover, one can appeal to the {\it McKean-Singer 
 index theorem} to find curvature related invariants appearing in small time 
 heat kernel expansions associated with $d+d^*$ to see that the index is given 
 by the integral of curvature related invariants. 
 
 In \cite{MR3458156}, this spectral approach is taken to show that the analog of the 
 Chern-Gauss-Bonnet theorem can be established for ergodic $C^*$-dynamical 
 systems. Let us consider the set up and the constructions presented in 
 \S \ref{ergodiccstartsubsec} for a $C^*$-algebra $\Aa$ with an ergodic action of a 
 compact Lie group 
 $G$ of dimension $n$. Then, one of the main results proved in \cite{MR3458156} is the 
 following statement. Here, $d$ is given by \eqref{Eqn:Defd}, $h = h^* \in \Aa^\infty$ 
 is the element that was used to implement with $e^h$ a conformal perturbation of the metric, and the Hilbert space $\mathcal{H}_{k, h}$ is the completion of the $k$-differential forms $\Omega^k(\Aa, G)$ with respect to the perturbed metric.   
 
 \begin{theorem}
 The Fredholm index of the operator 
 \[
 d+d^*: \bigoplus_k \mathcal{H}_{2k, h} \to \bigoplus_k \mathcal{H}_{2k+1, h}
 \]
 is equal to the Euler characteristic $\chi(A, G)$ of the complex $(\Omega^\bullet(A, G), d)$. Since  $\chi(A, G) = \sum_k (-1)^k dim \left ( H^k(A, G)\right )$ is the alternating sum 
 of the dimensions of the cohomology groups,  the index of $d+d^*$ is independent of 
 the conformal factor $e^h$ used for perturbing the metric. 
 \end{theorem}


\section{{\bf The Ricci curvature}}

Classically, scalar curvature is only a deem shadow of the full Riemann  curvature tensor.  In fact there is no evidence that Riemann considered anything else but  the  full curvature tensor, and, equivalently, the sectional curvature. Both were defined by him for a Riemannian manifold. The Ricci and scalar curvatures were later defined  by  contracting the Riemann curvature tensor with the metric tensor.  Once the metric is given in a local coordinate chart,  all three curvature tensors can be computed  explicitly via algebraic formulas involving only partial derivatives of the metric tensor. This is a  purely algebraic process, with deep geometric and analytic implications. It is also a top down process, going from the full Riemann curvature tensor, to Ricci curvature,  and then  to scalar curvature. 

The situation in the noncommutative  case is reversed and we have to move up the ladder,  starting from  the scalar  curvature first, which is the easiest to define spectrally, being given by the second term of the heat expansion for the scalar Laplacian, the square of the Dirac operaor in general.  Thus after treating the scalar curvature,  which we recalled  in previous sections together with examples, one  should next  try to  define and possibly compute, in some cases,  a Ricci curvature tensor. But how? In \cite{arXiv1612.06688} and  motivated by the local formulas for the asymptotic expansion of heat kernels in spectral geometry,  the authors propose a definition of Ricci curvature in a noncommutative  setting.  One necessarily  has to use the asymptotic expansion of  Laplacians on functions and 1-forms and a version of the Weitzenb\"ock formula. 

As we shall see in this section, the Ricci operator of an oriented closed Riemannian manifold can be realized as a spectral functional, namely the functional defined by the zeta function of the full Laplacian of the de Rham complex, localized by smooth endomorphisms of the cotangent bundle and their trace. In the noncommutative case, this Ricci functional uniquely determines a density element, called the Ricci density, which plays the role of the Ricci operator. The main result of  \cite{arXiv1612.06688}  provides a general definition and an explicit computation of the Ricci density when the conformally flat geometry of the curved noncommutative two torus is encoded in the  de Rham spectral triple. In a  follow up  paper \cite{arXiv1808.02977},   the Ricci curvature of a   noncommutative  curved three torus is computed. In this section we  explain these recent developments in  more detail. \\

\subsection{A Weitzenb\"ock formula} 

The Weitzenb\"ock formula 
$$ \text{\bf Hodge - Bochner = Ricci}$$
in conjunction with Gilkey's asymptotic expansion gives an opening to define the Ricci curvature in spectral terms.   
Let $M$ be a closed oriented  Riemannian manifold. 
Consider the de Rham spectral triple
$$(C^\infty(M),L^2(\Omega^{ev}(M)) \oplus L^2(\Omega^{odd} (M)),d+\delta, \gamma),$$ which is the even spectral triple constructed from the de Rham complex.
Here $d$ is  the exterior derivative,   $\delta$  is its adjoint acting  on the exterior algebra, and $\gamma$ is the $\mathbb{Z}_2$-grading on  forms.  The  eigenspaces for eigenvalues 1 and -1 of $\gamma$ are even and odd forms, respectively.
 The full Laplacian on forms $\Delta =d\delta+\delta d$ is the Laplacian of the Dirac operator $d+\delta$, and is the direct sum of Laplacians on $p$-forms, $\Delta =\oplus \Delta_p$. As a Laplace type operator, $\Delta $ can be written as $\nabla^*\nabla -E$ by Weitzenb\"ock formula, where $\nabla$ is the Levi-Civita connection extended to all forms and 
\begin{equation*}
E=-\frac12c(dx^\mu)c(dx^\nu)\Omega(\partial_\mu,\partial_\nu).
\end{equation*} 
Here $c$ denotes the Clifford multiplication and $\Omega$ is the curvature operator of the Levi-Civita  connection acting on exterior algebra. The restriction of $E$ to one forms gives the Ricci operator.

\subsection{Ricci curvature as a spectral functional}
The Ricci curvature of a Riemannian manifold 
 $(M^m,g)$ is originaly defined as follows. 
Let $\conn$ be the Levi-Civita connection of the metric $g$.
The Riemannian operator and the curvature tensor are define for vector fields $X, Y, Z, W$  by
\begin{eqnarray*}
\riemop(X,Y)&:=&\conn_X\conn_Y-\conn_Y\conn_X-\conn_{[X,Y]},\\
\riem(X,Y,Z,W)&:=&g(\riemop(X,Y)Z,W).
\end{eqnarray*}
With respect to the coordinate frame  $\partial_\mu=\frac{\partial}{\partial x^\mu}$, the components 
of the curvature tensor are denoted by
$$\riem_{\mu\nu\rho\epsilon}:=\riem(\partial_\mu,\partial_\nu,\partial_\rho,\partial_\epsilon).$$
The components of  the Ricci tensor $Ric$ and scalar curvature $R$ are given by
 \begin{eqnarray*}
\ric_{\mu\nu}&:=&g^{\rho\epsilon}\riem_{\mu\rho\epsilon\nu},\\
\scalar&:=&g^{\mu\nu}\ric_{\mu\nu}=g^{\mu\nu}g^{\rho\epsilon}\riem_{\mu\rho\epsilon\nu}.
\end{eqnarray*}
Now these algebraic formulas have no chance to be extended to a noncommutative setting in general. One must thus look for an spectral alternative reformulation.

Let $P:C^\infty(V)\to C^\infty(V)$ be a positive elliptic differential operator of order two  acting on the sections of a smooth Hermitian vector bundle $V$ over $M$.
The heat trace $\Tr(e^{-tP})$  has  a short time   asymptotic expansion of the form 
\begin{eqnarray*}
\Tr(e^{-tP})\sim \sum_{n=0}^\infty a_n(P)t^{\frac{n-m}2},\qquad t\to 0^+,
\end{eqnarray*}
where $a_n(P)$  are integrals of local densities 
\begin{equation*}
a_n(P)=\int \tr(a_n(x,P))dx.
\end{equation*}
Here $dx = d  vol_x$ is the Riemannian volume form and $\tr$ is the fiberwise matrix trace. 
The endomorphism $a_n(x,P)$ can be uniquely determined by localizing the heat trace by an  smooth endomorphism $F$ of $V$.  It is easy to see that the asymptotic expansion of the 
 localized heat trace $\Tr(Fe^{-tP})$ is of the form 
\begin{eqnarray}\label{smearedheatasymp}
\Tr(Fe^{-tP})\sim \sum_{n=0}^\infty a_n(F,P)t^{\frac{n-m}2},
\end{eqnarray}
with 
\begin{equation}\label{localizedcoef}
a_n(F,P)=\int  \tr\big(F(x)a_n(x,P)\big) dx.
\end{equation}

If $P$ is a Laplace type operator i.e., its leading symbol is given by the metric tensor,  then the densities $a_n(x,P)$ can be expressed in terms of the Riemannian curvature, an endomorphism $E$, and their derivatives. 
The endomorphism $E$ measures how far the operator $P$ is from being the Laplacian $\conn^*\conn$ of a connection $\conn$ on $V$, that is  
\begin{equation}\label{decomtolapandend}
E=\conn^*\conn-P.
\end{equation}
The first two densities of the heat equation for such $P$ are given by \cite[Theorem 3.3.1]{MR1396308} 
\begin{eqnarray}
a_0(x,P)&=&(4\pi)^{-m/2}{\rm I},\label{a0}\\
a_2(x,P)&=&(4\pi)^{-m/2}\big(\frac{1}{6}\scalar(x)+E\big).
\end{eqnarray}

For the scalar Laplacian  $\lap_0$, the connection is the de Rham  differential  $d:C^\infty(M)\to \Omega^1(M)$, and  $E=0$. 
Hence the first two first terms  of  the heat kernel of $\lap_0$ are given by 
\begin{eqnarray}
a_0(x,\lap_0)&=&(4\pi)^{-m/2},\label{a0functions}\\
a_2(x,\lap_0)&=&(4\pi)^{-m/2}\,\frac{1}{6}{\scalar(x)}.\label{a2functions}
\end{eqnarray}

For Laplacian on one forms $\lap_1:\Omega^1(M)\to \Omega^1(M)$,  the Hodge-de Rham Laplacian, the connection in \eqref{decomtolapandend} is the Levi-Civita connection on the cotangent bundle. The endomorphism $E$ is the 
 negative of the Ricci operator, $E=-\ric$, on the cotangent bundle,  which is defined by raising the first index of the Ricci tensor (denoted by $\ric$ as well), 
\begin{equation*}
\ric_x(\alpha^\sharp,X)=\ric_x(\alpha)(X),\quad \alpha\in T_x^*M,\,  X\in T_xM.
\end{equation*}
Therefore, one has
\begin{eqnarray}
a_0(x,\lap_1)&=&(4\pi)^{-m/2}{\rm I},\\ \label{a01forms}
a_2(x,\lap_1)&=&(4\pi)^{-m/2}\big(\frac{1}{6}{\scalar(x)}-\ric_x\big). \label{a21forms}
\end{eqnarray}
We can use the function  $\tr(F)$  to localize the heat trace of the scalar Laplacian $\lap_0$ and get
the identity 
\begin{equation}\label{differenceofterms}
a_2(\tr(F),\lap_0)-a_2(F,\lap_1)=(4\pi)^{-m/2}\int_M \tr\big(F(x)\ric_x\big)dx. 
\end{equation}
 This motivates the following definition.
\begin{definition}[\cite{arXiv1612.06688}] \label{riccifunctional}
The { Ricci functional} of the  closed  Riemannian manifold $(M, g)$ is the functional 
on $C^\infty(\End(T^*M))$ defined as 
\begin{equation*}
\ricfun(F)=a_2(\tr(F),\lap_0)-a_2(F,\lap_1).
\end{equation*}
\end{definition} 
\begin{proposition} For a  closed Riemannian manifold $M$ of dimension $m$, we have the short time asymptotics
\begin{equation*}
\Tr\left(\tr(F)e^{-t\lap_0}\right)-\Tr\left(F e^{-t\lap_1}\right)\sim \ricfun(F)\, t^{1-\frac{m}{2}}.
\end{equation*}
\end{proposition}
\begin{proof}
By \eqref{a0} and \eqref{localizedcoef}, we have $\tr(F)a_0(x,\lap_0)=\tr(F(x)a_0(x,\lap_1)).$
This implies that
\begin{equation}\label{equalityoffirstterms}
a_0(\tr(F),\lap_0)=a_0(F,\lap_1),\qquad F\in C^\infty(\End(T^*M)).
\end{equation} 
The   asymptotic expansion of the localized heat kernel \eqref{smearedheatasymp} then shows that the first terms will cancel each other. The  difference of the second terms, which are multiples of $t^{1-\frac{m}{2}}$, will become the first term in the asymptotic expansion of the differences of localized heat kernels.  
\end{proof}

\subsection{Spectral zeta function and the Ricci functional} \label{zetaformulationofricci} 

The  spectral zeta function of  a positive elliptic operator $P$ is defined as 
\begin{equation*}
\zeta(s,P)=\Tr(P^{-s}({\rm I}-\pr)),\quad \Re (s)\gg 0,
\end{equation*} 
where $\pr $ is the projection on the kernel of $P$. 
Its  localized version is $\zeta(s,F,P)=\Tr(FP^{-s}({\rm I}-\pr))$. 
These  function have  a meromorphic extension to the complex plane  with  isolated  simple poles. Using the Mellin transform, one finds explicit relation between 
 residue at the poles and  coefficients  of the heat kernel. This leads to following expression for  the Ricci functional in terms of zeta functions.
\begin{proposition}
For a  closed Riemannian manifold $M$ of dimension $m>2$, we have 
\begin{equation}\label{zetariccimnot2}
\ricfun(F)=
 \Gamma(\frac{m}{2}-1)\res_{s=\frac{m}{2}-1}\Big(\zeta(s,\tr(F),\lap_0)-\zeta(s,F,\lap_1)\Big).
\end{equation}
If  $M$ is two dimensional, then
\begin{equation}\label{zetariccim2}
\ricfun(F)=\zeta(0,\tr(F),\lap_0)-\zeta(0,F,\lap_1)+\Tr(\tr(F)\pr_0)-\Tr(F\pr_1),
\end{equation}
where $\pr_j$ is the projection on the kernel of Laplacian $\lap_j$, $j=0,1$.
\end{proposition}

 It follows that 
the difference of zeta functions $\zeta(s,\tr(F),\lap_0)-\zeta(s,F,\lap_1)$ is regular at $m/2$, and its first pole is located at $s=m/2-1$.

To work with the Laplacian on one forms, we will use smooth endomorphisms $F$ of the cotangent bundle.   
The smearing endomorphism $\tilde{F}=\tr(F){\rm I}_0 \oplus F\in C^\infty(\End(\bigwedge^\bullet M))$, where ${\rm I}_0$ denotes the identity map on functions, can be used to localize the heat kernel of the full Laplacian and
\begin{equation}\label{riccifunctionalgamma}
\ricfun(F)=a_2(\gamma \tilde{F},\lap).
\end{equation}
With the above notation,  one can express the Ricci functional as special values of the (localized) spectral zeta functions
\begin{equation}\label{Riccizetagamma}
\ricfun(F)=\begin{cases}
 \Gamma(\frac{m}{2}-1)\res_{s=\frac{m}{2}-1}\zeta(s,\tilde{F}\gamma,\lap)& m>2,\\&\\
\zeta(0,\gamma\tilde{F},\lap)+\Tr(\tr(F)\pr_0)-\Tr(F\pr_1)& m=2.
\end{cases}
\end{equation}

The flat de Rham spectral triple of the noncommutative two torus can be perturbed by a  Weyl factor $e^{-h}$ with $h\in \Ai$ a self adjoint element.
 This procedure gives rise to the de Rham spectral triple of a curved noncommutative torus. 
The Ricci functional  is defined in a similar fashion as above,  and it can be shown  that there exists an element $\ricden\in \Ai\otimes M_2(\mathbb{C})$, called the Ricci density, such that 
\begin{equation*}
\ricfun(F)=\frac{1}{\Im(\tau)}\varphi(\tr(F\ricden)e^{-h}), \quad F\in \Ai\otimes M_2(\mathbb{C}).
\end{equation*}
.

\subsection{The de Rham spectral triple for the noncommutative two torus}
 In this section, we describe the de Rham spectral triple of a  noncommutative two torus $\A$ equipped with a complex structure. This is a variation of the Dolbeault complex that we used in Section. Consider the vector space $W=\mathbb{R}^2$, and let $\tau$ be a complex number in the upper half plane. Let  $g_\tau$ be the positive definite symmetric bilinear form on $W$  given by 
\begin{equation}\label{flattaumetricontangent}
g_\tau=\frac{1}{\Im(\tau)^2}
\begin{pmatrix}
|\tau|^2 & -\Re(\tau)\\ 
-\Re(\tau) & 1 
\end{pmatrix}.
\end{equation} 
The inverse $g_\tau^{-1}=\begin{pmatrix}
1 & \Re(\tau)\\ 
\Re(\tau) & |\tau|^2 
\end{pmatrix}$ of $g_\tau$ is a metric on the dual of $W$. 
The entries of $g_\tau^{-1}$ will be denoted by $g^{jk}$.

Let  $\extp^\bullet  W^*_\mathbb{C}$ be the  exterior algebra of $W^*_\mathbb{C}=(W\otimes \mathbb{C})^*$.
The algebra $\Ai\otimes \extp^\bullet W^*_\mathbb{C}$
is  the algebra of differential forms on  the noncommutative two torus $\A$. 
In this framework, the Hilbert space of functions, denoted $\mathcal{H}^{(0)}$, is simply the Hilbert space given by the GNS construction of $\Ai$ with respect to 
$\frac{1}{\Im(\tau)}\varphi$. Additionally, the
Hilbert space of one forms, denoted $\mathcal{H}^{(1)}$, is the space $\mathcal{H}_0\otimes (\mathbb{C}^2,g_\tau^{-1})$ with inner product given by  
\begin{equation}\label{oneforminnerproduct}
\langle a_1\oplus a_2, b_1\oplus b_2  \rangle=  \frac{1}{\Im(\tau)}\sum_{j,k} g^{jk}\varphi(b_k^*a_j),\;\;a_i,\,b_i\in \Ai,
\end{equation}
while the Hilbert space of two forms, denoted $\mathcal{H}^{(2)}$, is given by the GNS construction of $\Ai$ with respect to $\Im(\tau)\, \varphi$.

The exterior derivative on elements of $\Ai$ is given by
\begin{equation}\label{donfunctionstau}
a\mapsto i\delta_1(a)\oplus i\delta_2(a),\quad a\in \Ai.
\end{equation}
It will be denoted by $d_0$, when considered as a densely defined operator from $\mathcal{H}^{(0)}$ to $\mathcal{H}^{(1)}$.
The operator $d_1: \mathcal{H}^{(1)}\to \mathcal{H}^{(2)}$ is defined on the elements of $\Ai\oplus\Ai$ as 
\begin{equation}\label{donformstau}
a\oplus b\mapsto i\delta_1(b)-i\delta_2(a),\qquad a,b\in\Ai.
\end{equation}
The adjoints of the operators $d_0: \mathcal{H}^{(0)}\to \mathcal{H}^{(1)}$ and $d_1:\mathcal{H}^{(1)}\to \mathcal{H}^{(2)}$ are then given by 
\begin{eqnarray*}
d_0^*(a\oplus b) &=& -i \delta_1(a) -i \Re(\tau) \delta_2(a)-i\Re(\tau) \delta_1(b)-i|\tau|^2 \delta_2(b),\\
d_1^*(a) &=& (i|\tau|^2 \delta_2(a)+i\Re(\tau) \delta_1(a))\oplus (-i\Re(\tau) \delta_2(a)-i\delta_1(a)),
\end{eqnarray*} for all $a,b\in \Ai$.
\begin{definition}\label{deRhamspectralflat}
 The (flat) de Rham spectral triple of $\A$ is the even spectral triple 
$(\A,\mathcal{H},D),$
where $\mathcal{H}=\mathcal{H}^{(0)}\oplus\mathcal{H}^{(2)} \oplus \mathcal{H}^{(1)}$, $ D=\begin{pmatrix}
0& d^*\\ d& 0
\end{pmatrix},$ and $d=d_0+d_1^*$.
\end{definition}
\noindent Note that the operator $d$ and its adjoint $d^*=d_1+d_0^*$ act on $\Ai\oplus \Ai$ as
\begin{equation}\label{nonpertddstar}
d=\begin{pmatrix}
i\delta_1 & i|\tau|^2 \delta_2+i\Re(\tau) \delta_1\\
i\delta_2 & -i\Re(\tau) \delta_2-i\delta_1
\end{pmatrix},\quad
d^*=\begin{pmatrix}
-i\delta_1-i\Re(\tau) \delta_2 &-i\Re(\tau) \delta_1- i|\tau|^2 \delta_2\\
-i\delta_2 & i\delta_1
\end{pmatrix}. 
\end{equation}
Note also that the de Rham spectral triple introduced in Definition \ref{deRhamspectralflat}, is isospectral to the de Rham complex spectral triple of the flat torus $\mathbb{T}^2$ with the metric given by \eqref{flattaumetricontangent}.\\

\subsection{The twisted de Rham spectral triple}
The conformal perturbation of the metric on the noncommutative two torus is implemented by changing the tracial state $\varphi$ by a noncommutative Weyl factor $e^{-h}$, where the dilaton $h$ is a selfadjoint smooth element of the noncommutative two torus, $h=h^*\in\Ai$.  
The conformal change of the metric by the Weyl factor $e^{-h}$ will change the inner product on functions and on two forms as follows.
On functions, the Hilbert space given by GNS construction of $\A$ with respect to the positive linear functional $\varphi_{0}(a)=\frac{1}{\Im(\tau)}\varphi(ae^{-h})$ will be denoted by $\mathcal{H}_h^{(0)}$. 
Therefore the inner product of $\mathcal{H}_h^{(0)}$ is given by   
\begin{equation*}
\langle a,b \rangle_0=\frac{1}{\Im(\tau)}\varphi(b^*ae^{-h}),\quad a,b\in \A.
\end{equation*}
On one forms, the Hilbert space will stay the same as in \eqref{oneforminnerproduct}, and will be denote by $\mathcal{H}^{(1)}_h$.
 On the other hand, the Hilbert space of two forms, denoted by $\mathcal{H}^{(2)}_h$, is the Hilbert space given by the GNS construction of $\A$ with respect to $\varphi_2(a)=\Im(\tau)\varphi(a e^{h})$. 
 Hence its inner product is given by
\begin{equation*}
\langle a,b \rangle_2={\Im(\tau)}\varphi(b^*ae^{h}),\quad a,b\in \A.
\end{equation*}
The  positive functional $a\mapsto \varphi(ae^{-h})$, called the conformal weight, is a twisted trace of which modular operator is given by  
$$\Delta(a)=e^{-h}ae^{h},\quad a \in\A.$$
The logarithm $\log\Delta$ of the modular operator will be denoted by $\nabla$, and is given by $\nabla(a)=-[h,a]$.  For more details the reader can check the previous sections.

The exterior derivatives are defined in the same way they were defined in the flat case \eqref{donfunctionstau} and \eqref{donformstau}. However, to emphasize that they are acting on different Hilbert spaces, we will denote them by $d_{h,0}:\mathcal{H}_h^{(0)}\to \mathcal{H}_h^{(1)}$ and $d_{h,1}:\mathcal{H}_h^{(1)}\to \mathcal{H}_h^{(2)}$.

Next, we consider the Hilbert spaces $\mathcal{H}^{+}_h=\mathcal{H}_h^{(0)}\oplus  \mathcal{H}_h^{(2)}$ and $\mathcal{H}^{-}_h=\mathcal{H}_h^{(1)}$, and the operator $d_h:\mathcal{H}^{+}_h\to \mathcal{H}^{-}_h$, $d_h=d_{h,0}+ d_{h,1}^*$. Therefore
\begin{equation*}
d_h=\begin{pmatrix}
i\delta_1 & \big(i|\tau|^2 \delta_2+i\Re(\tau) \delta_1\big)\comp R_{k^2}\\
i\delta_2 & \big(-i\Re(\tau) \delta_2-i\delta_1\big)\comp R_{k^2}
\end{pmatrix},
\end{equation*}
and its adjoint is given by
 \begin{equation*}
d_h^*=\begin{pmatrix}
R_{k^2}\comp \big( i\delta_1-i\Re(\tau) \delta_2\big) &R_{k^2}\comp\big(-i\Re(\tau) \delta_1- i|\tau|^2 \delta_2\big)\\
-i\delta_2 & i\delta_1
\end{pmatrix} .
\end{equation*}
We also consider the operator   
\begin{equation*}
D_h=\begin{pmatrix}
0 & d_h^*\\ d_h & 0
\end{pmatrix},
\end{equation*}
which acts on $\mathcal{H}_h= \mathcal{H}^{+}_h\oplus\mathcal{H}_h^{-}$. Define the Hilbert space $\mathcal{H}=\mathcal{H}_0\oplus\mathcal{H}_0\oplus\mathcal{H}_0\oplus\mathcal{H}_0$ and the unitary operator $W:\mathcal{H}\to \mathcal{H}_h,$ $$W=R_k\oplus R_{k^{-1}}\oplus {\rm {I}\sb{\mathcal{H}_0\oplus\mathcal{H}_0}}.$$ The operator $D_h$ can be transferred to an operator $\tilde{D}_h$ on $\mathcal{H}$ by the inner perturbation 
$$\tilde{D}_h:=W^* D_h W=\begin{pmatrix} 0 &  R_k \comp d^*\\ d\comp R_k & 0 \end{pmatrix}.$$

In order to define the twisted, or modular,  de Rham spectral triple for the noncommutative two torus, we employ the following constructions from \cite{MR3194491}. 
Let $(\mathcal{A},\mathcal{H}^+\oplus \mathcal{H}^-,D)$ be an even spectral triple with  grading operator $\gamma$, where $D=\begin{pmatrix}
0 & T^*\\ T & 0
\end{pmatrix}$  and $T:\mathcal{H}^+\to\mathcal{H}^-$ is an unbounded operator with adjoint $T^*$. If $f\in\mathcal{A}$ is positive and invertible,
then $(\mathcal{A,H},D_{(f,\gamma)})$ is a modular spectral triple with respect to the inner automorphism $\sigma(a)=faf^{-1}$, $a\in \mathcal{A}$  ( \cite[Lemma 1.1]{MR3194491}), where the Dirac operator is given by $$D_{(f,\gamma)}=\begin{pmatrix}
0 & f T^*\\ Tf & 0
\end{pmatrix}. $$
On the other hand, any modular spectral triple $(\mathcal{A,H},D)$ with an automorphism $\sigma$ admits a transposed modular spectral triple $(\mathcal{A}^{\rm op},\bar{\mathcal{H}}, D^t) $ \cite[Proposition 1.3]{MR3194491}, where $\mathcal{A}^{\rm op}$ is the opposite algebra of $\mathcal{A}$, $\bar{\mathcal{H}}$ is the dual Hilbert space, the action of $\mathcal{A}^{\rm op}$ on $\bar{\mathcal{H}}$ is the transpose of the the action of $\mathcal{A}$ on $\mathcal{H}$, $D^t$ is the transpose of $D$, and $\sigma'$ is the automorphism of $\mathcal{A}^{\rm op}$ given by $\sigma'(a^{\rm op})=(\sigma^{-1}(a))^{\rm op}$.

\begin{proposition}
Let $k=e^{h/2}$, where $h=h^*\in\Ai$.
 The triple $(\A^{\rm op},\mathcal{H},\tilde{D}_h)$  is a modular spectral triple, where the automorphism  of $\A^{\rm op}$ is given by
\begin{equation*}
a^{\rm op}\mapsto (k^{-1}ak)^{\rm op},\quad a\in \Ai,
\end{equation*} 
and the representation of $\A^{\rm op}$ on $\mathcal{H}$ is given by the right multiplication of $\A$ on $\mathcal{H}$.
Moreover, the transposed of the modular spectral triple $(\A^{\rm op},\mathcal{H},\tilde{D}_h)$ is isomorphic to the perturbed spectral triple 
\begin{equation}\label{finalspectraltriple}
(\A,\mathcal{H},\bar{D}_h),\quad 
\bar{D}_h=\begin{pmatrix}
0 & k d\\ d{}^* k & 0
\end{pmatrix},
\end{equation}
where the operators $d$ and $d^*$ are as in \eqref{nonpertddstar}.
\end{proposition}

\begin{definition}\label{modularspectraltriple}
The modular spectral triple $(\A,\mathcal{H},\bar{D}_h)$ in \eqref{finalspectraltriple} will be called the {modular de Rham spectral triple} of the noncommutative two torus with dilaton $h$.
\end{definition}

\subsection{Ricci functional and Ricci curvature  for the curved noncommutative torus}

Using  the  pseudodifferential calculus with symbols in  $\Ai\otimes M_4(\mathbb{C})$, 
 one shows  that the localized heat trace of  $\bar{D}_h^2$  has an asymptotic expansion  with coefficients of the  form  
\begin{equation*}
a_n(E,\bar D_h^2)=\varphi \circ \tr\left(E\, c_n(\bar D_h^2)\right), \qquad E\in \Ai\otimes M_4(\mathbb{C}),
\end{equation*} 
where $c_n(\bar D_h^2)\in \Ai\otimes \big(M_2(\mathbb{C})\oplus M_2(\mathbb{C})\big)$ and $\tr$ is the matrix trace.
The Ricci functional can now be defined:
\begin{definition}[\cite{arXiv1612.06688}]\label{riccifunctionalnc}
The { Ricci functional} of the modular de Rham spectral triple $(\A,\mathcal{H}, \bar D_h)$ is the functional on $\A\otimes M_2(\mathbb{C})$ defined as
\begin{equation*}
\ricfun(F)=a_2(\gamma \tilde{F},\bar D^2)=\zeta(0,\gamma\tilde{F},\bar {D}^2_h)+\Tr(\tr(F)\pr_0)-\Tr(F \pr_1), 
\end{equation*}
where $\tilde{F}=\tr(F)\oplus 0\oplus F$, and 
$\pr_j$ is  the orthogonal projection on the kernel of $\lap_{h,j}$, for $j=0,1$.
\end{definition}

\begin{lemma}\label{Riccidensitylemma} 
There exists an element $\ricden\in\Ai\otimes M_2(\mathbb{C})$ such that  for all $F\in \Ai\otimes M_2(\mathbb{C})$
\begin{equation*}
\ricfun(F)=\frac{1}{\Im(\tau)}\varphi(\tr(F\ricden)e^{-h}).
\end{equation*}
\end{lemma}
\begin{proof}
For any such $F$  we have
$$a_2(\gamma \tilde{F},\bar D^2 )=a_2(\tr(F),\lap_{h,0})-a_2(F,\lap_{h,1}).$$
Now  $\tr(F)e^{-t\lap_{h,0}}=\tr(Fe^{-t\lap_{h,0}\otimes {\rm I}_2})$,  and thus
$$a_2(\tr(F),\lap_{h,0})=a_2(F,\lap_{h,0}\otimes I_2).$$
As a result, we have
\begin{eqnarray*}
\ricden(F)&=&a_2(\tr(F),\lap_{h,0})-a_2(F,\lap_{h,1})\\
&=& \varphi\left(\tr \Big(F\big(c_2(\lap_{h,0})\otimes {\rm I}_2- c_2(\lap_{h,1})\big)\Big)\right)\\
&=& \frac{1}{\Im(\tau)} \varphi\left(\Im(\tau)\tr \Big(F\big(c_2(\lap_{h,0})\otimes {\rm I}_2- c_2(\lap_{h,1})\big)\Big)e^h e^{-h}\right).
\end{eqnarray*}
Hence, 
\begin{equation*}\label{ricdenexists}
\ricden=\Im(\tau)\Big(c_2(\lap_{h,0})\otimes {\rm I}_2- c_2(\lap_{h,1})\Big)e^h.
\end{equation*}
\end{proof}
\begin{definition}
The element $\ricden$  is called the {Ricci density} 
of the curved noncommutative torus  with dilaton $h$. 
\end{definition}

The terms $c_2(\lap_{h,j})$ can be computed  by integrating  the symbol of the parametrix of $\lap_{h,j}$. 
Since the operator $\lap_{h,1}$ is a  first order perturbation of $\lap_\varphi^{(0,1)}$,  we will only need to compute the difference
$c_2(\lap_{h,1})-c_2(\lap_\varphi^{(0,1)})\otimes {\rm I}_2$. 
The terms $c_2(\lap_{h,0})=c_2(k\lap_0k)$ and $c_2(\lap_\varphi^{(0,1)})$ are computed previously  in two places   by Connes-Moscovici and Fathizadeh-Khalkhali, and their  difference is given by
\begin{eqnarray*}\label{scalar}
   R^\gamma &=& \big(c_2(k\lap k)\otimes{\rm I}_2-c_2(\lap_\varphi^{(0,1)})\big) e^h\\\nonumber
   &=&-\frac{\pi}{\Im(\tau)}\Big(K_\gamma(\nabla)(\lap_0(\log k))+H_\gamma\big(\nabla_1,\nabla_2\big)\left(\square_\Re (\log k)\right)\\
   && \qquad\qquad\qquad\qquad\qquad\qquad+ S(\nabla_1,\nabla_2)
(\square_\Im (\log k))\Big)e^h.\notag
 \end{eqnarray*}
Here,  
\begin{eqnarray*}
\square_\Re (\ell) &=&
(\delta_1(\ell))^2+ \Re(\tau)\left(\delta_1(\ell)\delta_2(\ell)+\delta_2(\ell)\delta_1(\ell)\right)+|\tau|^2 (\delta_2(\ell))^2,\\
\square_\Im(\ell)&=& i\Im(\tau)(\delta_1(\ell)\delta_2(\ell)-\delta_2(\ell)\delta_1(\ell))
\end{eqnarray*}
with $\ell = \log k.$
Moreover, 
$$
K_\gamma(u) = \, \frac{\frac 12+\frac{\sinh(u/2)}{u}}{\cosh^2(u/4)}, 
$$
\begin{align*}
\begin{split}
&H_\gamma(s,t)= \big(1-\cosh((s+t)/2)\big) \times \\
&\frac{t (s+t) \cosh(s)-s (s+t) \cosh(t)+(s-t) (s+t+\sinh(s)+\sinh(t)-\sinh(s+t))}{s t (s+t) \sinh\left(\frac{s}{2}\right) \sinh\left(\frac{t}{2}\right) \sinh\left(\frac{s+t}{2}\right)^2},
\end{split}
\end{align*}
\begin{equation*}\label{Sfunction}
S(s,t)=\frac{(s+t-t\, \cosh(s)-s\, \cosh(t)-\sinh(s)-\sinh(t)+\sinh(s+t))}{s\, t\left(\sinh\left(\frac{s}{2}\right) \sinh\left(\frac{t}{2}\right) \sinh\left(\frac{s+t}{2}\right)\right)}.
\end{equation*}
The term $S$  coincides with the function $S$ found in \cite{MR3194491, MR3148618} for scalar curvature.\\

Now the main result of \cite{arXiv1612.06688} can be stated as follows. It computes the Ricci curvature density of a curved noncommutative two torus with a conformally flat metric. The proof of this theorem  is quite long and  complicated and will not be reproduced here.
\begin{theorem}[\cite{arXiv1612.06688}] \label{maintheorem}
Let $k=e^{h/2}$ with  $h\in \Ai$ a selfadjoint element.
Then the Ricci density of the modular de Rham spectral triple with dilaton $h$ is given by 
\begin{equation*}
\ricden= \frac{\Im(\tau)}{4\pi^2}R^\gamma\otimes {\rm I}_2- \frac{1}{4\pi} S(\nabla_1,\nabla_2)\big([\delta_1(\log k),\delta_2(\log k)]\big)e^h\otimes \begin{pmatrix}
i\Im(\tau) & \Im(\tau)^2  \\
 -1& i\Im(\tau)
\end{pmatrix}\,.
\end{equation*}\\
\end{theorem}

It is important to check the classical limit for consistency. In the commutative  limit   the Ricci density $\ricden$ is retrieved as
$\lim_{(s,t)\to (0,0)} \ricden$.
Since (cf. \cite{MR3194491}  for a proof)
$$\lim_{(s,t)\to (0,0)}R^\gamma=-\frac{\pi}{\Im(\tau)} \lap_0(\log k),$$
and  $[\delta_1(\log k),\delta_2(\log k)]=0$, we have
\begin{equation*}
\ricden|_{\theta=0}=\frac{-1}{4\pi} \lap_0(\log k)e^h\otimes {\rm I}_2.
\end{equation*}   
If we take into account  the normalization of the classical case that comes from the heat kernel coefficients,  this gives the formula for the Ricci operator in the classical case.

Unlike the commutative case, the Ricci density $\ricden$ in  the noncommutative case   is not  a symmetric matrix. 
Indeed, it has nonzero off diagonal terms,  which are multiples of 
$S(\nabla_1,\nabla_2)([\delta_1(\log k),\delta_2(\log k)])$. 
This phenomenon, observed in \cite{arXiv1612.06688}  for the first time, is obviously a consequence of the noncommutative nature of the space.   
It is an interesting feature of noncommutative geometry that, contrary to the commutative case, the Ricci curvature is not a multiple of the scalar curvature  even in dimension two. This manifests itself in the existence of off diagonal terms 
in the Ricci operator  $\ricden $ above.

It is clear that one can define in a similar fashion a Ricci curvature operator for higher dimensional noncommutative tori, as well as for noncommutative toric manifolds. Its computation in these cases poses an interesting   problem. This problem now is completely solved for noncommutative three tori is 
\cite{arXiv1808.02977}.  It would also be interesting to find the analogue of the Ricci flow based on our definition of Ricci curvature functional. It should be noted that for noncommutative two tori a  definition of Ricci flow, without a notion of Ricci curvature, is proposed in \cite{MR2947960}.

\section{{\bf Beyond conformally flat metrics and beyond dimension four}}

In  the study of spectral geometry of noncommutative tori   one is naturally interested in going beyond conformally flat metrics and beyond dimension four.  Even in the case of noncommutative two torus it is important to consider metrics which are not conformally flat. 
In fact while by uniformization theorem  we know that any metric on the two torus is conformally flat, there is strong evidence that this is not so in the noncommutative case. This is closely related
to the problem of classification of complex structures on the noncommutative two torus  via positive Hochschild cocycles, which is still unsolved. 

As far as higher dimensions go,  our  original methods does not allow us  to treat the dimension as a variable in the calculations and obtain explicit formulas in all dimensions
in a uniform manner.  This is in sharp contrast with the classical case where   formulas work in a uniform manner in all dimensions. In this section we report on a very recent development \cite{arXiv1811.04004}  where progress has been made on both fronts. 

In  the recent  paper \cite{arXiv1811.04004},
using a new strategy  based on Newton divided differences,  it is shown how  to consider non-conformal
metrics and how to  treat all higher dimensional noncommutative tori  in a uniform way. In fact based on older methods it was not clear how to extend the computation of the scalar curvature to a general higher dimensional case.   The class of non-conformal
metrics  introduced in \cite{arXiv1811.04004}  is quite large and leads to beautiful combinatorial identities for the
curvature via divided differences. In this section we shall briefly sketch the results obtained in \cite{arXiv1811.04004}, following closely its  organization of material.

\subsection{Rearrangement lemma revisited} To compute and effectively work with integrals of the form 

\begin{equation*}
\int_0^\infty (uk^2+1)^{-m}b(uk^2+1)u^{m} du,
\end{equation*}
the rearrangement lemma  was proved by Connes and Tretkoff  in \cite{MR2907006}. Here   $k=e^{h/2}$,  $h, b   \in C^\infty(\nctorus[2])$   and $h$  is   selfadjoint. The problem stems  from the fact that $h$ and $b$ need not commute. 
Later on this lemma was generalized, for the sake of curvature calculations,  for  more than one $b$ in   \cite{MR3194491,MR3148618}.
A  detailed study of this  lemma for more general integrands of the form  
\begin{equation*}
\int_0^\infty f_0(u,k)b_1 f_1(u,k)\rho_2\cdots b_n f_n(u,k) du,
\end{equation*}
was given by M. Lesch in \cite{MR3626561}, with a new proof and a new point of view. 
This approach uses the multiplication map 
$$\mu :a_1\otimes a_2\otimes \cdots\otimes a_n\mapsto a_1a_2\cdots a_n$$  from the projective tensor product $A^{\otimes_\gamma  n}$ to $A$. 
The above integral   is  expressed as  the contraction of the product of an element $F(k_{(0)},\cdots,k_{(n)})$ of $A^{\otimes_\gamma (n+1)}$, with  the element $b_1\otimes b_2\otimes \cdots b_n\otimes 1$ which is
\begin{equation*}
\mu\Big(F(k_{(0)},\cdots,k_{(n)})(b_1\otimes b_2\otimes \cdots b_n\otimes 1)\Big).
\end{equation*} 
 The above  element  is usually written in the so called {\it contraction form}
\begin{equation}\label{contractionform}
F(k_{(0)},\cdots,k_{(n)})(b_1\cdot b_2\cdot \cdots b_n).
\end{equation}

The following  version of the rearrangement lemma is stated  in  \cite{arXiv1811.04004}  with the domain of integration 
changed   from $[0,\infty)$ to any domain in $\mathbb{R}^N$.

\begin{lemma}[Rearrangement lemma  \cite{arXiv1811.04004}]\label{rearrangmentlemma}
Let $A$ be a unital $C^\ast$-algebra, $h\in A$ be a selfadjoint element, and $\Lambda$ be an open neighborhood of the spectrum of $h$ in $\mathbb{R}$.
For  a domain $U$ in $\mathbb{R}^N$, let 
$f_j:U \times \Lambda\to \mathbb{C}$, $0\leq j\leq n,$  be  smooth functions  such that $f(u,\lambda)=\prod_{j=0}^n f_j(u,\lambda_j)$ satisfies the following integrability condition:
for any compact subset $K\subset \Lambda^{n+1}$ and every given  multi-index $\alpha$ we have    
$$\int_U \sup\limits_{\lambda\in K}|\partial_\lambda^\alpha f(u,\lambda)|du<\infty.$$
Then,
\begin{equation}\label{formulaofrearrangement}
\int_U f_0(u,h)b_1f_1(u,h)\cdots b_n f_n(u,h)du=F(h_{(0)},h_{(1)},\cdots,h_{(n)})(b_1\cdot b_2\cdots b_n),
\end{equation}
where $F(\lambda)=\int_U f(u,\lambda)du$.\qed
\end{lemma}

In particular it follows that   every expression in the contraction form with a Schwartz function $F\in\mathcal{S}(\mathbb{R}^{n+1})$ used in the operator part can  be written as  an integral. 
In fact if we set   
$$f_n(\xi,\lambda)=\hat{f}(\xi)e^{i\xi_n \lambda},\quad f_j(\xi,\lambda)=e^{i\xi_j \lambda},\quad 0\leq j\leq n-1,$$ 
and $f(\xi,\lambda_0,\cdots,\lambda_n)=\prod_{j=0}^n f_j(\xi,\lambda_j)$,
by the Fourier inversion formula, we have 
$F(\lambda)=\int f(\xi,\lambda)d\xi$.
Then,  Lemma \ref{rearrangmentlemma}  gives the equality 
\begin{equation}\label{integralformformula}
F(h_{(0)},\cdots,h_{(n)})(b_1\cdot b_2\cdots  b_n)=\int_{\mathbb{R}^{n}}e^{i\xi_0 h}b_1e^{i\xi_1 h}b_2\cdots b_n e^{i\xi_n h}\hat{f}(\xi)d\xi.
\end{equation} 
This is crucial for calculations in  \cite{arXiv1811.04004}.

\subsection{A new idea}
As we saw in previous sections, to prove the Gauss-Bonnet theorem and to compute the scalar curvature of a curved noncommutative two  torus in  \cite{MR2907006, MR2956317} and      \cite{MR3194491, MR3148618},  
 the second density of the heat trace of the Laplacian $D^2$ of the Dirac operator had to be computed. 
 First, the  symbol of the  parametrix of $D^2$ was  computed,  next a contour integral coming from Cauchy's formula for the heat operator had to be computed, and finally  one had to integrate out the momentum variables. It was for this last step that the rearrangement lemma played an important role. Luckily, the contour integral could be avoided using a homogeneity argument.  

A key observation in \cite{arXiv1811.04004}  is that   one  need  not  wait  till  the  last  step to have elements in the contraction form.  It is just enough to start 
 off with operators whose symbol is  written in the contraction form
 \begin{equation*}
F(h_{(0)},\cdots,h_{(n)})(b_1\cdot b_2\cdots b_n).
\end{equation*} 
It is further noted  that the symbol calculus can be effectively applied to   differential operators whose symbols can be written in the contraction form.
These operators are called  {\it $h$-differential operators} in \cite{arXiv1811.04004}.  This is a new and larger class of differential operators that lends itself to precise spectral analysis. It is strictly larger than the class of Dirac Laplacians for conformally flat metrics on noncommutative tori which has been the subject of intensive studies lately.  
 
 Next, the Newton divided difference calculus was brought in  to find the   action   of derivations on elements in contracted form (Theorem \ref{derivationoffunctionofh} below). For example one has
\begin{gather*}
\delta_j\big(f(h_{(0)},h_{(1)})(b_1)\big)=\\
f(h_{(0)},h_{(1)})(\delta_j(b_1))
+\big[h_{(0)},h_{(1)};f(\cdot, h_{(2)} )(\delta_j(h)\cdot b)\big]
+\big[h_{(1)},h_{(2)};f(h_{(0)},\cdot)(b\cdot \delta_j(h))\big].
\end{gather*}
Using this  fact,  and applying the pseudodifferential calculus,   one can  compute the spectral densities of  positive $h$-differential operators whose principal symbol is given by a functional metric.
These  operators are called   {\it Laplace type $h$-differential operator} in \cite{arXiv1811.04004}. 

This change in order of the computations, i.e. writing symbols in the contraction form  first,  led to
  a smoother computation  symbolically, and played  a fundamental  in computing with   more general functional metrics. It also paved the way in  calculating  the curvature in all higher dimensions for conformally flat and twisted product of flat metrics.

\subsection{Newton divided differences}
A nice application of the rearrangement lemma  is to find   a formula for the differentials  of a  smooth element written in contraction form.  To this end,  Newton divided differences were used in \cite{arXiv1811.04004}.

Let $x_0, x_1,\cdots, x_n$ be distinct points in an interval $I\subset \mathbb{R}$ and let $f$ be a function on $I$. 
The {\it $n^{th}$-order Newton divided difference} of $f$, denoted by $[x_0,x_1,\cdots,x_n;f]$,
is the coefficient of $x^n$ in the interpolating polynomial of $f$ at the given points.
In other words, if the interpolating polynomial is $p(x)$ then 
\begin{equation*}
p(x)=p_{n-1}(x)+[x_0,x_1,\cdots,x_n;f](x-x_0)\cdots(x-x_{n-1}),
\end{equation*}
where $p_{n-1}(x)$ is a polynomial of    degree at most $n-1$.
There  is a  recursive formula for the divided difference which is given by
\begin{equation*}
\begin{aligned}[]
  [ x_0; f]&=f(x_0)\\
[x_0,x_1,\cdots,x_n;f]&=\frac{[x_1,\cdots,x_n;f]-[x_0,x_1,\cdots,x_{n-1};f]}{x_n-x_0}.
\end{aligned}
\end{equation*}
There is also an explicit formula for the  divided difference:
\begin{equation*}
[x_0,x_1,\cdots,x_n;f]=\sum_{j=0}^n \frac{f(x_j)}{\prod_{j\neq l} (x_j-x_l)}.
\end{equation*}
 The {\it Hermite-Genocchi formula} gives an integral represenation for the divided differences of  
an $n$ times continuously differentiable function $f$ as an integral over  the standard simplex:
\begin{equation}\label{HermiteGenocchi}
[x_0,\cdots,x_n;f]=\int_{\Sigma_n} f^{(n)}\Big(\sum_{j=0}^n s_j x_j\Big)ds.
\end{equation}

Let $\delta$ be a densely defined, unbounded and  closed derivation on a $C^\ast$-algebra $A$.
If $a\in {\rm Dom}(\delta)$, then $e^{za}\in {\rm Dom}(\delta)$ for any $z\in \mathbb{C}$,   and one has 
\begin{equation}\label{generalduahmmel}
\delta(e^{za})=z\int_0^1 e^{zsa}\delta(a) e^{z(1-s)a}ds.
\end{equation}
Using the rearrangement lemma, one can now express the differential of an smooth element  given in contraction form. This result generalizes the expansional
formula, also known as Feynman-Dyson formula,  for $e^{A+B}$, 
 and 
 not only for elements of the form $f(h)$,  but for any element  written in the contraction form.

\begin{theorem}\label{derivationoffunctionofh}
Let  $\delta$ be a closed derivation of a $C^\ast$-algebra $A$ and $h\in {\rm Dom}(\delta)$ be a selfadjoint element. 
Let $b_j\in {\rm Dom}(\delta)$,   $1\leq j \leq n$, and let $f:\mathbb{R}^{n+1}\to \mathbb{C}$ be a smooth function.
Then $f(h_{(0)},\cdots, h_{(n)})(b_1\cdot b_2\cdot \dots\cdot b_n)$ is in the domain of $\delta$  and  
\begin{equation*}
\begin{aligned}
&\delta(f(h_{(0)},\cdots,h_{(n)})(b_1\cdot b_2 \cdots  b_n))\\
&=\sum_{j=1}^{n} f(h_{(0)},\cdots,h_{(n)})\big(b_1\cdots b_{j-1}\cdot\delta(b_j)\cdot b_{j+1} \cdots  b_n\big)\\
&\quad +\sum_{j=0}^n f_j(h_{(0)},\cdots,h_{(n+1)})\big(b_1\cdots b_{j}\cdot\delta(h)\cdot b_{j+1} \cdots  b_n\big),
\end{aligned}
\end{equation*}
where $f_j(t_0,\cdots,t_{n+1})$, which we call the partial divided difference, is defined as 
$$f_j(t_0,\cdots,t_{n+1})=\left[t_j,t_{j+1};t\mapsto f(t_0,\cdots, t_{j-1},t,t_{j+2},\cdots,t_n)\right].$$\\
\end{theorem}

\subsection{Laplace type $h$-differential operators and asymptotic expansions}\label{rawcomp} Let us first recall this class of differential operators  which is introduced in \cite{arXiv1811.04004}.  It  extends the previous classes of differential operators on noncommutative tori, in particular Dirac Laplacians of conformally flat metrics. 

\begin{definition}[\cite{arXiv1811.04004}]
Let $h\in C^{\infty}(\mathbb{T}_{\theta}^d)$ be a smooth selfadjoint element.
\begin{itemize}
\item[(i)] An {\it $h$-differential operator} on $ \mathbb{T}_{\theta}^d$, is a differential operator 
$P=\sum_{\vec{\alpha}}p_{\vec{\alpha}}\delta^{\vec{\alpha}},$ 
with $ C^{\infty}(\mathbb{T}_{\theta}^d)$-valued coefficients $p_{\vec{\alpha}}$ which can be written in the contraction form 
$$p_{\vec{\alpha}}=P_{\vec{\alpha},\vec{\alpha}_1,\cdots,\vec{\alpha}_k}(h_{(0)},\cdots,h_{(k)})(\delta^{\vec{\alpha}_1}(h)\cdots \delta^{\vec{\alpha}_k}(h)).$$
\item[(ii)] A second order $h$-differential operator  $P$ is called a {\it Laplace type $h$-differential operator} if its symbol is a sum of homogeneous parts $p_j$ of the form
\begin{equation*}\label{noncommutativelaplacetypesym}
\begin{aligned}
p_2&=P_2^{ij}(h)\xi_i\xi_j,\\
p_1&=  P_1^{ij}(h_{(0)},h_{(1)})(\delta_i(h)) \xi_j,\\
p_0&= P_{0,1}^{ij}(h_{(0)},h_{(1)})(\delta_i\delta_j(h))+P_{0,2}^{ij}(h_{(0)},h_{(1)},h_{(2)})(\delta_i(h)\cdot \delta_j(h)),
\end{aligned}
\end{equation*}
where the principal symbol $p_2:\mathbb{R}^n\to   C^{\infty}(\mathbb{T}_{\theta}^d)$ is a $C^{\infty}(\mathbb{T}_{\theta}^d)$-valued quadratic form such that $p_2(\xi)>0$ for all $\xi\in\mathbb{R}$. 
\end{itemize} 
\end{definition}

\noindent We can allow  the symbols to be matrix valued, that is  $p_j:   C^{\infty}(\mathbb{R}^d)\to       C^{\infty}(\mathbb{T}_{\theta}^d) \otimes M_n(\mathbb{C})$, provided that all $p_2(\xi)\in   C^{\infty}(\mathbb{T}_{\theta}^d)\otimes {\rm I}_n$ for all nonzero $\xi\in\mathbb{R}$.

Many of the elliptic second order differential operators on noncommutative tori which  were studied in the literature  are Laplace type $h$-differential operators. 
 For instance, the two differential operators on $\nctorus[2]$ whose spectral invariants are studied in 
 \cite{MR3194491, MR3148618} are indeed Laplace type $h$-differential operator.
 In fact with  $k=e^{h/2}$, 
these operators are given by
\begin{equation*}
k\lap k=k\delta\delta^* k,\qquad \lap_\varphi^{(0,1)}=\delta^* k^2 \delta,
\end{equation*}
where, $\delta=\delta_1+\bar{\tau}\delta_2$ and  $\delta^*=\delta_1+\tau\delta_2$ for some complex number $\tau$ in the upper half plane.

Let  $P$ be a positive Laplace type $h$-differential operator. 
 Using the Cauchy integral formula, one has 
\begin{equation*}
e^{-tP}=\frac{-1}{2\pi i}\int_\gamma e^{-t\lambda}(P-\lambda)^{-1}d\lambda,\quad t>0,
\end{equation*}
for a suitable contour $\gamma$. 
Expanding  the symbol of the paramatrix $\sigma((P-\lambda)^{-1})$, one obtains a short time  asymptotic expansion  for  localized heat trace  for any $  a\in C^{\infty}(\mathbb{T}_{\theta}^d)$:
\begin{equation*}
{\rm Tr}(ae^{-tP})\sim \sum_{n=0}^\infty c_n(a) t^{(n-\dim)/2}.
\end{equation*} 
Here,  $c_n(a) = \varphi(a b_n)$  with
\begin{equation}\label{localterminbjs}
b_n =\frac1{(2\pi)^\dim}\int_{\mathbb{R}^\dim}\frac{-1}{2\pi i}\int_\gamma b_n(\xi,\lambda)d\lambda d\xi.
\end{equation} 
Using the rearrangement lemma  (Lemma \ref{rearrangmentlemma}) and the fact that the contraction map and integration commute,
one obtains 

\begin{equation*}
\begin{aligned}
b_2 =&\left(\frac1{(2\pi)^\dim}\int_{\mathbb{R}^\dim}\frac{-1}{2\pi i}\int_\gamma B_{2,1}^{ij}(\xi,\lambda,h_{(0)},h_{(1)})\, e^{-\lambda} d\lambda d\xi\right)\big(\delta_i\delta_j(h)\big)\\
&+\left(\frac1{(2\pi)^\dim}\int_{\mathbb{R}^\dim}\frac{-1}{2\pi i}\int_\gamma B_{2,2}^{ij}(\xi,\lambda,,h_{(0)},h_{(1)},,h_{(2)})\, e^{-\lambda} d\lambda d\xi\right)\big(\delta_i(h)\cdot\delta_j(h)\big).
\end{aligned}
\end{equation*} 
The dependence of $B_{2,k}^{ij}$ on  $\lambda$ comes only from different powers of $B_0$ in its terms, while its dependence on  $\xi_j$'s is the result of appearance of $\xi_j$  as well as of $B_0$ in the terms. 
Therefore, the contour integral will only contain $e^{-\lambda}$ and 
product of powers of $B_0(t_j)$.
Hence, we need to deal with a certain kind of contour integral for which  we shall use the following notation and will call them {\it $T$-functions}:
\begin{equation}\label{Tfunctionsdef}
T_{\vec{n};\vec{\alpha}}(t_0,\cdots,t_n):=\frac{-1}{\pi^{\dim/2}}\int_{\mathbb{R}^{\dim}}\xi_{n_1}\cdots \xi_{n_{2|\vec{\alpha}|-4}}\frac{1}{2\pi i}\int_\gamma e^{-\lambda} B_0^{\alpha_0}(t_0)\cdots B_0^{\alpha_n}(t_n)  d\lambda d\xi,
\end{equation}
where $\vec{n}=(n_1,\cdots,n_{2|\vec{\alpha}|-4})$ and $\vec{\alpha}=(\alpha_0,\cdots,\alpha_n)$.
We recall the  $T$-functions and their properties a bit later.
There is an explicit formula for $b_2(P)$ which we now recall from \cite{arXiv1811.04004}:
\begin{proposition}\label{b2proposition}
For a positive   Laplace type $h$-differential operator $P$ with the symbol given by \eqref{noncommutativelaplacetypesym},  the term $b_2(P)$ in the contraction form is given by
\begin{equation*}
b_2(P)=(4\pi)^{-\dim/2}\left(B_{2,1}^{ij}(h_{(0)},h_{(1)})\big(\delta_i\delta_j(h)\big)+B_{2,2}^{ij}(h_{(0)},h_{(1)},h_{(2)})\big(\delta_i(h)\cdot\delta_j(h)\big)\right),
\end{equation*}
where the functions are defined by   
\begin{align*}
B_{2,1}^{ij}(t_0,t_1)
=& -T_{;1,1}(t_0,t_1)P_{0,1}^{ij}(t_0,t_1)
+2T_{k \ell;2,1}(t_0,t_1) P_2^{ik}(t_0)P_1^{j\ell}(t_0,t_1)\\
&+T_{k \ell;2,1}(t_0,t_1)\, P_2^{ij}(t_0) \,\big[t_0,t_1;P_2^{k\ell}\big]\\
&-4T_{k\ell  m n;3,1}(t_0,t_1) P_2^{ik}(t_0) P_2^{j\ell}(t_0)\,\big[t_0,t_1;P_2^{mn}\big],
\end{align*}
and
\begin{align*}
B_{2,2}^{ij}(t_0,t_1,t_2)=
&-T_{;1,1}(t_0,t_2) P_{0,2}^{ij}(t_0,t_1,t_2)\\
&+T_{k\ell;1,1,1}(t_0,t_1,t_2)\,   P_1^{ik}(t_0,t_1)P^{j\ell}_1(t_1,t_2)\\
&-2T_{k\ell mn;2,1,1}(t_0,t_1,t_2)\, P_2^{im}(t_0)\,\big[t_0,t_1;P_2^{k\ell}\big]\, P^{jn}_1(t_1,t_2) \\
&+2T_{k\ell;2,1}(t_0,t_2)\,  P_2^{ik}(t_0)\,\big[t_0,t_1;P_1^{j\ell}(\cdot,t_2)\big]\,\\
&+2T_{k\ell;2,1}(t_0,t_2)\, P_2^{jk}(t_0)\,\big[t_1,t_2;P_1^{i\ell}(t_0,\cdot)\big]\,\\
&+T_{k\ell;1,1,1}(t_0,t_1,t_2)\, P_1^{ij}(t_0,t_1)\,\big[t_1,t_2;P_2^{k\ell}\big]\, \\
& -2T_{k\ell mn;2,1,1}(t_0,t_1,t_2)\, P_2^{j\ell}(t_0) P_1^{ik}(t_0,t_1)\,\big[t_1,t_2;P_2^{mn}\big]\,  \\
&-2T_{k\ell mn;1,2,1}(t_0,t_1,t_2)\,  P_1^{ik}(t_0,t_1)P_2^{j\ell}(t_1)\,\big[t_1,t_2;P_2^{mn}\big]\,  \\
&- 2T_{k\ell mn;2,1,1}(t_0,t_1,t_2) \,  P_2^{ij}(t_0)\,\big[t_0,t_1;P_2^{\ell m}\big]\,\big[t_1,t_2;P_2^{kn}\big]\, \\ 
&- 4T_{k\ell mn;2,1,1}(t_0,t_1,t_2)\, P_2^{ik}(t_0)\,\big[t_0,t_1;P_2^{ j\ell}\big]\,\big[t_1,t_2;P_2^{mn}\big]\, \\
&+ 8T_{k\ell mnpq;3,1,1}(t_0,t_1,t_2)\, P_2^{ik}(t_0) P_2^{jn}(t_0)\,\big[t_0,t_1;P_2^{\ell m}\big]\,\big[t_1,t_2;P_2^{pq}\big]\, \\
&+ 4T_{k\ell m n p q;2,2,1}(t_0,t_1,t_2) P_2^{ik}(t_0)\,\big[t_0,t_1;P_2^{\ell m}\big]\,  P_2^{jn}(t_1)\,\big[t_1,t_2;P_2^{pq}\big]\\
&+2T_{k\ell;2,1}(t_0,t_2)\, P_2^{ij}(t_0)\,\big[t_0,t_1,t_2;P_2^{k\ell}\big]\,\\ 
&-8T_{k\ell mn;3,1}(t_0,t_2)\, P_2^{ik}(t_0) P_2^{j\ell}(t_0)[t_0,t_1,t_2;P_2^{mn}\big].
\end{align*} \qed
\end{proposition}

The computation of the higher heat trace densities for a Laplace type $h$-operator can be similarly carried out,  expecting many more terms in the results.
This would give a way to generalize results obtained for the conformally flat noncommutative two torus in         \cite{arXiv1611.09815}  where $b_4$ of the Laplacian $D^2$ of the Dirac operator $D$ is computed.  This problem won't be discussed further in this paper, but is certainly an interesting problem. 

Evaluating $T$-functions \eqref{Tfunctionsdef},  the only parts of formulas for $B_{2,1}^{i,j}$ and $B_{2,2}^{i,j}$ that need to be evaluated, is not always an easy task.
In \cite{arXiv1811.04004} a concise integral formula for $T$-functions  is given and   their properties is studied.  
For the contour integral in \eqref{Tfunctionsdef}. 
it is clear  that there are functions $f_{\alpha_1,\cdots,\alpha_{n}}$ such that
\begin{equation*}
\frac{-1}{2\pi i}\int_\gamma e^{-\lambda} B_0^{\alpha_0}(t_0)\cdots B_0^{\alpha_n}(t_n)  d\lambda
=f_{\alpha_0,\cdots,\alpha_{n}}\left(\|\xi\|_{t_0}^{2},\cdots,\|\xi\|_{t_n}^{2}\right).
\end{equation*}
Here, we denoted $P_2^{ij}(t_k)\xi_i\xi_j$ by $\|\xi\|_{t_k}^2$.
Examples of such functions are
\begin{equation*}
f_{1,1}(x_0,x_1)=-\frac{e^{-x_0}}{x_0-x_1}-\frac{e^{-x_1}}{x_1-x_0},\qquad 
f_{2,1}(x_0,x_1)=-\frac{e^{-x_0}}{x_0-x_1}- \frac{e^{-x_0}}{(x_0-x_1)^2}+\frac{e^{-x_1}}{(x_1-x_0)^2}.
\end{equation*}

\begin{lemma} [\cite{arXiv1811.04004}] \label{integralformofTna}
Let $P_2(t)$ be a positive definite $\dim\times \dim$ matrix of smooth real functions. 
Then
\begin{equation*}
T_{\vec{n};\vec{\alpha}}(t_0,\cdots,t_n)=\frac{1}{2^{|\vec{\alpha}|-2}\vec{\beta}!}\int_{\Sigma_n}\prod_{j=0}^ns_j^{\alpha_j-1} \, \frac{\underset{\vec{n}}{\sum \prod} P^{-1}(s)_{n_in_{\sigma(i)}}}{\sqrt{\det P(s)}}  ds, 
\end{equation*}
where $P(s)=\sum_{j=0}^n s_j P_2(t_j)$  and $\vec{\beta}=(\alpha_0-1,\cdots,\alpha_{n}-1)$. \qed \\
\end{lemma}

\subsection{Functional metrics and scalar curvature}
A natural question is if there exists a  large  class of noncommutative metrics whose Laplacians  are  $h$-differential operators and hence amenable to the spectral analysis developed  in the last section?  As we saw, conformally flat metrics on noncommutative tori is such a class. But there are more. One of the interesting concepts develped in \cite{arXiv1811.04004} is the notion of a {\it functional  metric} which is a much larger class than conformally flat metrics and whose Laplacian is still an $h$-differential operator. In this section we shall first recall this concept and reproduce the scalar curvature formula for these metrics developed in \cite{arXiv1811.04004}.

\begin{definition}
Let $h$ be a selfadjoint smooth element of a  noncommutative $ d$-torus
and let $g_{ij}:\mathbb{R}\to \mathbb{R}$,  $1\leq i,j\leq \dim$, be smooth functions such that the matrix $\big(g_{ij}(t)\big)$ is a positive definite matrix for every $t$ in a neighborhood of the spectrum of $h$.
We shall refer to $g_{ij}(h)$ as a {\it functional metric} on $\A^{d}$.
\end{definition}

The  construction of  the Laplacian on functions on $\nctorus[\dim]$ equipped with a functional metric $g=g_{ij}(h)$, follows the same pattern as in  previous sections. Details can be found in \cite{arXiv1811.04004}, where the following crucial result  is also proved. 
The Laplacian  $\delta^*\delta:\mathcal{H}_{0,g}\to \mathcal{H}_{0,g}$ on elements of $  C^{\infty}(\mathbb{T}_{\theta}^d)$ is given by
\begin{equation*}
\delta_j(a)g^{jk}(h)\delta_k(|g|^{\frac12}(h))|g|^{-\frac12}(h)+\delta_j(a)\delta_k(g^{jk}(h))+i\delta_k(\delta_j(a))g^{jk}(h).
\end{equation*}
To carry the spectral analysis of the Laplacian $\delta^*\delta:\mathcal{H}_{0,g}\to \mathcal{H}_{0,g}$, we switch to the antiunitary equivalent setting as follows.
Let $\mathcal{H}_0$ be the Hilbert space obtained by the GNS construction from $\A^d$ using the nonperturbed tracial state $\varphi$.

  \begin{proposition}\label{laplacianforg}
The operator $\delta^*\delta:\mathcal{H}_{0,g}\to \mathcal{H}_{0,g}$ is antiunitary equivalent to  a  Laplace type $h$-differential operator $\lap_{0,g}:\mathcal{H}_0\to \mathcal{H}_0$ whose symbol, when expressed in the contraction form, has the functional parts given by 
\begin{align*}
&P_2^{jk}(t_0)= g^{jk}(t_0),\\
&P_1^{jk}(t_0,t_1)
=|g|^{-\frac14}(t_0)\big[t_0,t_1;|g|^{\frac14}\big]g^{jk}(t_1)+\big[t_0,t_1;g^{jk}\big]+|g|^{\frac14}(t_0)g^{jk}(t_0)\big[t_0,t_1;|g|^{-\frac14}\big],\\
&P_{0,1}^{jk}(t_0,t_1)
=|g|^{\frac14}(t_0)g^{jk}(t_0)\big[t_0,t_1;|g|^{-\frac14}\big],\\
&P_{0,2}^{jk}(t_0,t_1,t_2)
=|g|^{-\frac14}(t_0)\big[t_0,t_1;g^{jk}|g|^{\frac12}\big] \big[t_1,t_2;|g|^{-\frac14}\big]+2|g|^{\frac14}(t_0)g^{jk}(t_0)\big[t_0,t_1,t_2;|g|^{-\frac14}\big].
\end{align*}
 \end{proposition}

An important case of the functional   metric  is the conformally flat metric
\begin{equation}\label{conflatmetric}
g_{ij}(t)=f(t)^{-1}g_{ij},
\end{equation}
 where $f$ is a positive smooth function and $g_{ij}$'s are the entries of a constant metric on $\mathbb{R}^\dim$.
The functions given by Proposition \ref{laplacianforg}, for the conformally flat metrics, gives us the following:
\begin{equation}\label{symfunconformal}
\begin{aligned}
&P_2^{jk}(t_0)= g^{jk}f(t_0),\\
&P_1^{jk}(t_0,t_1)
=g^{jk}\Big(f(t_0)^{\frac{\dim}4}\big[t_0,t_1;f^{1-\frac{\dim}4}\big]+f(t_0)^{1-\frac{\dim}4}\big[t_0,t_1;f^{\frac{\dim}4}\big]\Big),\\
&P_{0,1}^{jk}(t_0,t_1)=g^{jk}f(t_0)^{1-\frac{\dim}4}\big[t_0,t_1;f^{\frac{\dim}4}\big],\\
&P_{0,2}^{jk}(t_0,t_1,t_2)= 
g^{jk}\Big(f(t_0)^{\frac{\dim}4}\big[t_0,t_1;f^{1-\frac{\dim}2}\big]\big[t_1,t_2;f^{\frac{\dim}4}\big]+2f(t_0)^{1-\frac{\dim}4}\big[t_0,t_1,t_2;f^{\frac{\dim}4}\big]\Big).
\end{aligned}
\end{equation}
A careful examination of  formula \eqref{symfunconformal} shows that  for any  function $P^{ij}_{\bullet}$ there exist a function $P_{\bullet}$ such that 
$P^{ij}_{\bullet}=g^{ij}P_{\bullet}.$
We have similar situation with the $T$-functions for conformally flat metrics. 
\begin{lemma}[\cite{arXiv1811.04004}]\label{Tfunctionsconfomallyflatlemma}
Let  $\vec{\alpha}$ and $\vec{n}=(n_1,\cdots,n_{2|\vec{\alpha}|-4})$ be two multi-indicies.
Then the $T$-function $T_{\vec{n,\alpha}}$ for the conformally flat metric \eqref{conflatmetric} is of the form  
$$T_{\vec{n,\alpha}}(t_0,\cdots,t_n)=\sqrt{|g|}\underset{\vec{n}}{\sum \prod}g_{n_in_{\sigma(i)}} T_{\vec{\alpha}}(t_0,\cdots,t_n).$$ 
The function $T_{\vec{\alpha}}$ in dimension $d\neq 2$ is given by 
\begin{equation}\label{Tfunctionsconfomallyflat}
T_{\vec{\alpha}}(t_0,\cdots,t_n)
=\frac{(-1)^{|\vec{\alpha}|-1}\Gamma(\frac{\dim}2-1)}{\Gamma(\frac{\dim}2+|\vec{\alpha}|-2)} \left.\partial_x^{\vec{\beta}} \big[x_0,\cdots,x_n;u^{1-\frac{d}2}\big]\right|_{x_j=f(t_j)},
\end{equation}
where $\vec{\beta}=(\alpha_0-1,\cdots,\alpha_{n}-1)$.
\end{lemma}

As an example, we have
\begin{equation*}
\begin{aligned}
T_{\alpha,1}(t_0,t_1)=&\frac{(-1)^{\alpha}\Gamma(\frac{\dim}2-1)}{2^{\alpha-1}\Gamma(\frac{\dim}2+\alpha-1)}\times \\
&\Big(\frac{f(t_1)^{1-\frac{\dim}2}}{(f(t_1)-f(t_0))^{\alpha}}-\sum_{m=0}^{\alpha-1}\frac{(-1)^m \Gamma(\frac{\dim}2+m-1)}{\Gamma(\frac{\dim}2-1)m!} \frac{f(t_0)^{-\frac{\dim}2-m+1}}{(f(t_1)-f(t_0))^{\alpha-m}}\Big).
\end{aligned}
\end{equation*} 

Note that  for dimension two, $T_{\alpha,1}(t_0,t_1)$ can be obtained by taking the limit of \eqref{Tfunctionsconfomallyflat} as  $\dim$ approaches 2. 
When $f(t)=t$, we have  
\begin{equation*}
T_{\alpha,1}(t_0,t_1)=\frac{(-1)^{\alpha-1}}{2^{\alpha-1} \Gamma(\alpha)^2} \partial_{t_0}^{\alpha-1}\big[t_0,t_1;\log(u)\big].
\end{equation*}

 Recall that  the {\it scalar curvature density} of a given functional metric is defined by  
\begin{equation*}\label{scalarcurvaturedef}
 R =(4\pi)^{\frac\dim2}b_2(\lap_{0,g}).
\end{equation*}    
This scalar curvature density is computed  for two classes of examples  in all dimensions: conformally flat metrics  and twisted products of conformally  flat metrics. Let us recall this result:

\begin{theorem}[\cite{arXiv1811.04004}]\label{scalarcrvatureconformal}
The scalar curvature of the $\dim$-dimensional noncommutative tori $\mathbb{T}_\theta^\dim$ equipped with the metric $f(h)^{-1}g_{ij}$ is given by
\begin{equation*}
R=\sqrt{|g|} \Big(K_{\dim}(h_{(0)},h_{(1)})(\triangle(h))+H_{\dim}(h_{(0)},h_{(1)},h_{(2)})(\Box(h))\Big),
\end{equation*}
where  $\lap(h)=g^{ij}\delta_i\delta_j(h)$, $\Box(h)=g^{ij}\delta_i(h)\cdot\delta_j(h)$. 
The functions $K_d$ and $H_d$ are given by
\begin{equation}\label{HKfunctions}
\begin{aligned}
K_{\dim}(t_0,t_1)=&K_{\dim}^t(f(t_0),f(t_1))\big[t_0,t_1;f\big],\\
H_{\dim}(t_0,t_1,t_2)=&H_{\dim}^t(f(t_0),f(t_1),f(t_2))\big[t_0,t_1;f\big]\big[t_1,t_2;f\big]\\
&+2K_{\dim}^t(f(t_0),f(t_2))\big[t_0,t_1,t_2;f\big],
\end{aligned}
\end{equation}
Where $K^t_d$ and $H^t_d$ are the functions $K_d$ and $H_d$ when $f(t)=t$.
For $\dim\neq 2$, they can be computed to be
\begin{align*}
K_{\dim}^t(x,y)=\frac{4\, x^{2-\frac{3 \dim }{4}} y^{2-\frac{3 \dim }{4}} }{\dim(\dim -2) (x-y)^3}\left((\dim -1) x^{\frac\dim 2} y^{\frac{\dim }{2}-1}-(\dim -1) x^{\frac{\dim }{2}-1} y^{\frac \dim 2}-x^{\dim -1}+y^{\dim -1}\right),
\end{align*}
and
\begin{align*}
H_d^t(x,y,z)=&\frac{2 x^{-\frac{3 \dim }{4}} y^{-\dim } z^{-\frac{3 \dim }{4}}}{(\dim -2) \dim  (x-y)^2 (x-z)^3 (y-z)^2}\times\\
 \Big(
 & x^{\dim } y^{\dim } z^2(x-y) \left(3 x^2 y-2 x^2 z-4 x y^2+4 x y z-2 x z^2+y z^2\right)
\\
 &+x^{\dim } y^{\frac{\dim }{2}+1} z^{\frac{\dim }{2}+1} (x-z)^2 (z-y) (\dim  x+(1-\dim ) y)
\\
 &
 +x^{\dim }y^3 z^{\dim } (z-x)^3
 + x^{\frac{\dim }{2}+1} y^{\frac{3 \dim }{2}}z^2 (x-y) (x-z)^2
\\
 &
 +2 (\dim -1) x^{\frac{\dim }{2}+1} y^{\dim } z^{\frac{\dim }{2}+1} (x-y) (x-z) (z-y) (x-2 y+z)
\\
 &-x^{\frac{\dim }{2}+1} y^{\frac{\dim }{2}+1} z^{\dim } (x-y) (x-z)^2 ((1-\dim ) y+\dim  z)
 -x^2 y^{\frac{3 \dim }{2}} z^{\frac{\dim }{2}+1} (x-z)^2 (z-y)
\\
 &
 +x^2 y^{\dim } z^{\dim } (y-z) \left(x^2 y-2 x^2 z+4 x y z-2 x z^2-4 y^2 z+3 y z^2\right)\Big).
\end{align*}
These functions for the dimension two are given by
\begin{align*}
K_2^t(x,y)=&-\frac{\sqrt{x} \sqrt{y} }{(x-y)^3}((x+y) \log (x/y)+2(y-x)),\\
H_2^t(x,y,z)=&\frac{2 \sqrt{x} \sqrt{z}}{(x-y)^2 (x-z)^3 (y-z)^2}\times\\
&\Big(-(x-y) (x-z) (y-z) (x-2 y+z) +y (x-z)^3 \log (y)  \\
&\,\,+(y-z)^2 (-2 x^2+x y+y z)\log (x) -(x-y)^2 (x y+zy-2 z^2)\log (z))\Big).
\end{align*}
\end{theorem}

Note that the Function $K_\dim^t(x,y)$ is the symmetric part of the function
\begin{equation*}
\frac{8 x^{2-\frac{ \dim }{4}} y^{2-\frac{3 \dim }{4}}}{\dim(\dim -2)   (x-y)^3} \left((\dim -1)  y^{\frac{\dim }{2}-1}-x^{\frac{\dim}2 -1}\right).
\end{equation*}
Similarly, $H_\dim^t(x,y,z)$ is equal to $(F_d(x,y,z)+F_d(z,y,x))/2$ where
\begin{align*}
F_\dim(x,y,z)=
&\frac{4 x^{-\frac{3 \dim }{4}} y^{-\dim } z^{-\frac{3 \dim }{4}}}{\dim(\dim -2)   (x-y)^2 (x-z)^3 (y-z)^2}\times
\\
 &\Big(
 x^{\dim } y^{\dim } z^2 (x-y) (3 x^2 y-2 x^2 z-4 x y^2+4 x y z-2 x z^2+y z^2)
\\
 &\, \, +x^{\dim } y^{\frac{\dim }{2}+1} z^{\frac{\dim }{2}+1} (x-z)^2 (z-y) (\dim  x+(1-\dim ) y)
 +\frac{1}{2}  x^{\dim } y^3 z^{\dim } (z-x)^3
\\
  &\,\, +x^{\frac{\dim }{2}+1} y^{\frac{3 \dim }{2}} z^2 (x-y) (x-z)^2
 +2 (\dim -1) x^{\frac{\dim }{2}+1} y^{\dim } z^{\frac{\dim }{2}+1} (x-y)^2 (x-z) (z-y)
\Big).
\end{align*}

As we recalled in earlier sections, in low dimensions   two, three, and four,   the curvature of the conformally flat metrics was  computed  in \cite{MR3194491,  MR3148618, MR3369894, MR3359018, arXiv1808.02977}.
It is shown in \cite{arXiv1811.04004} that the above general formula reproduces those results.  
We should first note that the functions found in all the aforementioned works are written in terms of  the commutator $[h,\cdot]$, denoted by $\Delta$. 
To produce those functions from our result,  a linear  substitution of the variables $t_j$ in terms of new variables $s_j$  is needed.  
On the other hand, it is important to  note that the functions $K_d^t(x,y)$ and $H_d^t(x,y,z)$ are homogeneous rational functions  of order $-\frac{d}2$ and  $-\frac{\dim}2-1$ respectively.
Using formula \eqref{HKfunctions}, it is clear that the functions $K_d(t_0,t_1)$ and $H_d(t_0,t_1,t_2)$ are homogeneous of order $1-\frac{\dim}2$ in $f(t_j)$'s.
This is the reason that for function $f(t)=e^t$  and a linear substitution such as  $t_j=\sum_{m=0}^js_m$, a factor of some power of $e^{s_0}$ comes out.
This term can be replaced by a power of $e^h$ multiplied from the left to the final outcome.
This explains how the functions in the aforementioned papers have one less variable than our functions.
In other words, we have
\begin{equation*}
K_{\dim}(s_0,s_0+s_1)=e^{(1-\frac{d}{2})s_0}K_{\dim}(s_1),\qquad 
H_{\dim}(s_0,s_0+s_1,s_0+s_1+s_2)=e^{(1-\frac{d}{2})s_0}H_{\dim}(s_1,s_2). 
\end{equation*} 
For instance, function $K_d(s)$ is given by
\begin{equation*}
K_{\dim}(s_1)=\frac{8 e^{\frac{\dim+2}{4} s_1} \left((\dim-1) \sinh \left(\frac{s_1}{2}\right)+\sinh \left(\frac{(1-\dim)s_1}{2}\right)\right)}{\dim(\dim-2) d \left(e^{s_1}-1\right)^2 s_1}.
\end{equation*}

Now, we can obtain functions in  dimension two:
\begin{equation*}
\begin{aligned}
H_2(s_1)=&-\frac{e^{\frac{s_1}{2}} \left(e^{s_1} \left(s_1-2\right)+s_1+2\right)}{\left(e^{s_1}-1\right){}^2 s_1},\\
K_2(s_1,s_2)=&\Big(s_1(s_1\hspace{-0.5mm}+\hspace{-0.5mm}s_2)\cosh(s_2)\hspace{-0.5mm}-(s_1\hspace{-0.5mm}-\hspace{-0.5mm}s_2)\left(s_1+s_2+\sinh(s_1)+\sinh(s_2)-\sinh(s_1+s_2)\right)\\
&
\quad -s_2(s_1+s_2) \cosh(s_1)\Big){\rm csch}(\frac{s_1}{2}) {\rm csch}(\frac{s_2}{2}){\rm csch}^2(\frac{s_1+s_2}{2})/(4 s_1 s_2 (s_1+s_2)).
\end{aligned}
\end{equation*}
we have $-4H_2=H$ and $-2K_2=K$ where $K$ and $H$ are the functions found in \cite{MR3194491, MR3148618} . 
The difference   is coming from the fact that the noncommutative parts of the results in \cite[Section 5.1]{MR3369894} are $\lap(\log(e^{h/2}))=\frac12\lap(h)$ and $\Box(\log(e^{h/2}))=\frac14\Box(h)$.

 The functions for dimension four, with the same conformal factor $f(t)=e^t$  and substitution $t_j=\sum_{m=0}^js_m$, gives the following  which up to a negative sign are in complete agreement with the results from  our papers \cite{MR3359018, MR3369894}:
 \begin{equation*}
\begin{aligned}
H_4(s_1)=\frac{1-e^{s_1}}{2e^{s_1} s_1},\quad 
K_4(s_1,s_2)=\frac{\left(e^{s_1}-1\right) \left(3 e^{s_2}+1\right) s_2-\left(e^{s_1}+3\right) \left(e^{s_2}-1\right) s_1}{4 e^{s_1+s_2}s_1 s_2 \left(s_1+s_2\right)}.
\end{aligned}
\end{equation*}

To recover the functions for curvature of a noncommutative three torus equipped with a conformally flat metric obtained in \cite{arXiv1610.04740, arXiv1808.02977}, we need to set  $f(t)=e^{2t}$ and $t_0=s_0,\, t_1=s_0 + s_1/3$ and $t_2=s_0 + (s_1+s_2)/3$. 
Then up to a factor of $e^{-s_0}$, we have
 \begin{equation*}
\begin{aligned}
H_3(s_1)=\frac{4-4e^{\frac{s_1}{3}}}{e^{\frac{s_1}{6}}(s_1 e^{\frac{s_1}{3}}+1) },\quad 
K_3(s_1,s_2)=\frac{6(e^{\frac{s_1}{3}}-1) (3 e^{\frac{s_2}{3}}+1) s_2-6 (e^{\frac{s_1}{3}}+3)(e^{\frac{s_2}{3}}-1) s_1}{ e^{\frac{s_1+s_2}6}(e^{\frac{s_1+s_2}{3}}+1) s_1 s_2 (s_1+s_2)}.
\end{aligned}
\end{equation*}

Finally, one needs to  check the classical limit of these formulas  as $\theta\to 0$.
In the commutative case,  the scalar curvature of a conformally flat metric   $\tilde{g}=e^{2h}g$ on a $\dim$-dimensional space reads 
\begin{equation*} 
 {\tilde {R}}=-2(\dim-1)e^{-2h}g^{jk}\partial_j \partial_k(h) -(\dim-2)(\dim-1)e^{-2h} g^{jk}\partial_j(h)\partial_k(h).
 \end{equation*}
For  $f(t)=e^{-2t}$,  the limit is
\begin{equation*}
\begin{aligned}
\lim_{t_0,t_1\to t} K_{\dim}(t_0,t_1)&=\frac{1}{3} (d-1) e^{(\dim-2 )t},\quad \lim_{t_0,t_1,t_2\to t}H_{\dim} (t_0,t_1,t_2)&=\frac{1}{6} (\dim-2) (\dim-1) e^{ (\dim-2) t}.
\end{aligned}
\end{equation*}
We should also add that since $\delta_j\to -i \partial_j$ as $\theta\to 0$, we have $\lap(h)\to -g^{jk}\partial_j \partial_k(h) $ and $\Box(h)\to -g^{jk}\partial_j(h)\partial_k(h)$.
Therefore, these results recovers the classical result up to a  factor of $\sqrt{|g|}e^{\dim h}/6 $.
The factor $\sqrt{|g|}e^{\dim h}$ represents the volume form in the scalar curvature density and the factor $1/6$ is due to  the choice of normalization in \eqref{scalarcurvaturedef}.

\subsection{Twisted product, warped product, and   scalar curvature}\label{twistedmetricsec}
In this  section, following \cite{arXiv1811.04004}, we shall recall the computation of the   scalar  curvature density of a noncommutative $\dim$-torus equipped with a class of functional metrics, which  
is  called a twisted product metric.
\begin{definition}[\cite{arXiv1811.04004}] 
Let $g$ be an $r\times r$ and $\tilde{g}$ be a $(\dim-r)\times (\dim-r)$ positive definite real symmetric  matrices and assume $f$ is a positive function on the real line. 
The functional metric 
\begin{equation}\label{twistedmetric}
G=f(t)^{-1}g\oplus \tilde{g},
\end{equation}
is called a  {\it twisted product functional metric} with the twisting element $f(h)^{-1}$.
\end{definition}
Some examples of the twisted product metrics on noncommutative tori were already studied.
The asymmetric two torus whose Dirac oeprator and spectral invariants are studied  in \cite{MR3402793} \ is a twisted product metric for $r=1$.
The scalar and Ricci curvature of noncommutative three torus of twisted product metrics with $r=2$ are studied in \cite{arXiv1808.02977}.
It is worth mentioning  that conformally flat metrics as well as warped  metrics  are two special cases of twisted product functional metrics. The following theorem is proved in \cite{arXiv1811.04004}.

\begin{theorem}\label{scalartwisted}
The scalar curvature density of the $\dim$-dimensional noncommutative tori $\nctorus[\dim]$ equipped with the twisted product functional metric \eqref{twistedmetric} with the twisting element $f(h)^{-1}$   is given by
\begin{equation*}
\begin{aligned}
R=\sqrt{|g||\tilde{g}|} \Big(& K_{r}(h_{(0)},h_{(1)})(\lap(h))+H_{r}(h_{(0)},h_{(1)},h_{(2)})(\Box(h))\\
&+\tilde{K}_{r}(h_{(0)},h_{(1)})(\tilde{\lap}(h))+\tilde{H}_{r}(h_{(0)},h_{(1)},h_{(2)})(\tilde{\Box}(h))\Big),
\end{aligned}
\end{equation*} 
where $\tilde{\lap}(h)=\sum_{r<i,j}\tilde{g}^{ij}\delta_i\delta_j(h)$ and $\tilde{\Box}(h)=\sum_{r<i,j} \tilde{g}^{ij}\delta_i(h)\delta_j(h)$ and  $\lap$, $\Box$, $K_r$ and $H_r$ are  given by Theorem \ref{scalarcrvatureconformal}. 
The functions $\tilde{K}_{r}$ and $\tilde{H}_{r}$ for $r\neq 2, 4$ are  given by
$$\tilde{K}_{r}(t_0,t_1)=\tilde{K}_{r}^t(f(t_0),f(t_1))\big[t_0,t_1;f\big],$$
$$\tilde{H}_{r}(t_0,t_1,t_2)=\tilde{H}_{r}^t(f(t_0),f(t_1),f(t_2))\big[t_0,t_1;f\big]\big[t_1,t_2;f\big]+2\tilde{K}_{r}^t(f(t_0),f(t_2))\big[t_0,t_1,t_2;f\big].$$
The functions $\tilde{K}_{r}^t$ and $\tilde{H}_{r}^t$ are 
\begin{align*}
\tilde{K}_{r}^t(x,y)=\frac{(2r-4)(x^2-y^2) x^{\frac{r}2} y^{\frac{r}2}+4 x^2 y^r-4 x^r y^2 }{(r-4) (r-2)x^{\frac34 r} y^{\frac34 r} (x-y)^3},
\end{align*}
and
\begin{align*}
&\tilde{H}_{r}^t(x,y,z)=\frac{2 x^{-\frac{3 r}{4}} y^{-r} z^{-\frac{3 r}{4}}}{(r-4) (r-2) (x-y)^2 (x-z)^3 (y-z)^2}\times\\
&\Big(x^r y^r z^2 (x-y) \left(x^2+2 x (y-2 z)-4 y^2+6 y z-z^2\right)\\
&\,\, +x^r y^{\frac{r}2} z^{\frac{r}2} (x-z)^2 (y-z) \big((r-3) y z-x ((r-3) z+y)\big)\\
&\,\,- x^r y^2 z^r (x-z)^3+ x^{\frac{r}2} y^{\frac{3 r}{2}}z^2 (x-y) (x-z)^2\\
&\,\,-x^{\frac{r}2} y^r z^{\frac{r}2} (x\hspace{-0.55mm}-\hspace{-0.55mm}y) (x\hspace{-0.55mm}-\hspace{-0.55mm}z) (y\hspace{-0.55mm}-\hspace{-0.55mm}z) \big(\hspace{-0.75mm}(r\hspace{-0.55mm}-\hspace{-0.55mm}3) x^2\hspace{-0.55mm}-2 x ((r\hspace{-0.55mm}-\hspace{-0.55mm}2) y+(1\hspace{-0.55mm}-r)z)+z ((r\hspace{-0.55mm}-3) z\hspace{-0.55mm}-2 (r\hspace{-0.55mm}-2) y)\hspace{-0.75mm}\big)\\
&\,\,+x^{\frac{r}2} y^{\frac{r}2} z^r (y-x) (x-z)^2 (y z-(r-3) x (y-z))\\
&\,\,+x^2 y^{\frac{3 r}{2}} z^{\frac{r}2} (x-z)^2 (y-z)
-x^2 y^r z^r (y-z) \left(x^2-6 x y+4 x z+4 y^2-2 y z-z^2\right)\Big).
\end{align*}
\end{theorem}

When the selfadjoint element $h\in \A^d$ has the property that  $\delta_j(h)=0$ for $1\leq j\leq r$, 
we call the twisted product functional metric \eqref{twistedmetric} a {\it warped functional metric} with the warping element $1/f(h)$. 
\begin{corollary}
The scalar curvature density of a warped product of $\tilde{g}$ and $g$ with the warping element $1/f(h)$  is given by
\begin{equation*}
R=\sqrt{|g||\tilde{g}|} \Big(\tilde{K}_{r}(h_{(0)},h_{(1)})(\tilde{\lap}(h))+\tilde{H}_{r}(h_{(0)},h_{(1)},h_{(2)})(\tilde{\Box}(h))\Big).
\end{equation*} 
\end{corollary}
\begin{proof}
It is enough to see that  $\lap(h)$ and $\Box(h)$ vanish for the warped metric.
\end{proof}

For $r=2$ and $r=4$, functions $\tilde{H}_r$ and $\tilde{K}_r$ are the limit of the functions given in Theorem \ref{scalartwisted} as $r$ approache $2$ or $4$. 
This is because of the fact that for these values of $r$, some of $T_{\vec{\alpha}}^k$ functions are the  limit case of formulas found earlier. 
For $r=2$ we have
\begin{equation*}
\tilde{K}_2(x,y)=\frac{-x^2+y^2+2 x y \log (\frac{x}{y})}{\sqrt{xy} (x-y)^3},
\end{equation*}
and  
\begin{equation*}
\begin{aligned}
\tilde{H}_2(x,y,z)=&\frac{1}{2 y\sqrt{xz}  (x-y)^2 (x-z)^3 (y-z)^2}\Big(-y (x+y) (x-z)^3 (y+z) \log (y)\\
&-z (x-y)^2  \left(-3 x^2 y+x^2 z-8 x y^2+10 x y z-2 x z^2+y z^2+z^3\right)\log (z)\\
&+x  (y-z)^2 \left(x^3+x^2 y-2 x^2 z+10 x y z+x z^2-8 y^2 z-3 y z^2\right) \log (x)\\
&+ 2 y (y-x) (x-z) (x+z) (z-y) (x-2 y+z)\Big).
\end{aligned}
\end{equation*}
For $r=4$, we have
\begin{equation*}
\tilde{K}_4(x,y)=\frac{x^2-y^2-(x^2+y^2) \log (\frac{x}{y})}{x y (x-y)^3},
\end{equation*}  
and
\begin{align*}
&\tilde{H}_4(x,y,z)=\frac1{2 x (x - y)^2 y^2 (x - z)^3 (y - z)^2 z}\times\\
& \Big(\left(x^2+y^2\right) (x-z)^3 \left(y^2+z^2\right) \log (y)\\
&\quad\,\, +\log (x) (y-z)^2 \left(x^4 y+x^4 z-6 x^3 y^2-2 x^3 y z-2 x^3 z^2+3 x^2 y^3+x^2 y^2 z+x^2 y z^2\right.\\
&\quad\,\, \qquad \qquad\qquad\qquad\left.+x^2 z^3+2 x y^3 z -4 x y^2 z^2+3 y^3 z^2+y^2 z^3\right)\\
&\quad\,\, -\log (z) (x-y)^2\left(x^3 y^2+x^3 z^2+3 x^2 y^3-4 x^2 y^2 z+x^2 y z^2-2 x^2 z^3+2 x y^3 z+x y^2 z^2\right.\\
&\quad\,\, \qquad \qquad\qquad\qquad\left.-2 x y z^3+x z^4+3 y^3 z^2-6 y^2 z^3+y z^4\right)\\
&\quad\,\, -2 (x-y) (x-z) (y-z) \left(x^3 z+x^2 y^2-2 x^2 z^2-2 x y^3+2 x y^2 z+x z^3-2 y^3 z+y^2 z^2\right)\Big).
\end{align*}

In \cite[section 4.1]{arXiv1808.02977}, the scalar curvature density of twisted product functional metric on noncommutative three torus for $f(t)=e^{2t}$ and $r=2$ is found. 
This result can be recovered from our formulas given in Theorem \ref{scalartwisted} by setting $t_0=s_0$, $t_1=s_0+s_1/2$ and $t_2=s_0+s_1/2+s_2/2$.

\subsection{Dimension two and Gauss-Bonnet theorem}\label{GBdim2sec}
The following result which is proved in  \cite{arXiv1811.04004} shows that the total scalar curvature of a noncommutative two torus equipped with a functional metric $g$ is independent of $g$. This result extends the Gauss-Bonnet theorem of  \cite{MR2907006, MR2956317} earlier proved for  conformally flat metrics. 
This is done by a careful  study of the functions $F_S^{ij}$ in dimension two, where it is 
 shown  that these functions vanish  for the noncommutative two torus equipped with a functional metric $g$.
This means that the total scalar curvature of  $(\nctorus,g)$ is independent of $g$. 
Similar to the case of conformally flat metrics, we call this result the Gauss-Bonnet theorem for functional metrics.
\begin{theorem}[Gauss-Bonnet Theorem  \cite{arXiv1811.04004}]\label{GBindim2} 
The total scalar curvature $\varphi(R)$ of the noncommutative two tori equipped with a functional  metric  vanishes, hence it is independent of the metric.
\end{theorem}

Let us  summarize the results obtained in  \cite{arXiv1811.04004} where   a new family of metrics, called functional metrics, on noncommutative tori is introduced and  their spectral geometry is studied. 
A class of  Laplace type operators for these metrics is introduced and their   spectral invariants are obtained from the heat trace asymptotics. 
A formula  for the second density of  the heat trace is also obtained.
In particular, the scalar curvature density and the total scalar curvature of  functional metrics are explicitly computed    in all dimensions for   certain classes of metrics including  conformally flat metrics and twisted product of flat metrics.
Finally a Gauss-Bonnet type theorem for a noncommutative two torus equipped  with a general functional metric is proved.

\section{{\bf Matrix Gauss-Bonnet}}

As we emphasized in the previous section, it is quite important to go beyond  conformally flat metrics, go beyond noncommutative tori,  and  beyond dimension four. For example, one naturally needs  to consider  noncommutative algebras that would represent higher genus noncommutative curves and other noncommutative manifolds. As far as  noncommutative  higher genus curves go, there is as yet no satisfactory theory, even at a topological level, and much less at a metric or spectral level.  This is a largely uninvestigated area and we expect new methods and ideas will be needed to make further progress with these objects.  

A reasonable class of noncommutative algebras are algebras of matrix valued functions on a smooth manifold. Now topologically they are Morita equivalent to commutative algebras and not so interesting, but their spectral geometry poses interesting questions. A first step was taken in \cite{MR3815127} to address this question. In this paper  a new class of noncommutative algebras  that are  amenable   to spectral   analysis,  namely algebras of matrix valued  functions on a Riemann surface of arbitrary genus, are studied.  The Dirac operator is 
conformally rescaled by a  diagonalizable matrix and  a Gauss-Bonnet theorem is proved for them. This is the matrix Gauss-Bonnet  in the title of this section. When the surface has genus one, scalar curvature is explicitly computed. We shall  briefly sketch these results in this section.

Let $M$ be a two dimensional closed spin Riemannian manifold and consider the algebra  of smooth matrix valued functions on $M$:
$$\mathcal{A} = C^\infty(M,  M_n(\mathbb{C})).$$
The Dirac operator of $M$,  $D: L^2(S) \to L^2(S)$    acts  on the Hilbert space  
of spinors.   The  algebra $\mathcal{A}$ 
acts diagonally on the Hilbert space $\mathcal{H} = L^2(S) \otimes \mathbb{C}^n$ and we have an spectral triple. 

  Let $h \in \mathcal{A}$ be a positive
element. We use $h$ to perturb the spectral triple of $\mathcal{A}$
in the following way. Consider the operator $D_h = h D h$ as 
a conformally rescaled Dirac operator. Now  $D_h$ does not
have bounded commutators with the elements of $\mathcal{A}$, but we still have a twisted spectral triple. This is similar to the situation with curved noncommutative tori.   The question is if    the Gauss-Bonnet theorem holds  for $D_h$. One is also interested in knowing if the scalar curvature can be computed explicitly.  The answer is positive as we sketch now.

To simplify the matters a bit,   it is  assumed  that the Weyl  conformal factor $h$ is diagonalizable, that is  $ h = U H U^*, $ where $U$ is unitary and $H$ is diagonal. 
Then we have 
$$ h D h = U H U^* D U H U^* = U \left( H \left( D + U^* [D,U] \right) H \right) U^*,$$
which shows that  the spectrum of $D_h$  and $D_{A,H} = H (D + A) H $, are equal. Here  
 $A = U^* [D,U]$ is a matrix valued one-form on $M$ and $D+A$ represents a fluctuation of the geometry 
 represented by $D$.  It is shown in \cite{MR3815127} 
 that the Gauss-Bonnet theorem  
holds for the family of conformally rescaled Dirac operators with 
possible fluctuations $D_{A,H} = H (D + A) H $  as above. Local expressions for the 
scalar curvature are computed  as well. The results demonstrate that unlike the case of 
higher residues in \cite{MR3623648}, the expressions for the value of the 
$\zeta$ function at $0$ are complicated also in the matrix 
case.

Let us consider first the canonical spectral triple for a flat torus 
$M= \mathbb{R}^2/\mathbb{Z}^2.$  Its spin structure is defined
by the Pauli spin matrices $\sigma^1, \sigma^2$ and its  Dirac operator is 
$$ D = \sigma^1 \delta_1 + \sigma^2 \delta_2. $$
Here $\delta_1, \delta_2$ are the partial derivatives $\frac{1}{i}\frac{\partial}{\partial x}$  and  $\frac{1}{i}\frac{\partial}{\partial y}.$ To compute the resolvent kernel we work in the algebra of matrix valued pseudodifferential operators obtained by tensoring the algebra  ${\bf \Psi}$   of pseudodifferential operators  on a smooth manifold $M$ by the algebra of $n$ by $n$ matrices. \\

{\bf The resolvent:}  The symbol of the Bochner Laplacian $D_{A,H}^2 = H (D+A) H^2 (D+A) H $ is given by  $ \sigma_{D_{A,H}^2} = a_2 + a_1 + a_0, $ where
\begin{equation*}
\begin{aligned}
a_2 =& H^4 \xi^2, \\
a_1 =& i \epsilon_{ij} \sigma^3 2 H^3 \delta_i(H) \xi^j 
      + 4 H^3 \delta_i(H) \xi^i 
      - i \epsilon_{ij} \sigma^3 H^3 A_i H \xi^j \\ 
     & + H^3 A_i H \xi^i 
      + i \epsilon_{ij} \sigma^3 H A_i H^3 \xi^j 
      + H A_i H^3 \xi^i, \\ 
a_0 =&   H^4 (\Delta H)
       + H^3 A_j \delta_i(H)   
       - H^3 i \sigma^3 \epsilon_{ij} A_i \delta_j(H)
       + H^3 \delta_i(A_i) H
       + i \sigma^3 H^3 \epsilon_{ij} \delta_j(A_i) H \\
     & + 2 H^2 \delta_i(H) \delta_i(H)  
       + 2 H^2 \delta_i(H) A_i H
       + 2 H A_i H^2 \delta_i(H) 
       + 2 i \sigma^3 H^2 \epsilon_{ij} \delta_i(H) A_j H \\
     & + i \sigma^3 \epsilon_{ij} H A_i H^2 \delta_i(H)  
       + i \sigma^3 \epsilon_{ij} H A_i H^2 A_j 
       + H A_i H^2 A_i H.         
\end{aligned}
\label{bsym}
\end{equation*} 
The first three terms of the symbols of $(D_{H,a})^{-2} = b_0 + b_1 +b_2 + \cdots$
are:

\begin{equation*}
\begin{aligned}
&b_0  = (a_2+1)^{-1},\\
&b_1  = -  \left( b_0 a_1 +  \partial_k(b_0) \delta_k(a_2) \right) b_0, \\
&b_2  = - \left(  b_1 a_1 + b_0 a_0 +  \partial_k(b_0) \delta_k(a_1) 
+  \partial_k(b_1) \delta_k(a_2) + \frac{1}{2} \partial_k \partial_j (b_0) \delta_k \delta_j (a_2) \right) b_0.
\end{aligned}
\end{equation*}

\subsection{Matrix curvature} Let us call 
a matrix-valued function $R: \mathbb{T}^2 \to M_n(\mathbb{C})$ the scalar curvature if  for any matrix valued function  $f \in \mathcal{A}$   we have:
$$ \zeta_{f,D}(0) = \int_{\mathbb{T}^2} \hbox{Tr} \, f R, $$
where the localized  spectral zeta function is defined by
$$ \zeta_{f,D} (s) = \hbox{Tr} \; f |D|^{-s}.$$
 It is  found that four terms contribute to the scalar curvature $R$ \cite{MR3815127}:\\

\noindent {\bf Terms not depending on $A$.}  They  depend only on $H$ and derivatives of $H$, and since  they
commute with each other, they can be  computed  as in the classical case.

\begin{equation}
\begin{aligned}
b_2(H, \xi) = & 96\, b_0^5 \delta_i(H)\delta_i(H) H^{14} (\xi^2)^3 
                        - 136\,  b_0^4 \delta_i(H)\delta_i(H) H^{10} (\xi^2)^2 \\
& + 46\, b_0^3 \delta_i(H)\delta_i(H) H^{6} (\xi^2)
- 2\, b_0^2 \delta_i(H)\delta_i(H) H^{2}  \\
& - 8 \, b_0^4 \Delta(H) H^{11} (\xi^2)^2
  + 8 \, b_0^3 \Delta(H) H^{7} (\xi^2) 
  - \, b_0^2 \Delta(H) H^3,
\end{aligned}
\end{equation} 

To integrate  over the $\xi$ space, we can use  the formula 
$$ \int_0^\infty \frac{r^{2k+1} dr}{  (1+a^2 r^2)^{2k+3} } = \frac{1}{2(k+1)a^{2(k+1)}}, $$ 
and obtain

\begin{equation}
R(H) = - \pi \left( \frac{1}{3} H^{-2} \delta_i(H) \delta_i(H) + \frac{1}{3} H^{-1} \Delta(H) \right).
\label{RH}
\end{equation}
To continue,  we use the rearrangement lemma of 
\cite{MR3626561}.  Let 
$$ \Delta(x) = H^{-4} x H^4.$$

\noindent {\bf Terms linear in $A$.}  We have:
$$
\begin{aligned}
b_2^{(1)}(H,A) = & 
    - b_0 h A_i  b_0 \delta_i(H)  H^2 
+ 5 b_0 h A_i b_0^2 \delta_i(H) H^6 \xi^2 
- 4 b_0 h A_i b_0^3 \delta_i(H) H^{10} (\xi^2)^2 \\
&   - b_0 H^3 A_i b_0 \delta_i(H)
+ 7 b_0 H^3 A_i b_0^2 \delta_i(H) H^4 \xi^2 
- 4 b_0 H^3 A_i b_0^3 \delta_i(H) H^8 (\xi^2)^2 \\
& + 3 b_0^2 H^5 A_i b_0 \delta_i(H)  H^2 \xi^2
 - 4 b_0^2 H^5 A_i b_0^2 \delta_i(H)  H^6 (\xi^2)^2 \\
& + b_0^2 H^7 A_i b_0 \delta_i(H) \xi^2 
- 4  b_0^2 H^7 A_i b_0^2 \delta_i(H) H^4 (\xi^2)^2 \\ 
\end{aligned}
$$
and
$$
\begin{aligned}
b_2^{(1)}(H,A) &= - 2 b_0 \delta_i(H)  H^2 A_i b_0 H 
    + 2 b_0^2 \delta_i(H)  H^4 A_i b_0 H^3 \xi^2  \\
& + 6 b_0^2 \delta_i(H)  H^6 A_i b_0 H \xi^2
   -  4 b_0^3 \delta_i(H)  H^8 A_i b_0 H^3 (\xi^2)^2 \\
&-  4 b_0^3 \delta_i(H)  H^{10} A_i b_0 H (\xi^2)^2.
\end{aligned}
$$
Explicit computations give
$$ R^{(1)}(H,A) =  \sum_{i=1,2} 2 \pi H G(\Delta)(A_i) \delta_i(H), $$
where $G$ is the following function:
$$ G(s) = 
{ (1+\sqrt{s})\sqrt{s} \over (s-1)^3 } 
\big( (s+1) \ln (s)  - 2 (s-1) \big), $$
and a second term,
$$ R^{(2)}(H,A) =  \sum_{i=1,2} -2 \pi H^{-2} \delta_i(H) G(\Delta)(A_i) H, $$
 with the same function $G(s)$.
After taking the trace these  terms cancel each other, and we get
$$ \hbox{Tr\ } \left( R^{(1)}(H,A)  + R^{(2)}(H,A)  \right) = 0. $$

\noindent {\bf Terms linear in $\delta_i(A_i)$.}  In this case we have:

$$
\begin{aligned}
b_2(H,\delta_i(A_j)) & =
   - b_0 H^3 \delta_i(A_i) b_0 H
   + b_0^2 H^5 \delta_i(A_i) b_0 H^3 \xi^2 \\
&+ b_0^2 H^7 \delta_i(A_i) b_0 H \xi^2,
\end{aligned}
$$
and  integrating out the   $\xi$ variables, we get 
$ \pi H^{-1} F(\Delta) (\delta_i(A_i)) H,$
where
$$F = - \frac{(1+\sqrt{s}) \sqrt{s}}{(s-1)^2} \ln(s) + \frac{\sqrt{s}+1}{s-1}. $$
Again, it is not difficult to check that $F(1)=0$ and the expression vanishes
after taking  the trace:
$$ \hbox{Tr\ } \left( R(H,\delta_i(A_j)) \right) = 0. $$

\noindent {\bf Quadratic terms in $A_i$.} We have:

$$ 
\begin{aligned}
b_2(H,A^2) &= - b_0 H A_i H^2 A_i b_0 H 
    + b_0 H A_i b_0 H^6 A_i b_0 H \xi^2 \\
& + b_0 H^3 A_i b_0 H^2 A_i b_0 H^3 \xi^2 .
\end{aligned}
$$
Integrating over $\xi$ we obtain:
$$
R(H,A^2) = - \pi H^{-1} Q(\Delta^{(1)},\Delta^{(2)})(A_i \cdot A_i) H
$$
where
$$
Q(s,t) = { \sqrt{s} (\sqrt{t} + s) \over (s-1) )(s-t) } \; \ln s 
- { \sqrt{s} \sqrt{s} \over (s-t) \sqrt{t} } \; \ln t .
$$
To compute the trace, let 
$$ F(s) = Q(s,1) = { (s+1) \ln s + 2 (1-s) \over  (s-1)^2 }, $$
and observe that due to the trace property:
$$ 
\begin{aligned}
\hbox{Tr} \left( H^{-1} F(\Delta)(A_i) A_i H \right) &= 
\hbox{Tr} \left( A_i  F(\Delta)(A_i)  \right) \\
& = \hbox{Tr} \left(  F(\Delta^{-1})(A_i) A_i  \right). 
 \end{aligned}
$$ 
Now since 
$$ 
\begin{aligned}
F(\frac{1}{s})
&= { - (\frac{1}{s}+1) \ln s + 2 (1-\frac{1}{s}) \over  (\frac{1}{s}-1)^2 } \\
&= { - (s+1) \ln s - 2 (1-s) \over (s-1)^2} \\
&= - F(s),
\end{aligned} 
$$
one gets  
$$  \hbox{Tr\ } R(H,A^2) = 0, $$
and so the quadratic term vanishes as well.

\subsection{The Gauss-Bonnet theorem}

The term which does not depend on $A$  is a total derivative term:

$$ \frac{1}{3} \delta_i \left( H^{-1} \delta_i(H) \right). $$
Using Stokes  theorem,  this term is seen to  vanish as well after integration.  Putting it all together, one thus obtains:

\begin{proposition}
For the matrix conformally rescaled Dirac operator on the 
two-dimensional torus,
$D_h = h D h$, where $h$ is a globally diagonalizable positive matrix,
the Gauss-Bonnet theorem holds:
$$ \zeta_{D_h}(0) = \zeta_{D}(0). $$
\end{proposition}

\subsection{Higher genus  matrix Gauss-Bonnet} Now we look at the general case where 
 $M$ is a closed Riemann surface   with a spin structure and a
 Dirac operator  $D$. 
Consider the operator
$$ D_{H,A} = H(D + A)H, $$
for $H$ a diagonal matrix valued function on $M$ and
$A$ a matrix-valued one-form,  identified here with its  Clifford
image. 

We must now compute  the value of $\zeta_{D_{H,A}^2}(0)$ 
using  methods of pseudodifferential calculus. Let us
denote the symbols of $D^2_H$ as:
$$ D_H^2 = (HDH)^2 = a_2^H + a_1^H + a_0^H, $$
and the symbol  of $D^2$ as 
$$ D^2 = a_2^o + a_1^o +a_0^o. $$
As  in the case of the torus,  the computation is divided 
into the cases of terms not depending on $A$, linear in $A$
and quadratic in $A$.\\

\noindent {\bf Terms independent of $A$. } Since  $H$ is globally diagonalizable,  we can  assume 
it is scalar. Thus we are reduced  to a conformal rescaling
of the classical Dirac operator. 
Since in this case the Gauss-Bonnet theorem holds, it remains
only to see that the contribution to the Gauss-Bonnet term from
the linear and quadratic terms in $A$ vanish.\\

\noindent {\bf Terms linear in $A$.}
Linear terms do arise in $b_2$ from the following terms:
$$ 
\begin{aligned}
&b_0 a_1^H b_0 a_1(A) b_0  
+ \partial_k^\xi(b_0) \partial^x_k(a_2^H) b_0 a_1(A) b_0
+ b_0 a_1(A) b_0 a_1^H b_0 \\
&- b_0 a_0(A) b_0 
- \partial_k^\xi(b_0) \partial^x_k(a_1(A)) b_0 
- \partial_k^\xi(b_0 a_1(A) b_0) \partial^x_k(a_2^H) b_0. 
\end{aligned}
$$  
where  $a_1(A), a_0(A)$  denote terms linear in $A$. 
Now, one  can use  normal coordinates at a given point $x$  of $ M$. The terms without derivatives reduce 
easily to the torus case. The only difficulty arises
from terms with derivatives in $x$, that is, $ \partial^x_k(a_2^H)$.
and $\partial^x_k(a_1(A))$. Since $a_2^H = H^4 g_{ij} \xi^i \xi^j$, and  in normal
coordinates the first derivatives of the metric vanish at the 
point $x$,  we see that the only remaining term would be with the 
derivative of $H^4$, and again this term would be reduced to
the term linear in $A$ from the torus case.

Similar argument works also for the other term, $a_1(A)$, which is
$$ \left( H^3 A_i H + H A_k H^3 \right) \sigma^k \sigma^i \xi_i, $$
and since in  $\partial_k^\xi(b_0) \partial^x_k(a_1(A)) b_0 $
there are no further $\sigma$ matrices,  one  can compute first the 
trace over the Clifford algebra and write  it  as
$$\frac{1}{2} \left( H^3 A_i H + H A_k H^3 \right) g^{ki} \xi_i.$$
Thus, in normal coordinates a around   $x$ the expression is  identical 
to the one for the flat torus. 
Therefore, the integration over $\xi$ would yield the same result, and 
the density of the linear $A$ contribution to the trace of $b_2$ vanishes 
at $x$. Consequently, the contribution to the Gauss-Bonnet term linear 
in $A$ also vanishes.\\

\noindent {\bf Quadratic terms.}
The  quadratic terms in $A$ are 
$$ b_0 a_1(A) b_0 a_1(A) b_0 - b_0 a_0(A^2) b_0.$$
It is easy to see that in normal coordinates we have:
\begin{eqnarray*}
  a_1(A) & = &\left( H^3 \sigma^j \xi_j (\sigma^i A_i) H 
             + \sigma^i H A_i H^3 \sigma^j \xi_j \right),\\
              a_0(A) & = &(\sigma^i H A_i H) (\sigma^k H A_k H). 
\end{eqnarray*}
Using normal coordinates, one  
reduces the $\xi$ -integral to 
the situation already considered for the torus. Hence the 
density of Gauss-Bonnet term  with quadratic contributions 
from $A$ identically vanishes as well. This finishes the proof of Gauss-Bonnet for higher genus matrix valued functions with a general Dirac operator with fluctuation. This result was obtained in  \cite{MR3815127}.

\section{{\bf Curvature of the determinant line bundle}}

 It  would be interesting to know how far  our hard analytic methods like pseudodifferential operators, spectral analysis and heat equation techniques, can be pushed in  the  noncommutative  realm, at least for noncommutative tori and  and toric manifolds. So far we have seen  that these analytic techniques, suitably modified and enhanced,  has been quite successful in dealing with scalar and Ricci curvature. Along this idea, in  \cite{MR3590385}
 the curvature of the determinant line bundle on a family of Dirac operators for a noncommutative two torus
 is computed.  Following Quillen's original construction for Riemann surfaces \cite{MR783704} and using zeta regularized determinant of Laplacians,  the determinant line bundle is endowed with a natural Hermitian metric. By defining  an analogue of Kontsevich-Vishik canonical trace, defined on Connes' algebra of classical pseudodifferential symbols for the noncommutative two torus, the curvature form of the determinant line bundle  is  computed  through the second variation $\delta_w \delta_{\bar{w}} \log \det(\Delta).$ Calculus of symbols and the canonical trace  were effectively used  to bypass    local calculations  involving  Green functions   in \cite{MR783704}
     which is  not applicable in the  noncommutative case.   
  In a sequel paper  \cite{arXiv1504.01174},
the spectral eta function for certain families of Dirac operators on noncommutative 3-torus is  studied  and its regularity at zero is proved. By using variational techniques,  it  is shown  that the eta function $\eta_D(0)$ is a conformal invariant. By studying the Laurent expansion at zero of $\text{TR}(|D|^{-z})$, the conformal invariance of $\zeta'_{|D|}(0)$ for noncommutative 3-torus is proved. Finally, for the coupled Dirac operator, a local formula for the variation $\partial_A \eta_{D+A} (0)$ is derived which is the analogue of the so called induced Chern-Simons term in quantum field theory literature.
 
 In this section we shall recall and comment on results obtained in \cite{MR3590385} on the curvature of the  determinant   line bundle on a noncommutative torus.

\subsection{The determinant line bundle}

Let $ \mathcal{F} = {\rm Fred} (\mathcal{H}_0, \mathcal{H}_1)$  denote the set of Fredholm operators between Hilbert spaces $\mathcal{H}_0$ and $ \mathcal{H}_1$. It is an open subset, in  norm topology,  in the complex Banach space of all bounded linear operators  between  $\mathcal{H}_0$ and $ \mathcal{H}_1$. The index map $ index : \mathcal{F}  \to \mathbb{Z}$ is a homotopy invariant and in fact defines a bijection between connected components of 
$\mathcal{F}$ and  the set of integers $\mathbb{Z}$. 
 It is well known that $\mathcal{F}$ is a classifying space for $K$-theory (Atiyah-Janich): for any compact space $X$ we have a natural ring  isomorphism 
$ K^0 (X) = [ X, \mathcal{F}]$
between the $K$-theory of $X$ and the set of homotopy classes of continuous maps from $X$ to 
$\mathcal{F}.$

In \cite{MR783704} Quillen defines a holomorphic line bundle $\text{DET} \to \mathcal{F}$ over the space of Fredholm operators such that for any $T \in  \mathcal{F}$ 
$${\rm DET}_T = \Lambda^{max}({\rm ker}(T))^* \otimes  \Lambda^{max}({\rm coker}(T)).$$
This is remarkable if we notice that ${\rm ker}(T)$ and ${\rm coker}(T)$ are not vector bundles due to  
discontinuities in their dimensions as $T$ varies within $\mathcal{F}$. 

 It is tempting to think that since $c_1 (\text{DET})$ is the generator of $H^2(\mathcal{F}_0, \mathbb{Z})
  \cong \mathbb{Z}$, $\mathcal{F}_0$  being the index zero operators,  there might exits  a natural Hermitian metric on DET  whose curvature 2-form would be a representative of this   generator. One problem is that the induced metric from   ${\rm ker}(T)$ and ${\rm ker}(T^*)$  on  DET is not even continuous.  
 In   \cite{MR783704} Quillen shows that  for  families  of Cauchy-Riemann operators  on a Riemann surface one can correct the Hilbert space metric by  multiplying it by   zeta regularized determinant and in this way one obtains a smooth Hermitian metric on the induced determinant line bundle. In  \cite{MR3590385} a   similar construction  for the noncommutative two torus is given as we explain later in this section.

\subsection{The canonical trace and noncommutative residue} To carry the calculations, 
 an analogue of the canonical trace of   \cite{MR1373003} for the  noncommutative torus is constructed in \cite{MR3590385}.  First we need to extend  our original algebra of   pseudodifferential operatprs to {\it classical} pseudodifferential operators.

A  smooth map $\sigma:\mathbb{R}^2\to \mathcal{A}_{\theta}$ is called a classical symbol of order $\alpha \in \mathbb{C}$  
if for any $N$ and each $0\leq j\leq N$ there exist $\sigma_{\alpha-j}:\mathbb{R}^2\backslash\{0\}\to\mathcal{A}_{\theta}$ positive homogeneous of degree $\alpha-j$,  and a symbol $\sigma^N\in\mathcal{S}^{\Re(\alpha)-N-1}(\mathcal{A}_{\theta})$, such that
 \begin{equation}
 \label{sigma}
\sigma (\xi)=\sum_{j=0}^{N}\chi(\xi)\sigma_{\alpha-j}(\xi)+\sigma^N(\xi)\quad\xi\in\mathbb{R}^2.
 \end{equation}
 Here  $\chi$ is a smooth cut off function on $\mathbb{R}^2$ which is equal to zero on a small ball around the origin,  and is equal to one outside the unit ball.
  It can be shown that the homogeneous terms in the expansion are uniquely determined by $\sigma$. 
 We denote the set of classical symbols of order $\alpha$ by $\mathcal{S}^{\alpha}_{cl}(\mathcal{A}_{\theta})$ and the associated classical pseudodifferential operators by $\Psi_{cl}^{\alpha}(\mathcal{A}_{\theta})$.

 The space of classical symbols $\mathcal{S}_{cl}(\mathcal{A}_{\theta})$ is equipped with a Fr\'{e}chet topology induced by the semi-norms
\begin{equation}\label{frechet}
p_{\alpha,\beta}(\sigma)=\sup_{\xi\in\mathbb{R}^2}(1+|\xi|)^{-m+|\beta|}||\delta^{\alpha}\partial^{\beta}\sigma(\xi)||.
\end{equation}

The analogue of the Wodzicki residue for classical pseudodifferential operators on the noncommutative torus is defined in \cite{MR2831659}.
 \begin{definition}
 The Wodzicki residue of a classical pseudodifferential operator $P_{\sigma}$ is defined as
 $$\Res(P_{\sigma})=\varphi_0\left(\res (P_\sigma)\right),$$
where $\res(P_\sigma):=\int_{|\xi|=1}\sigma_{-2}(\xi)d\xi$. 
 \end{definition}
It is evident from its  definition that Wodzicki residue vanishes on differential operators and on non-integer order classical pseudodifferential operators.
 
 To  define the analogue of  the canonical trace   on non-integer order pseudodifferential operators on the noncommutative torus,   one needs   the existence of the so called cut-off integral for classical symbols.
 
\begin{proposition}\label{expansioninR}
 Let $\sigma\in\mathcal{S}_{cl}^{\alpha}(\mathcal{A}_{\theta})$ and $B(R)$ be a disk of radius $R$ around the origin. One has the following asymptotic expansion as $R\rightarrow\infty$
 $$\int_{B(R)}\sigma(\xi)d\xi \,\sim \sum_{j=0,\alpha-j+2\neq0}^{\infty}\alpha_j(\sigma)R^{\alpha-j+2}+\beta(\sigma)\log R+c(\sigma),$$
 where $\beta(\sigma)=\int_{|\xi|=1}\sigma_{-2}(\xi)d\xi$
 and the constant term in the expansion, $c(\sigma)$, is given by 
\begin{equation}\label{cutoffinexpansion}
\int_{\mathbb{R}^n}\sigma^N+\sum_{j=0}^{N}\int_{B(1)}\chi(\xi)\sigma_{\alpha-j}(\xi)d\xi-\sum_{j=0,\alpha-j+2\neq0}^N \frac{1}{\alpha-j+2}\int_{|\xi|=1}\sigma_{\alpha-j}(\omega)d\omega.
\end{equation}
  \end{proposition}
 
\begin{definition} 
 The cut-off integral of a symbol $\sigma\in\mathcal{S}_{cl}^{\alpha}(\mathcal{A}_{\theta})$ is defined to be the constant term in the above asymptotic expansion, and we denote it by $\Int \sigma(\xi)d\xi$.
\end{definition}
 The cut-off integral of a symbol is independent of  the choice of $N$. It is also independent of the choice of the cut-off function 
  $\chi$.

 \begin{definition}
 The canonical trace of a classical pseudodifferential operator $P\in\Psi^{\alpha}_{cl}(\mathcal{A}_{\theta})$ of non-integral order $\alpha$ is defined as 
 \begin{equation*}
 \TR(P):=\varphi_0\left(\Int \sigma_P(\xi)d\xi\right).
 \end{equation*}
 \end{definition}
Note that any pseudodifferential operator $P$  of order less that  $-2$, is a trace-class operator on $\mathcal{H}_0$ and its trace is given by
 $$\Tr(P)=\varphi_0\left(\int_{\mathbb{R}^2}\sigma_{P}(\xi)d\xi\right).$$
 On the other hand, for such operators the symbol is integrable and we have
\begin{equation}\label{TRTr}
\Int\sigma_P(\xi)d\xi=\int_{\mathbb{R}^2}\sigma_P(\xi)d\xi.
\end{equation}
 Therefore, the $\TR$-functional and operator trace coincide on classical pseudodifferential operators of order less than $-2$. 
 
 The canonical trace  $\TR$  is an analytic continuation of the operator trace and using this fact  one can prove that it is actually a trace.

 \begin{proposition}\label{laurentofholo}
 Given a holomorphic family $\sigma(z)\in\mathcal{S}_{cl}^{\alpha(z)}(\mathcal{A}_{\theta})$, $z\in W \subset\mathbb{C}$, the map 
 $$z\mapsto \Int\sigma(z)(\xi)d\xi,$$
 is meromorphic with at most simple poles.
 Its  residues   are given by
 $$\Res_{z=z_0} \Int\sigma(z)(\xi)d\xi=-\frac{1}{\alpha'(z_0)}\int_{|\xi|=1}\sigma(z_0)_{-2}d\xi.$$
 \end{proposition}

Using the above result one can show that 
 if  $A\in\Psi ^{\alpha}_{cl}(\mathcal{A}_{\theta})$  is  of order $\alpha\in\mathbb{Z}$ and  $Q$  is   a positive elliptic classical pseudodifferential operator of positive order $q$, then 
 $$\Res_{z=0}\TR(AQ^{-z})= \frac{1}{q}\Res(A).$$
 Using this and the uniqueness of analytic continuation one can prove the trace property of $\TR$. That is 
 $\TR(AB)=\TR(BA)$ for any $A,B\in\Psi_{cl}(\mathcal{A}_{\theta})$, provided that $\ord(A)+\ord(B)\notin\mathbb{Z}$.

 \subsection{Log-polyhomogeneous symbols}  
  In general, $z$-derivatives of a classical holomorphic family of symbols is not classical anymore and therefore one needs to  introduce log-polyhomogeneous symbols which include the $z$-derivatives of the symbols of the holomorphic family $\sigma(AQ^{-z})$. 

\begin{definition}
A symbol $\sigma$ is called a log-polyhomogeneous symbol if it has the following form
\begin{equation}\label{logpolysym}
\sigma(\xi) \sim \sum_{j\geq 0}\sum_{l=0}^\infty\sigma_{\alpha-j,l}(\xi)\log^l|\xi|\quad |\xi|>0,
\end{equation}
with $ \sigma_{\alpha-j,l}$ positively homogeneous in $\xi$ of degree $\alpha - j$.
\end{definition}

 A  prototypical  example  of an operator with such a symbol is $\log Q$ where $Q\in\Psi^q_{cl}(\mathcal{A}_{\theta})$ is a positive elliptic pseudodifferential operator of order $q>0$.
 The logarithm of $Q$ can be  defined by
 $$\log Q
 =Q\left.\frac{d}{dz}\right|_{z=0}Q^{z-1}
 =Q\left.\frac{d}{dz}\right|_{z=0}\frac{i}{2\pi}\int_C\lambda^{z-1}(Q-\lambda)^{-1}d\lambda.$$

  For an operator  $A$ with log-polyhomogeneous symbol as \eqref{logpolysym} we define
 $$\res(A)=\int_{|\xi|=1}\sigma_{-2,0}(\xi)d\xi.$$

The following result can be proved along the lines of its classical counterpart in  \cite{MR2322493}.
 
 \begin{proposition}\label{Laurentat0}
Let $A\in\Psi_{cl}^\alpha(\mathcal{A}_{\theta})$ and $Q$ be a positive , in general an admissible, elliptic pseudodifferential operator of positive order $q$. If $\alpha\in P$ then $0$ is a possible simple pole for the function $z\mapsto \TR(AQ^{-z})$ with the following Laurent expansion around zero. 
\begin{align*}
\TR(AQ^{-z})&=\frac{1}{q}\Res(A)\frac{1}{z}\\
 &+\varphi_0\left(\Int\sigma(A)(\xi) d\xi- \frac{1}{q}\res(A\log Q)\right)-\Tr(A\Pi_Q)\\
& +\sum_{k=1}^K(-1)^k\frac{(z)^k}{k!} \\
&\times \left(\varphi_0\left( \Int\sigma(A(\log Q)^k)(\xi)d\xi-\frac{1}{q(k+1)}\res(A(\log Q)^{k+1})\right)-\Tr(A(\log Q)^k \Pi_Q)\right)\\
&+o(z^{K}).
\end{align*}
Where $\Pi_Q$ is the projection on the kernel of $Q$.
\end{proposition} 

For  operators $A$ and $Q$ as  in the previous Proposition, 
the {\it generalized zeta function}  is defined by 
\begin{equation}\label{zetafunction}
\zeta(A,Q,z)=\TR(AQ^{-z}).
\end{equation}
From Proposition \ref{laurentofholo}, it follows that $\zeta(A,Q,z)$ is a  meromorphic function with simple poles.
Moreover,  $\zeta(A,Q,z)$ is the analytic continuation of the spectral zeta function $\Tr(AQ^{-z})$. 
If $A$ is a differential operator,  the   zeta function \eqref{zetafunction} is regular at $z=0$ with a  value 
$$\varphi_0\left(\Int\sigma(A)(\xi)d\xi- \frac{1}{q}\res(A\log Q)\right)-\Tr(A\Pi_Q).$$

\subsection{Cauchy-Riemann operators on noncommutative tori}

As we did  before,     we can fix a complex structure on $\A$ by a complex number $\tau$ in the upper half plane.  Consider  the spectral triple
\begin{equation}\label{firstspec}
(\A, \mathcal{H}_{0}\oplus\mathcal{H}^{0,1},D_0=\left(\begin{array}{ll} 0 & \bar{\partial}^*\\ \bar{\partial}&0\end{array}\right)),
\end{equation}
where $\bar{\partial}:\A\to \A$ is given by $\bar\partial=\delta_1+\tau\delta_2$.  The Hilbert space $\mathcal{H}_0$ is  defined  by the GNS construction from $\A$ using the trace $\varphi_0$ and $\bar{\partial}^*$ is the adjoint of the operator $\bar{\partial}$.

As in the classical case, the  Cauchy-Riemann operator on $\A$   is the positive part of the twisted Dirac operator. All such operators  define  spectral triples of the  form
$$(\A, \mathcal{H}_{0}\oplus\mathcal{H}^{0,1},D_A=\left(\begin{array}{ll} 0 & \bar{\partial}^*+\alpha^*\\ \bar{\partial}+\alpha&0\end{array}\right)),$$
where $\alpha\in \A$ is the positive part of a selfadjoint element 
$$A=\left(\begin{array}{ll} 0 & \alpha^*\\ \alpha&0\end{array}\right)\in\Omega^1_{D_0}(\A).$$
We recall that $\Omega^1_{D_0}(\A)$ is the space of quantized one forms 
 consisting of the elements $\sum a_i[D_0, b_i]$ where $a_i,b_i\in\A$ \cite{MR1303779}.
Note that the in this case, the space  $\mathcal{A}$ of Cauchy-Riemann operators is the space of $(0,1)$-forms on $\A$.

 \subsection{The curvature of the determinant line bundle for $\A$}\label{sec:computation}
For any $\alpha\in\mathcal{A}$,  the Cauchy-Riemann operator $$\bar{\partial}_{\alpha}=\bar{\partial}+\alpha:\mathcal{H}_0\to\mathcal{H}^{0,1}$$ is a Fredholm operator.
We  pull back the determinant line bundle DET on  the space of Fredholm operators  ${\rm Fred}(\mathcal{H}_0,\mathcal{H}^{0,1}),$ to get a line bundle $\mathcal{L}$ on $\mathcal{A}$. Following Quillen \cite{MR783704},  one can  define a Hermitian metric on $\mathcal{L}$ and the problem is to compute its curvature.  Let us define a metric on  the fiber 
$$\mathcal{L}_{\alpha} = \Lambda^{max} ({\rm ker} \,\bar{\partial}_{\alpha})^*\otimes \Lambda^{max}({\rm ker} \,\bar{\partial}_{\alpha}^*)$$ 
as the product of the induced metrics on $\Lambda^{max} ({\rm ker} \,\bar{\partial}_{\alpha})^* $ and  $\Lambda^{max}({\rm ker} \,\bar{\partial}_{\alpha}^*)$, 
 with the zeta regularized determinant $e^{-\zeta'_{\Delta_{\alpha}}(0)}$.
Here we define   the Laplacian  as $\Delta_{\alpha}=\bar{\partial}_{\alpha}^*\bar{\partial}_{\alpha}:\mathcal{H}_0\to \mathcal{H}_0$,  and its zeta function by
\begin{align*}
&\zeta(z)=\TR(\Delta_{\alpha}^{-z}).
\end{align*}
It is a meromorphic function and  is regular at $z=0$ .
Similar proof as in \cite{MR783704} shows that this defines a smooth Hermitian metric on the determinant line bundle $\mathcal{L}$. 

On the open set of invertible  operators  each fiber of $\mathcal{L}$  is canonically isomorphic to $\mathbb{C}$ and the  nonzero holomorphic section $\sigma=1$  gives a trivialization. Also, according to the definition of the Hermitian metric, the norm of this section is given by 
\begin{equation}
\|\sigma\|^2=e^{-\zeta'_{\Delta_\alpha}(0)}.
\end{equation} 

\subsection{Variations of LogDet and curvature form}
 A  holomorphic line bundle equipped with a  Hermitian  inner product has a canonical  connection compatible with the two structures. This is also known as the Chern connection. The curvature form of this connection is given  by  $\bar{\partial} \partial \log\|\sigma\|^2,$ 
where $\sigma$  is any non-zero local holomorphic section.

In the case at hand,   the second variation 
$ \bar{\partial} \partial \log \|\sigma\|^2$
on the open set of invertible Cauchy-Riemann operators must be computed. Let us  consider a holomorphic family of invertible  Cauchy-Riemann operators $D_w=\bar{\partial}+\alpha_w$, where $\alpha_w$ depends holomorphically on the complex variable $w$. The second variation of logdet, that is   
$ \delta_{\bar{w}} \delta_w \zeta'_\Delta(0)$, is computed in  \cite{MR3590385} as we recall now.

 \begin{lemma}\label{secondvarnotsim} 
 For the holomorphic family of Cauchy-Riemann operators $D_w$, the second variation of $\zeta'(0)$ is given by
 $$\delta_{\bar w}\delta_{w}\zeta'(0)=\frac{1}{2}\varphi_0\left(\delta_w D\delta_{\bar{w}}{\rm res}(\log\Delta\,D^{-1})\right).$$
 \qed
 \end{lemma}
The final step is  to compute $\delta_{\bar{w}}{\res}(\log\Delta\,D^{-1})$. 
This combined with the above lemma will show that  the curvature  form of the determinant line bundle equals the K\"ahler form on the space of connections. We refer 
the reader to \cite{MR3590385} for the proof which is long and technical. We emphasize that the original Quillen proof, based on Green function calculations, cannot be be extended to the noncommutative case. 
 \begin{lemma}\label{secondvarsim} 
 With above definitions and notations,  we have
\begin{align*}
 \sigma_{-2,0}(\log\Delta\, D^{-1})&=\frac{(\alpha+\alpha^*)\xi_1+(\bar\tau\alpha+\tau\alpha^*)\xi_2}{(\xi_1^2+2\Re(\tau) \xi_1\xi_2+ |\tau|^2\xi_2^2)(\xi_1+\tau\xi_2) }\\ \\
 &-\log \left( \frac{\xi_1^2+2\Re(\tau) \xi_1\xi_2+ |\tau|^2\xi_2^2}{|\xi|^2}\right) \frac{\alpha}{\xi_1+\tau\xi_2},
 \end{align*}
 and
 $$\delta_{\bar{w}}{\rm res}(\log(\Delta) D^{-1})=\frac{1}{2\pi\Im(\tau)}(\delta_{{w}}D)^*.$$
  \end{lemma}
  
Now we can state the main result of \cite{MR3590385} which  computes the curvature of the determinant line bundle in terms of the natural K\"{a}hler form on the space of connections.
\begin{theorem}
The curvature of the determinant line bundle for the  noncommutative two torus is given by
\begin{equation}\label{formulaofmainthm}
\delta_{\bar w}\delta_{w}\zeta'(0)=\frac{1}{4\pi\Im(\tau)}\varphi_0\left(\delta_w D(\delta_w D)^*\right).
\end{equation}
\end{theorem}
In order to recover the classical result of Quillen  in the classical limit of  $\theta=0$, one  has  to notice that   the volume form has  changed due to a change  of the metric. This means we just need  to  multiply the  above  result by $\Im(\tau)$.

\section{{\bf Open problems}}
In this final section we formulate  some of the open problems that we think are  worthy of study  for further understanding of local invariants of noncommutative manifolds.

\begin{enumerate}

\item Beyond dimension two and beyond conformally flat. The class of conformally flat metrics in dimensions bigger than two cover only a small part of all possible metrics. It would be very important to formulate large classes of metrics that are not conformally flat, but at the same time lend themselves to spectral analysis and to heat asymptotics techniques. It is also very important to have curvature formulas that work uniformly in all dimensions. The largest such class so far is the class of the so called functional metrics introduced in \cite{arXiv1811.04004} and surveyed in Section 9 of this paper. It is an interesting problem to further enlarge this class.

\item To extend the definition of curvature invariants to noncommutative spaces with non-itegral dimension, including zero dimensional spaces.
This would require rethinking the heat trace asymptotic expansion, and the nature of its leading and sub leading terms. In particular since quantum spheres are zero dimensional, its spectrum is of exponential growth  and does not satisfy the usual Weyl's asymptotic law. A first step would be to see how to formulate a  Gauss-Bonnet typle theorem  for quantum spheres.

\item Weyl tensor and full curvature tensor. It is not clear that the classical differential geometry would,  or should, give us a blueprint in the noncommutative case.  One should be prepared for new phenomena. Having that in mind, one should still look for analogues of  Weyl  and  full Riemann curvature tensors. The problem is that the components of these tensors are quite entangled in the heat trace expansion and separating and identifying their different components seem to be a hard task, if not impossible. One needs new ideas to make progress here.

\item Gauss-Bonnet terms in higher dimensions. The Gauss Bonnet density in 2 dimension is particularly simple and is in fact equal to the scalar curvature multiplied by the volume form. In dimensions 4 and above this term is classically more complicated, being the Pfaffian of the curvature tensor. In dimension four it is a linear combination of norms of the Riemann tensor, the Ricci tensor and the Ricci scalar. It is not clear how this can be expressed in terms of the heat kernel coefficients.

\item Higher genus noncommutative Riemann surfaces. It is highly desirable to define noncommutative Riemann surfaces of higher genus quipped with a spectral triple and check the Gauss -Bonnet theorem for them. This would greatly extend our understanding of local invariants of noncommutative spaces.

\item Noncommutative uniformization theorem. The study of curved noncommutative 2-tori suggests a natural problem in noncommutative geometry. At least for the class of noncommutative 2-tori it is desirable to know to what extent the uniformization theorem holds, or what form and shape it would take.

\item Analytic versus algebraic curvature. In classical differential geometry, as we saw in this paper, there are algebraic as well as analytic techniques (based on the heat equation)   to define the scalar and Ricci curvature. The two approaches give the same  results. This is not so in the noncommutative case. For noncommutative tori, when the deformation parameter satisfies some diophantine condition, Rosenberg in \cite{Ros} proved a Levi-Civita type theorem and hence gets an algebraic definition of curvature. The resulting formula is very different from the formula of Connes-Moscovici-Fathizadeh-Khalkhali \cite{MR3194491, MR3359018} surveyed  in   this paper. It is important to see if there is any relation at all between these formulas and what this means for the study of curved noncommutative tori.

\end{enumerate}

 \section*{{\bf Acknowledgements}} 
F.F.  acknowledges support from the 
Marie Curie/SER Cymru II Cofund 
Research Fellowship  663830-SU-008. M.K. would like to thank Azadeh Erfanian for providing the original picture used in page four, Saman Khalkhali for his support and care,  and Asghar Ghorbanpour for many discussions on questions related to this article. \\

\bibliographystyle{abbrv}




$ $\\
Department of Mathematics,  Computational Foundry, Swansea University, \\
Swansea University Bay Campus, Swansea, SA1 8EN, United Kingdom; \\
Max Planck Institute for Biological Cybernetics, \\
 Max-Planck-Ring 8-14, 72076 T\"ubingen, Germany\\ 
{\it E-mail}: farzad.fathizadeh@swansea.ac.uk \\

$ $\\
 Department of Mathematics, 
University of Western Ontario\\
 London, Ontario, Canada, N6A 5B7\\
 {\it E-mail}: masoud@uwo.ca

\end{document}